\documentclass[10pt]{article}
\usepackage{color}
\usepackage{amssymb}
\usepackage{amsthm,array,amsfonts}
\usepackage{amsmath}
\usepackage{verbatim}
\usepackage{tikz-cd}
\usepackage{enumitem}
\setitemize{itemsep=0pt}
\setenumerate{itemsep=0pt}

\usepackage{hyperref}
\hypersetup{
    colorlinks,
    citecolor=blue,
    filecolor=black,
    linkcolor=black,
}

\newtheorem{theo}{Theorem}[section]
\newtheorem{thm}{Theorem}[section]
\newtheorem{prop}[theo]{Proposition}

\newtheorem{lemma}[theo]{Lemma}

\newtheorem{cor}[theo]{Corollary}
\newtheorem{question}[theo]{Question}

\newtheorem{defn}[theo]{Definition}

\newtheorem{conj}[theo]{Conjecture}

\theoremstyle{remark}

\newtheorem{rmk}[theo]{Remark}
\newtheorem{example}[theo]{Example}

\newcommand{\BA}{{\mathbb{A}}}

\newcommand{\BC}{{\mathbb{C}}}

\newcommand{\BE}{{\mathbb{E}}}
\newcommand{\BF}{{\mathbb{F}}}
\newcommand{\BG}{{\mathbb{G}}}

\newcommand{\BL}{{\mathbb{L}}}

\newcommand{\BQ}{{\mathbb{Q}}}
\newcommand{\BR}{{\mathbb{R}}}

\newcommand{\BZ}{{\mathbb{Z}}}

\newcommand{\CA}{{\mathcal A}}
\newcommand{\CB}{{\mathcal B}}
\newcommand{\CC}{{\mathcal C}}
\newcommand{\CD}{{\mathcal D}}

\newcommand{\CO}{{\mathcal O}}

\newcommand{\CS}{{\mathcal S}}
\newcommand{\CT}{{\mathcal T}}

\newcommand{\CW}{{\mathcal W}}
\newcommand{\CX}{{\mathcal X}}

\newcommand{\Fq}{{\mathfrak{q}}}

\newcommand{\ch}{\mathsf{ch}}

\newcommand{\id}{\mathrm{id}}

\newcommand{\pt}{{\mathsf{p}}}

\newcommand{\blangle}{\big\langle}
\newcommand{\brangle}{\big\rangle}

\newcommand{\Mbar}{{\overline M}}
\newcommand\ev{\operatorname{ev}}
\newcommand{\Pic}{\mathop{\rm Pic}\nolimits}
\newcommand{\PT}{\mathsf{PT}}

\newcommand{\GW}{\mathsf{GW}}
\newcommand{\U}{\mathsf{U}}

\DeclareFontFamily{OT1}{rsfs}{}
\DeclareFontShape{OT1}{rsfs}{n}{it}{<-> rsfs10}{}
\DeclareMathAlphabet{\curly}{OT1}{rsfs}{n}{it}

\newcommand{\p}{\mathbb{P}}

\newcommand\Hilb{\operatorname{Hilb}}

\newcommand{\vd}{\mathsf{vd}}
\newcommand{\vir}{\mathsf{vir}}
\newcommand{\red}{\mathsf{red}}

\newcommand{\full}{\mathsf{full}}
\newcommand{\rel}{\mathrm{rel}}

\DeclareFontFamily{OT1}{rsfs}{}
\DeclareFontShape{OT1}{rsfs}{n}{it}{<-> rsfs10}{}
\DeclareMathAlphabet{\curly}{OT1}{rsfs}{n}{it}

\newcommand{\1}{\mathsf{1}}

\newcommand{\Aut}{\mathrm{Aut}}

\newcommand{\taut}{{\mathrm{taut}}}

\newcommand{\Mon}{\mathsf{Mon}}

\newcommand{\vacuum}{\1}

\newcommand{\K}{{\mathsf{K}}}

\newcommand{\rub}{{\textup{rub}}}
\usepackage{accents}
\newcommand{\dbtilde}[1]{\accentset{\approx}{#1}}

\renewcommand{\d}{{\mathsf{d}}}

\newcommand{\pr}{{\mathsf{pr}}}
\newcommand{\Nak}{{\mathsf{Nak}}}

\title{Marked relative invariants and GW/PT correspondences}
\author{Georg Oberdieck}
\date{\today}

\begin{document}
\maketitle
\setcounter{section}{0}

\begin{abstract}
We introduce marked relative Pandharipande-Thomas (PT) invariants
for a pair $(X,D)$ of a smooth projective threefold and a smooth divisor.
These invariants are defined by integration over
the moduli space of $r$-marked stable pairs on $(X,D)$,
and appear naturally
when degenerating diagonal insertions via the Li-Wu degeneration formula.
We propose a Gromov-Witten (GW) / PT correspondence for marked relative  invariants.
We show compatibility of the conjecture with the degeneration formula and a splitting formula
for relative diagonals.
The results provide new tools to prove GW/PT correspondences
for varieties with vanishing cohomology. 

As an application we prove the GW/PT correspondence for:
\begin{enumerate}
\item[(i)] all Fano complete intersections, and
\item[(ii)] the reduced theories of $(S \times C, S \times \{ z_1, \ldots, z_N \})$
where $S$ is a K3 surface and $C$ is a curve, for all curve classes which have divisibility at most $2$ over the K3 surface.
\end{enumerate} 

In the appendix we introduce a notion of higher-descendent invariants which can be seen as an analogue of the nodal Gromov-Witten invariants
defined by Arg\"uz, Bousseau, Pandharipande and Zvonkine in \cite{ABPZ}.
We show that the higher-descendent invariants reduce to marked relative invariants with diagonal insertions.
\end{abstract}
\setcounter{tocdepth}{1}
\tableofcontents

\section{Introduction}
\subsection{Overview}
There are two basic curve counting invariants
associated to a smooth projective threefold:
\begin{enumerate}
\item[(i)] Pandharipande-Thomas theory, defined via the moduli space of stable pairs,
\item[(ii)] Gromov-Witten theory, defined via the moduli space of stable maps.
\end{enumerate}
A correspondence between these two theories was proposed in \cite{MNOP1, MNOP2, PP_GWPT_Toric3folds}
and proven in many instances in \cite{MOOP, PaPix_GWPT}. In particular, it is known for
all toric threefolds, Calabi-Yau complete intersections in projective space,
and all Fano complete intersections when the cohomology insertions are even.
The goal of this paper is to prove the GW/PT correspondence in two new instances:
For Fano complete intersections with arbitary cohomology insertions,
and for the relative reduced theory of a product of a K3 surface and a curve (with some primitivity assumption on the curve class).
For the proof we introduce \emph{marked relative Pandharipande-Thomas invariants} 
which allow us to control the vanishing cohomology of the threefolds.

\subsection{Pandharipande-Thomas theory}
Let $X$ be a smooth projective 3-fold.
The moduli space of stable pairs $P_{n,\beta}(X)$ parametrizes pairs
$(F,\sigma)$ of a pure $1$-dimensional coherent sheaf on $X$
and a section $\sigma \in H^0(X,F)$ with zero-dimensional cokernel,
satisfying the numerical conditions:
\[
\chi(\mathcal{F})=n\in \mathbb{Z}
\  \ \ \text{and} \ \ \ \ch_2(F)= \beta \in H_2(X,\mathbb{Z})\,.
\]
The moduli space $P_{n,\beta}(X)$ is fine, projective,
and admits a virtual fundamental class \cite{HT}
\[ [ P_{n,\beta}(X)]^{\vir} \in A_{\vd}(P_{n,\beta}(X)), \quad \vd = \int_{\beta} c_1(T_X). \]

Consider the $k$-th descendent of a class $\gamma \in H^{\ast}(X)$ on the moduli space:
\[
\tau_{k}(\gamma) := \pi_{\ast}\left( \pi_X^{\ast}(\gamma) \ch_{2+k}(\BF) \right) \ \ \in H^{\ast}(P_{n,\beta}(X))
\]
where $\pi, \pi_X$ are the projections of $P_{n,\beta}(X) \times X$ to the factors,
and $(\BF,\sigma)$ is the universal stable pair.
The Pandharipande-Thomas invariants of $X$ were defined in \cite{PT} by
\[
\left\langle \, \tau_{k_1}(\gamma_1) \cdots \tau_{k_r}(\gamma_r) \right\rangle^{X,\PT}_{n,\beta}
=
\int_{ [ P_{n,\beta}(X) ]^{\vir} } \prod_i \tau_{k_i}(\gamma_i)
\]
for all $\gamma_i \in H^{\ast}(X)$ and $k_i \geq 0$.

\subsection{Gromov-Witten theory}
Let $\Mbar^{\bullet}_{g,r,\beta}(X)$ be the moduli space of $r$-marked genus $g$ degree $\beta$ stable maps $f : C \to X$,
where the map 
$f$ is required to have positive degree on each connected component of the (possibly disconnected) domain $C$.
Consider the cotangent line classes at the markings:
\[ \psi_i \in H^2(\Mbar^{\bullet}_{g,r,\beta}(X)), \quad i=1,\ldots, r. \]
Gromov-Witten invariants are defined by integration over the virtual fundamental class of the moduli space:
\[
\left\langle \, \tau_{k_1}(\gamma_1) \cdots \tau_{k_r}(\gamma_r) \right\rangle^{X,\GW, \bullet}_{g,\beta}
=
\int_{ [ \Mbar^{\bullet}_{g,r,\beta}(X) ]^{\vir} } \prod_i \psi_i^{k_i} \ev_i^{\ast}(\gamma_i)
\]

\subsection{GW/PT correspondence}
A universal correspondence matrix between the descendent insertions $\tau_{k_i}(\gamma_i)$
in Pandharipande-Thomas (PT) theory and Gromov-Witten (GW) theory was constructed in \cite{PP_GWPT_Toric3folds}.
The matrix\footnote{The universal GW/PT correspondence matrix is denoted by $\tilde{K}_{\alpha, \widehat{\alpha}}$ in \cite{PP_GWPT_Toric3folds, PaPix_GWPT},
and related to our matrix $\K_{\alpha, \widetilde{\alpha}}$ by the variable change:
\[ \K_{\alpha,\widetilde{\alpha}} = \widetilde{\mathsf{K}}_{\alpha, \widetilde{\alpha}}\Big|_{u = -i z}. \]}
\[ \K_{\alpha,\widetilde{\alpha}}\in \mathbb{Q}[i,c_1,c_2,c_3]((z)) \]
is indexed by partitions $\alpha$ and $\widetilde{\alpha}$ of positive size,
and depends on $i = \sqrt{-1}$ and the formal variables $c_i$
which below will be specialized to the Chern classes $c_i(T_X)$.
The basic vanishing 
\[ \widetilde{\mathsf{K}}_{\alpha,\widetilde{\alpha}}=0 \text{ for all } |\alpha| < |\widetilde{\alpha}|. \]
ensures that in the sums below all except finitely many terms are zero.

Let $\alpha = (\alpha_1, \alpha_2 , \ldots, \alpha_r)$ be a partition, and write
\[ \tau_{[\alpha]} = \tau_{\alpha_{1}-1} \cdots \tau_{\alpha_r - 1}. \]
Let $P$ be a set partition of the index set $\{1, \ldots, r \}$.
For any part $T \in P$ given by a subset $T \subset \{ 1, \ldots, r \}$ we
let $\alpha_T := (\alpha_i )_{i \in T}$
be the partition formed from the $T$-indices of $\alpha$.

\begin{defn}[\cite{PP_GWPT_Toric3folds}, Section 0.5]
For any classes $\gamma_1, \ldots, \gamma_r \in H^{\ast}(X)$ define:
\begin{equation*}
\overline{\tau_{\alpha_1-1}(\gamma_1)\cdots
\tau_{\alpha_{\ell}-1}(\gamma_{\ell})}
=
\sum_{\substack{P \textup{ set partitions}\\\textup{ of }\{1,\ldots,r\}}}\ \prod_{T\in P}\ 
\left[ \sum_{\tilde{\alpha}} \tau_{[\tilde{\alpha}]}\Bigg(\Delta_{1, \ldots, \ell(\tilde{\alpha})} \cdot \pi_1^{\ast}\big( \mathsf{K}_{\alpha_T,\tilde{\alpha}} \cdot \prod_{i \in T} \gamma_i \big) \Bigg) \right].
\end{equation*}
where $\tilde{\alpha}$ runs over all partitions,
$\pi_1 : X^{\ell(\tilde{\alpha})} \to X$ is the projection to the first factor,
and $\Delta_{1, \ldots, \ell(\tilde{\alpha})}$ is the class of the small diagonal in $X^{\ell(\tilde{\alpha})}$.
\end{defn}

Using $\d_{\beta} = \int_{\beta} c_1(T_X)$,
we define the partition functions of GW and PT invariants:
\begin{gather*}
Z^{X}_{\PT, \beta}\left( \tau_{k_1}(\gamma_1) \cdots \tau_{k_r}(\gamma_r) \right)
=
\sum_{m \in \frac{1}{2} \BZ} i^{2m} p^m
\big\langle \, \tau_{k_1} (\gamma_1) \cdots \tau_{k_r} (\gamma_r) \big\rangle^{X, \PT}_{m + \frac{1}{2} d_{\beta},\beta}, \\
Z^{X}_{\GW, \beta}\left( \tau_{k_1}(\gamma_1) \cdots \tau_{k_r}(\gamma_r) \right) 
=
(-iz)^{\d_{\beta} } \sum_{g \in \BZ} (-1)^{g-1} z^{2g-2} \left\langle \, \tau_{k_1}(\gamma_1) \cdots \tau_{k_r}(\gamma_r) \right\rangle^{X, \GW, \bullet}_{g, \beta}.
\end{gather*}

\begin{conj}[{\cite[Conjecture 2]{PP_GWPT_Toric3folds}}] \label{conj:GWPT absolute} 
We have that
\[ Z^{X}_{\PT, \beta}\big( \tau_{\alpha_1-1}(\gamma_1) \cdots \tau_{\alpha_r-1}(\gamma_r) \big) \]
is the Fourier expansion of a rational function in $p$, and that
\begin{equation} \label{eqn_correspondence}
Z^{X}_{\PT, \beta}\big(\, \tau_{\alpha_1-1}(\gamma_1) \cdots \tau_{\alpha_r-1}(\gamma_r) \big) \\
=
Z^{X}_{\GW, \beta}\left(\, \overline{\tau_{\alpha_1-1}(\gamma_1) \cdots \tau_{\alpha_r-1}(\gamma_r) } \right)
\end{equation}
under the variable change $p=e^{z}$.
\end{conj}

\begin{rmk}
(i) The matrix $\K_{\alpha,\widetilde{\alpha}}$ was constructed geometrically in \cite{PP_GWPT_Toric3folds}.
A more explicit description for essential descendents (where all $\deg(\gamma_i)>0$) was obtained in \cite{OOP, MoreiraOOP}.
However, an explicit formula for the matrix is still missing and would be very interesting.\\
(ii) The correspondence of Conjecture~\ref{conj:GWPT absolute} is invertible, that 
is, it also determines arbitrary GW invariants in terms of PT invariants, see Remark~\ref{rmk:GW in terms of PT} below.
\end{rmk}

\subsection{Fano complete intersection}
Our first main result is the following:
\begin{thm} \label{thm:Fano complete intersection}
The GW/PT correspondence of Conjecture~\ref{conj:GWPT absolute} holds for any complete intersection
$X \subset \p^n$ which is Fano.
\end{thm}

The case of Theorem~\ref{thm:Fano complete intersection} where all the cohomology insertions $\gamma_i$ are of even degree was proven in \cite{PaPix_GWPT}.
Since Fano complete intersections can have odd cohomology
(e.g. the  cubic threefold has $10$-dimensional middle cohomology),
this extension is non-trivial.

\subsection{$\mathrm{K3} \times \mathrm{Curve}$}
We come to our second main result, concerning the geometry of K3 surfaces.

Let $S$ be a smooth projective K3 surface, let
$C$ be a smooth curve and let $z=(z_1, \ldots, z_N)$ be a tuple of distinct points $z_i \in C$.
We consider the 
relative geometry
\begin{equation} (S \times C, S_z), \quad S_{z} = \bigsqcup S \times \{ z_i \}. \label{relative geometry2} \end{equation}
Let $\beta \in H_2(S,\BZ)$ be a non-zero (i.e. effective) curve class and consider
\[ (\beta, n) = \iota_{\ast} \beta + n [C] \in H_2(S \times C, \BZ) \cong H_2(S,\BZ) \oplus \BZ [C]. \]
where $\iota : S \times \{ x \} \to S \times C$ is the inclusion of a fiber and $\beta \neq 0$.

Consider a $H^{\ast}(S)$-weighted partition
\[ \lambda = \big( (\lambda_1, \delta_1), \ldots, (\lambda_{\ell(\lambda)}, \delta_{\ell(\lambda)}) \big), \quad \delta_i \in H^{\ast}(S) \]
of size $|\lambda| = \sum_i \lambda_i = d$.
We write $\underline{\lambda} = (\lambda_1, \lambda_2, \ldots, \lambda_{\ell(\lambda)})$ for the underlying partition.
Let 
\begin{equation*} \lambda = \frac{1}{\prod_i \lambda_i} \prod_{i} \Fq_i(\delta_i) \vacuum \quad \in H^{\ast}(S^{[n]}) \end{equation*}
be the class on the Hilbert scheme of $d$ points on $S$ associated to $\lambda$, where
\[ \Fq_i(\alpha) : H^{\ast}(S^{[k]}) \to H^{\ast}(S^{[k+i]}) \]
is the $i$-th Nakajima creation operator \cite{Nak} with cohomology weight $\alpha \in H^{\ast}(S)$, see Example~\ref{example:Nakajima}
for the convention on Nakajima operators that we follow.

The relative Pandharipande-Thomas invariants of \eqref{relative geometry2}
with insertions given by $H^{\ast}(S)$-weighted partitions $\lambda_1, \ldots, \lambda_N$ are defined by
the integration over the reduced\footnote{The (standard) perfect obstruction theory
of the moduli space admits a everywhere surjective cosection to a trivial bundle, hence the standard virtual class vanish.
The theory has to be defined with respect to a reduced virtual class, see \cite{MP_GWNL, MPT}.} virtual fundamental class
of the moduli space of relative stable pairs:
\begin{gather*}
\big\langle \, \lambda_1, \ldots, \lambda_N \big\rangle^{(S \times C,S_z), \PT}_{n, (\beta,d)}
=
\int_{ [ P_{n,(\beta,d)}(S \times C,S_z) ]^{\red} }
\ev_{z_1}^{\ast}(\lambda_1) \cdots \ev_{z_N}^{\ast}(\lambda_N).
\end{gather*}
Similarly, relative Gromov-Witten invariants are defined by
\begin{gather*}
\big\langle \, \lambda_1, \ldots, \lambda_N \big\rangle^{(S \times C,S_z),\GW, \bullet}_{g,(\beta,d)}
=
\int_{ [ \Mbar^{\bullet}_{g,(\beta,d)}((S \times C,S_z), \vec{\lambda}) ]^{\red} }
\prod_{i=1}^{N} \prod_{j=1}^{\ell(\lambda_i)} \ev^{\text{rel}}_{i,j}( \delta_{i,j} ) \,,
\end{gather*}
where the supscript '$\bullet$' stands for allowing disconnected domain curves as long as the stable map has non-zero degree on each component.
We refer to Sections~\ref{sec:comparision with std definition} and~\ref{subsec:defn GW invariants} for the precise notation.
We form the partition functions of the Pandharipande-Thomas invariants,
\[ 
Z^{(S \times C,S_z)}_{\PT, (\beta,d)}\left( \lambda_1, \ldots, \lambda_N \right)
=
\sum_{m \in \BZ} (-p)^m
\big\langle \, \lambda_1, \ldots, \lambda_N \big\rangle^{(S \times C,S_z), \PT}_{m + d,(\beta,d)},
\]
and of the Gromov-Witten invariants:
\[
Z^{(S \times C,S_z)}_{\GW, (\beta.d)}\left( \lambda_1, \ldots, \lambda_N \right) \\
=
(-z)^{\sum_i (\ell(\lambda_i) - |\lambda_i|)}
\sum_{g \in \BZ} 
(-1)^{g-1+d} z^{2g-2+2d}
\left\langle \, \lambda_1, \ldots, \lambda_N \, \right\rangle^{(S \times C,S_z), \bullet}_{g, \beta}.
\]

Our second main result is the following.

\begin{thm} \label{thm:K3 GWPT} Let $\beta \in H_2(S,\BZ)$ be a primitive effective curve class. The series
\[
Z^{(S \times C,S_z)}_{\PT, (\beta,d)}\left( \lambda_1, \ldots, \lambda_N \right)
\]
is the Laurent expansion of a rational function in $p$, and
we have the equality
\[
Z^{(S \times C,S_z)}_{\PT, (\beta,d)}\left( \lambda_1, \ldots, \lambda_N \right)
=
Z^{(S \times C,S_z)}_{\GW, (\beta.d)}\left( \lambda_1, \ldots, \lambda_N \right)
\]
under the variable change $p=e^{z}$.
\end{thm}

On the PT side the invariants of $(S \times C,S_z)$ are known to satisfy a multiple cover formula,
which expresses the invariants for imprimitive curve classes $\beta$
as an explicit linear combination of invariants where $\beta$ is primitive,
see \cite[Thm.5.1]{QuasiK3}.
A similar multiple cover formula for the Gromov-Witten invariants of the K3 surface $S$
has been conjectured in \cite[Conj.C2]{K3xE},
and is reviewed in Section~\ref{subsec:MCF K3}.
We prove that the GW/PT correspondence is compatible with these multiple cover formulas.

\begin{prop} \label{prop:MCF implies GWPT}
Let $\beta \in H_2(S,\BZ)$ be any effective curve class.
The series
\[ Z^{(S \times C,S_{z}),\red}_{\PT, (\beta,n)}\left( \lambda_1, \ldots, \lambda_N \right) \]
is the expansion of a rational function in $p$.
Moreover, if the multiple cover formula for K3 surfaces (as in \cite[Conj.C2]{K3xE}, recalled in Conjecture~\ref{conj:MCF}) holds
for the class $\beta$, then
\[ 
Z^{(S \times C,S_{z}),\red}_{\PT, (\beta,d)}\left( \lambda_1, \ldots, \lambda_N \right)
=
Z^{(S \times C,S_{z}),\red}_{\GW, (\beta,d)}\left( \lambda_1, \ldots, \lambda_N \right)
\]
under the variable change $p=e^{z}$.
\end{prop}

The multiple cover formula for K3 surfaces is proven in \cite{BB}
for all curve classes $\beta \in H_2(S,\BZ)$ which are of divisibility $2$.
We obtain the following.
\begin{cor}
Theorem~\ref{thm:K3 GWPT} holds also for effective curve classes $\beta$ of divisibility $2$.
\end{cor}

\subsection{A triangle of correspondences}
Assume that $2g(C)-2+N>0$ so that $(C, z_1, \ldots, z_N)$
is a marked stable curve,
\[ [ (C, z_1, \ldots, z_N) ] \in \Mbar_{g,N}. \]
Recall (e.g. from \cite[Sec.1]{HilbK3}) that there is a canonical decomposition
\[ H_2(S^{[d]},\BZ) \cong H_2(S,\BZ) \oplus \BZ A \]
where $A$ is the class of the exceptional curve of the Hilbert-Chow morphism $S^{[d]} \to S^{(d)}$.

Define the generating series of Gromov-Witten invariants of the Hilbert scheme
with complex structure of the stabilization of the domain curve fixed to be $(C,z)$:
\[
Z^{S^{[d]}, (C,z)}_{\GW, \beta}\left( \lambda_1, \ldots, \lambda_N \right)
=
\sum_{m \in \mathbb{Z}}
(-p)^m
\int_{[ \Mbar_{g(C),N}(S^{[d]}, \beta+mA) ]^{\vir}}
\tau^{\ast}([(C,z)]) \prod_{i=1}^{N} \ev_i^{\ast}(\lambda_i).
\]
Here $\tau$ is the forgetful morphism to the moduli space of stable curves $\Mbar_{g,N}$.

By work of Denis Nesterov we have the following Hilb/PT correspondence:
\begin{thm}[Nesterov, \cite{N1,N2}] \label{thm:Nesterov} For all effective curve class $\beta \in H_2(S,\BZ)$ we have:
\[
Z^{S^{[d]}, (C,z)}_{\GW, \beta}\left( \lambda_1, \ldots, \lambda_N \right)
=
Z^{(S \times C, S_z)}_{\PT, (\beta,d)}\left( \lambda_1, \ldots, \lambda_N \right)
\]
\end{thm}
\vspace{5pt}

Taken Theorem~\ref{thm:K3 GWPT} and Theorem~\ref{thm:Nesterov} together we hence
have established a correspondencess 
for all primitive effective curve classes $\beta \in H_2(S,\BZ)$
between the following theories (taken all in the reduced sense):
\begin{enumerate}
\item[(i)] relative Gromov-Witten theory of $(S \times C,S_z)$,
\item[(ii)] relative Pandharipande-Thomas theory of $(S \times C, S_z)$, and
\item[(iii)] Gromov-Witten theory of $\Hilb^n(S)$ with domain curve fixed to be $(C,z)$
\end{enumerate}
This set of correspondences can be geometrically represented by the triangle
shown in Figure~\ref{figure}.
This triangle was conjectured first in \cite{K3xE}.
While the original conjecture only applied to primitive $\beta$, 
by Proposition~\ref{prop:MCF implies GWPT}
we expect it to hold for imprimitive curve classes as well.

A parallel triangle of correspondences has been established for other surfaces as well,
in particular for $\BC^2$ in \cite{BP, OP, OPLocal}
and for the resolution $\widetilde{\CA_n}$ of the $A_n$-singularity $\CA_n = \BC^2/\BZ_{n+1}$
in \cite{M, MOblomDT, MOblomQHilb, Liu}
(these cases are crucially used in the proof of the GW/PT correspondence \cite{MOOP}).
For arbitrary surfaces we expect a similar triangle but with non-trivial wall-crossing.
The Hilb/PT edge has been recently obtained by Nesterov \cite{N1} (the wallcrossing correction is known for del Pezzo surfaces but open in general).
The GW/PT correspondence is known in special cases (e.g. for toric surfaces), see \cite{PaPix_GWPT}.
We have established here the correspondence for a non-toric case.

A more general version of the triangle,
where the source curve $(C,z)$ is allowed to vary in the moduli space of stable curves,
was established in \cite{PTHigherGenusHilb} for $\BC^2$.
It would be very interesting to prove such a generalization also for K3 surfaces
since it allows access to the full higher genus Gromov-Witten theory of $S^{[n]}$.

The triangle of Figure~\ref{figure} will be crucial for establishing the holomorphic anomaly equation
for the genus $0$ Gromov-Witten theory of $\Hilb^n(K3)$ in the forthcoming work \cite{HilbHAE},
which was the main motivation for the current paper.

\begin{figure}
\begin{center}
\begin{tikzpicture}
 \draw[very thick]
   (0,0) node[below left, align = center]{Gromov-Witten theory of\\ $(K3 \times C, K3 \times z)$}
-- (1,1.732) node[above, align = center]{$(C,z)$-domain Gromov-Witten theory\\of $\Hilb^n(K3)$}
-- (2,0) node[below right, align = center]{Pandharipande-Thomas theory of\\ $(K3 \times C, K3 \times z)$}
-- (0,0);
\end{tikzpicture}
\end{center}
\caption{Triangle of correspondences for K3 surfaces \label{figure}}
\end{figure}
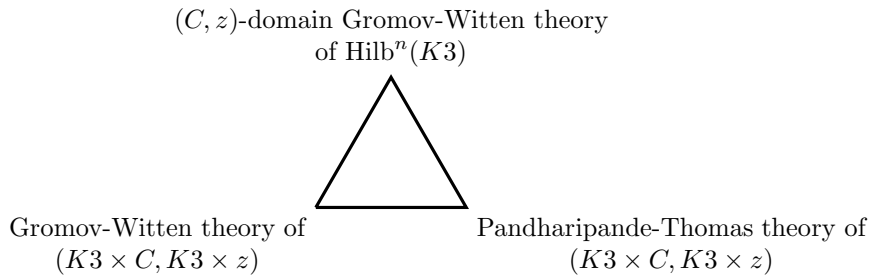

\subsection{Contents}
In Section~\ref{sec:relative geometries} we recall the stack of target expansions for a relative pair $(X,D)$,
the moduli spaces $(X,D)^{r}$ of $r$-marked points on $(X,D)$, and the relative Hilbert schemes of points $(X,D)^{[r]}$.
Section~\ref{sec:relative PT} contains the main definition of the paper, the marked relative invariants.
For that we introduce the stack
\[ P_{n,\beta,r}(X,D) \]
which parametrizes stable pairs $(F,s,p_1, \ldots, p_r)$ together with $r$ marked points on expansions $X[k]$,
where the points $p_i \in X[k]$ are not allowed to lie on the relative divisor $D[k]$ or the singular locus.
Similar to the moduli space of stable maps to a relative target,
the stack admits natural relative and {\em interior} evaluation maps
\[ \ev^{\rel} : P_{n,\beta,r}(X,D) \to D^{[\beta \cdot D]}, \quad \ev : P_{n,\beta,r}(X,D) \to (X,D)^r. \]
Relative-descendents are defined by pullback of cohomology classes via the evaluation maps,
weighted by Chern classes of the universal stable pair pulled back along the universal section.
We show that whenever the cohomology insertion from $(X,D)^r$ is a class pulled back via the canonical projection $(X,D)^r \to X^r$,
these marked relative invariants specialize to the usual descendent PT invariants.
We further discuss a degeneration formula, rubber invariants, rigidification, and a splitting formula for relative big diagonals (parallel to the GW side in \cite{ABPZ})
for the marked relative invariants. All take the expected form.

In Section~\ref{sec:GW theory} we introduce the details we need on the Gromov-Witten side.
The marked relative invariants here are defined by integration over the usual moduli space of relative stable maps,
$\Mbar_{g,r}(X,D, \vec{\lambda})$,
and have been used already in \cite{PaPixJap} for the statement of the GW/PT correspondence in the relative setting.

In Section~\ref{sec:GW/PT corr} we extend the conjectural GW/PT correspondence of \cite{PaPix_GWPT}
to marked relative invariants.
This is most naturally done by viewing the correspondence matrix as a cycle in the product $(X,D)^{r + \ell(\hat{\alpha})}$.
We prove compatibility of the correspondence with the degeneration formula and the splitting of relative big diagonals.
We also remark on possible Chow-theoretic lifts of the GW/PT correspondence.

After having finished introducing our technical tools, we turn to application.
First in Section~\ref{sec:Fano complete intersection} we consider Fano complete intersections $X$.
By using the monodromy and an argument of \cite{ABPZ}, the PT and GW invariants of $X$ are determined by the invariants
where all cohomological insertions are even or products of big diagonals.
Using marked relative invariants these invariants can be determined by a degeneration to complete intersections of lower degree. 
The compatibility of the splitting formula and degeneration formula with the GW/PT correspondence
then shows that it suffices to check the claim for the end points of the degeneration,
which are known by induction and the results of \cite{PaPix_GWPT}.

As a preparation for the K3 case, in Section~\ref{sec:birelative residue theory}
we consider a rational elliptic surface $R$ with a smooth elliptic fiber $E \subset R$.
Following \cite{PaPix_GWPT} we study the birelative capped residue theories
of $(R \times \p^1, R_{\infty} \cup E \times \p^1)$.
Relying on \cite{PaPix_GWPT} we prove the GW/PT correspondence for these invariants for curve classes on $R$ of degree $1$ over the base.
This will be used later in the degeneration formula.

From Section~\ref{sec:GWPT K3} onwards we turn to the particular case $K3 \times \mathrm{Curve}$.
We first discuss partitions functions, degeneration formulas, etc, in the framework of the reduced virtual class.
It is most natural here to introduce a formal parameter $\epsilon \in \BQ[\epsilon]/\epsilon^2$ for the statement of results.
The GW/PT correspondence in the primitive case is stated in Section~\ref{subsec:GWPT K3},
and then proven in Section~\ref{sec:Proof gwpt k3}.
For that the vanishing cohomology for the degeneration of a K3 surface into two rational elliptic surfaces
is controlled by the monodromy of the K3 surface. 
The discussion in the imprimitive case is done in Section~\ref{sec:K3 Curve imprimitive case}.

In the appendix we introduce higher-descendents, which are an analogues of the nodal Gromov-Witten invariants of \cite{ABPZ}.

The content of this work follows many of the existing ideas in the field,
in particular the work of Pandharipande and Pixton \cite{PaPix_GWPT, PP_GWPT_Toric3folds, PaPixJap} on the GW/PT correspondence,
the work of Li and Li-Wu \cite{Li1, Li2, LiWu} on the degeneration formula (as well as \cite{MPT} for the reduced virtual class),
the ideas of Arg\"uz, Bousseau, Pandharipande and Zvonkine \cite{ABPZ} on how to deal with vanishing cohomology,
and various standard constructions such as rigidification \cite{MP,OkPandVir, OPLocal},
Instead of writig a very short paper assuming all of these notions,
we tried to lay out the main ideas of these theories and
give examples and intuition along the way.
Since skipping sections is easy to do, we hope this to be also in the best interest of the reader.

\subsection{Convention} \label{sec:convention}
Given a function $Z : H^{\ast}(X) \to \BQ$. We will often write
\[ Z(\Delta_1) \cdot Z(\Delta_2) \]
and say that $\Delta_1, \Delta_2$ stands for summing over the K\"unneth decomposition of the class of the diagonal $\Delta \subset X \times X$.
This is defined by
\[ Z(\Delta_1) \cdot Z(\Delta_2) := \sum_{i} Z( \phi_i ) Z( \phi_i^{\vee} ) \]
where $[\Delta] = \sum_i \phi_i \otimes \phi_i^{\vee} \in H^{\ast}(X \times X)$ is the K\"unneth decomposition.

To avoid cumbersome sign notation {\em we will assume in Sections 2-5 that all our varieties have even cohomology}.
It is straightforward to introduce signs in all cases: Whenever we switch in an expression the order of two odd cohomology classes,
we need to multiply the resulting expression with a sign.

\subsection{Acknowledgements}
I thank Denis Nesterov and Rahul Pandharipande for discussions related to this work.
The author is funded by the Deutsche Forschungsgemeinschaft (DFG) - OB 512/1-1.

\section{Relative geometries and target expansions} \label{sec:relative geometries}
Let $X$ be a smooth projective variety and let $D \subset X$ be a smooth divisor,
which we assume here for simplicity to be connected.
The general case of disconnected $D$ is parallel.
It only requires us to keep track of number of expansions at each connected component of $D$,
see for example \cite[2.3]{LiWu}.

\subsection{Stack of expansions}
Let $N_{D/X}$ be the normal bundle of $D \subset X$, and consider the projective bundle
\[ \p = \p(N_{D/X} \oplus \CO_{D}) \to D. \]
The projection has two canonical sections
\[ D_{0}, D_{\infty} \subset \p \]
called the zero and infinite section specified by the condition that the zero section has normal bundle $N_{E/\p} \cong N_{E/X}^{\vee}$
and the infinite section has normal bundle $N_{E/X}$.
(In other words, if we consider the degeneration of $X$ to the normal cone of $D$, we find the central fiber $X \cup_{D \cong D_0} \p$.)
For every $k \geq 1$ let
\[ \p_{k} = \p \cup \ldots \cup \p \]
be the chain formed by $k$ copies of $\p$ where the infinity section of the $i$-th copy is glued to the zero section of the $(i+1)$-th copy for $i=1, \ldots, k-1$.
A ($k$-step) \emph{expanded degeneration} of $(X,D)$ is the pair $(X[k], D[k])$
defined for $k \geq 1$ by
\[ X[k] = X \cup_{D \cong D_{1,0}} \p_{k} \]
where $D_{1,0} \subset \p_{k}$ is the zero section in the first copy of $\p$, and
\[ D[k] := D_{k,\infty} \]
is the infinite section in the $k$-th copy of $\p$ in $\p_{k}$.
In the case $k=0$ we set
$(X[0], D[0]) := (X,D)$ and also say the pair is unexpanded.
We call $X[k]$ {\em an expansion} of $X$, and $D[k] \subset X[k]$ the {\em relative divisor of the expansion},
and $\p_k \subset X[k]$ the bubble of the expansion.

The $1$-step expanded degeneration of $(X,D)$ arises naturally as the special fiber of the degeneration of $X$ to the normal cone of $D$.
More generally, a $k$-step expanded degeneration arises naturally by iterated degeneration to the normal cone.
A universal family of expanded degeneration over a stack $\CT$ of expanded degenerations was constructed in \cite{Li1, Li2}.
The universal family is denoted by
\[ (\CX,\CD) \to \CT. \]

We shortly recall the definition of $(\CX,\CD) \to \CT$ following \cite{LiICTP}:
Let $(\CX_0, \CD_0) := (X,D)$.
Let $\CX_1$ be the blow-up of $X \times \BA^1$ along $D \times 0$, and let $\CD_1$ be the proper transform of $D \times \BA^1$.
Inductively, let $\CX_n$ be the blow-up $\CX_{n-1} \times \BA^1$ along $\CD_{n-1} \times 0$, and set $\CD_n$ to be the proper transform of $\CD_{n-1} \times \BA^1$.
For every $n \geq 0$ we have a natural morphism
\[ p_{n} : \CX_{n} \to \BA^{n} \]
given by the composition $\CX \to \CX_{n-1} \times \BA^1 \xrightarrow{p_{n-1} \times \id} \BA^n$.
The fibers of $p_n$ are expanded degenerations of $X$,
and families of expanded degenerations are defined by pullback of this family.
More precisely, one has the following definition, see \cite[Sec.4, p.567]{Li1}:

\begin{defn}
A family of expanded degeneration of $(X,D)$ over a scheme $S$ is a pair $(\CX_S, \CD_S)$ where
\begin{itemize}
\item $\CX_S$ is a scheme over $S$ with an $S$-projection $\CX_S \to X \times S$, and
\item $\CD_S$ is a Cartier divisor of $\CX_S$,
\end{itemize}
such that there exists an \'etale open cover $S_{\alpha} \to S$ of $S$
such that $(\CX_S \times_S S_{\alpha}, \CD_S \times_{S} S_{\alpha})$
is isomorphic to $(f_{\alpha}^{\ast}\CX_n, f_{\alpha}^{\ast}\CD_n)$ for some $n \geq 0$ and morphism $f_{\alpha} : S_{\alpha} \to \BA^n$,
and the isomorphism is compatible with the projections to $X \times S_{\alpha}$.

An isomorphism between $(\CX_S,\CD_S)$ and $(\CX_S', \CD'_S)$ 
is an isomorphism $\CX_S \xrightarrow{\cong} \CX_{S'}$ which sends $\CD_S$ to $\CD'_S$ and is compatible with the projection to $X \times S$.
\end{defn}

We let $\CT$ be the stack whose objects are family of expanded degenerations (see \cite{Li1} why it is a stack).
By construction, there exists a universal expanded degeneration $(\CX,\CD) \to \CT$ consisting of a universal family of expansions
\[ \CX \to \CT \]
given by a flat\footnote{The morphism $\CX_n \to \BA^n$ is flat (e.g. use miracle flatness).}, proper and representable morphism and a universal relative divisor $\CD \subset \CX$.
For any scheme $S$ and morphism $S \to \CT$ we let
denote the pullback of the universal expanded degeneration by $(\CX_S, \CD_S) \to S$,
that is we have the fiber diagram:
\[
\begin{tikzcd}
\CD_S \ar{d} \ar{r} \ar{r} & \CD \ar{d} \\
\CX_S \ar{d} \ar{r} & \CX \ar{d} \\
S \ar{r} & \CT.
\end{tikzcd}
\]
By construction, every geometric fiber of $(\CX_S, \CD_S) \to S$ is isomorphic to some $k$-step expanded degeneration $(X[k], D[k])$.

\begin{rmk}
The stack $\CT$ can be also directly constructed from the pairs $(\CX_n, \CD_n)$ as a limit as follows.
Consider the standard action of $\BG_m^n$ on $\BA^n$ given by
\[ \sigma \cdot (x_1, \ldots, x_n) = (\sigma_1 x_1, \ldots, \sigma_n x_n), \quad \sigma = (\sigma_1, \ldots, \sigma_n) \in \BG_m^n. \]
The action lifts to an action of $\BG_m^n$ on $\CX_n$. Scaling by $\BG_m^n$-certainly yields isomorphic expanded degenerations,
so we certainly need to quotient by this group.
However, the stack $\CX_n \to \BA^n$ also comes with a choice of ordering in which the degenerations are taken,
i.e. the restrictions of $\CX_n$ to $\{ t_j \neq 0 | j \in J \}$ and $\{ t_j \neq 0 | j \in J' \}$ for any two subsets $J, J' \subset \{ 1, \ldots, n \}$ of the same size
are $\BG_m^n$-equivariantly isomorphic, so we need to quotient out by this equivalence relation as well. Here $t_1, \ldots, t_n$ are the coordiantes on $\BA^n$.
The stack $\CT$ is then constructed by taking the limit over $n$ of the the quotients of $\BA^n$ by the joint equivalence relation given by translating by $\BG_m^n$ and reordering
of non-zero coordinates. The universal family $\CX$ is likewise constructed as limits of the quotients of $\CX_n$.

When working over schemes of finite type, the naive picture of viewing $\CT$ as a quotient of $\BA^n$ and the universal family as a quotient of $\CX_n$ hence suffices.
\end{rmk}
\subsection{Cotangent line classes on stack of expansions}
Let $\iota : \CT \times D \cong \CD \to \CX$ be the inclusion of the universal relative divisor.
Following \cite[Sec.1.5.2]{MP} we define the cotangent line bundle on $\CT$ corresponding to $D$ by
\[ \BL_D = p_{\CT \ast}(N_{\CD/\CX}^{\vee} \otimes p_D^{\ast} N_{D/X}) \]
where $p_{\CT}$ and $p_D$ are the projections of $\CT \times D$ to its factors.
Note that in $\p$ we have
$N_{D_0/\p} \cong N_{D_{\infty}/\p}^{\vee}$, so
so that $N_{\CD/\CX}^{\vee} \otimes p_D^{\ast} N_{D/X}$ restricted to a fixed expansion $X[k]$
is simply the product of two normal bundles $N_{D/X} \otimes N_{D_0/\p'}$ (where $\p' \subset \p_{k}$ is the first component).
Since the zero section has normal bundle $N_{D/X}^{\vee}$ we see that
$N_{D/X} \otimes N_{D/\p'}$ is trivial, so $\BL_D$ is in fact a line bundle.
We define the cotangent line class of the divisor $D$ to be
\[ \psi_D = c_1(\BL_D). \]

The line bundle $\BL_D$ on $\CT$ admits a natural section with vanishing locus parametrizing non-trivial expansion.
Consider the universal contraction:
\[ \pi_X : \CX \to X \times \CT \]
Taking the differential along $\CD$ yields a morphism of line bundles on $\CD = \CT \times \CD$,
\[ s : N_{\CD/\CX} \to N_{D/X}. \]
On a fiber $X[k]$ the morphism $s'$ is an isomorphism if $k=0$, and zero if $k > 0$.
Hence the dual morphism
\[ s : \CO_{\CD} \to N_{\CD/\CX}^{\vee} \otimes N_{D/X} \]
vanishes precisely on the locus of non-trivial expansion.
Pushing fordward by $p_{\CT}$, we obtain the section
\[ s_{\CD} = p_{\CD \ast}(s) : \CO \to \BL_D \]
vanishing non-trivial expansions.

\begin{example}
Consider the family of expanded degenerations:
\[ p : \CX_{\p^1} = \mathrm{Bl}_{D \times 0}(X \times \p^1) \to \p^1. \]
We have the fiber $p^{-1}(0) = X \cup_{D \cong D_0} \p(N_{D/X} \otimes \CO)$. 
The relative divisor $\CD_{\p^1} \cong D \times \p^1$ is the proper transform of $D \times \p^1 \subset X \times \p^1$.
In the total space $\CX_{\p^1}$ we have the rational equivalence
\[ [\CD_{\p^1}] = \pi^{\ast}([D]) - [ \p(N_{D/X} \otimes \CO) ]. \]
where $\pi : \CX_{\p^1} \to X \times \p^1$ is the blow-down map (or equivalently, the canonical projection). Hence we find that
\[ N_{\CD_{\p^1}/\CX_{\p^1}} = \CO( \CD_{\p^1} )|_{\CD_{\p^1}} = 
\CO( \pi^{\ast}D ) \otimes \CO( \p(N_{D/X} \otimes \CO) )^{\vee}|_{\CD_{\p^1}}
=
\mathrm{pr}_{D}^{\ast}( N_{D/X} ) \otimes \mathrm{pr}_{\p^1}^{\ast}( \CO_{\p^1}(1) ), \]
where we used that
\[ \CO( \p(N_{D/X} \otimes \CO) )|_{\CD} = \pi^{\ast}( \CO_{\p^1}(0))|_{\CD_{\p^1}} = \mathrm{pr}_{\p^1}^{\ast}( \CO_{\p^1}(1) ). \]
We see that for the classifying morphism $f : \p^1 \to \CT$ of the family we have
\[ f^{\ast} \BL_{D} = \mathrm{pr}_{\p^1, \ast}\left( N_{\CD_{\p^1}/\CX_{\p^1}}^{\vee} \otimes N_{D/X} \right)
= \CO_{\p^1}(1) \]
and $f^{\ast} s_{D}$ is the section of $\CO_{\p^1}(1)$ vanishign at the origin.
\end{example}

\begin{example}
Consider $p_n : \CX_n \to \BA^n$. Then one has a $\BG_m^n$-equivariant isomorphism
\[ \BL_{D} \cong \CO_{\BA^n}( D_1 + \ldots + D_n ) \]
where $D_i = \{ t_i = 0 \}$ is the $i$-th coordinate hyperplane.
We have $s_{D} = t_1 \cdots t_n$.
\end{example}

\subsection{Moduli of ordered points}
\subsubsection{Definition}
Let $(X,D)^r$ be the moduli space of $r$ ordered points on the relative pair $(X,D)$,
which appeared crucially in \cite{PaPix_GWPT} and \cite{ABPZ} and many other places.
It is defined as follows:

\begin{defn}
A family of $r$ ordered points on $(X,D)$ over a scheme $S$ consists of
a tuple $((\CX_S,\CD_S) \to S, p_1, \ldots, p_r)$ where
\begin{itemize}
\item $(\CX_S,\CD_S) \to S$ is an expanded degeneration of $(X,D)$,
\item $p_1, \ldots, p_r : S \to \CX_S$ are sections
\end{itemize}
such that for every geometric point $s \in S$ with $\CX_{S,s} \cong X[\ell]$
\begin{enumerate}
\item[(i)] the points $p_i(s) \in X[\ell]$ do not lie on the relative divisor or the singular locus,
\item[(ii)] The automorphism group $\Aut(\CX_{S,s}, p_1(s), \ldots, p_r(s))$ of all automorphisms of $\CX_{X,s}$ fixing $X$ and preserving the markings
is finite.
\end{enumerate}
An isomorphism of $(\CX_S, p_1, \ldots, p_r)$ and $(\CX_S, p_1', \ldots, p'_r)$
is a isomorphism $\varphi : \CX_S \to \CX_S$ such that $\pi_X \circ \varphi = \pi_X$
and $\varphi \circ p_i = p_i'$,
where $\pi_X : \CX_S \to \CX \to X$ is the canonical projection.
\end{defn}

The universal target $\CX_r := (X,D)^r \times_{\CT} \CX$ over the moduli space fits into the diagram
\begin{equation} \label{fib pr}
\begin{tikzcd}
\CX_r \ar{r} \ar{d}{\pi} & \CX \ar{d} \\
(X,D)^r \ar[bend left]{u}{p_1, \ldots, p_r} \ar{r} & \CT.
\end{tikzcd}.
\end{equation}

\begin{prop} \label{prop:moduli of ordered points}
(i) The functor $(X,D)^r$ is a smooth proper variety of dimension $r \dim(X)$.\\
(ii) For any $r \geq 0$ there exists flat morphisms
\[ \pi_i : (X,D)^{r+1} \to (X,D)^{r} \]
given by forgetting the $i$-th marking and contracting unstable components.
\end{prop}
\begin{proof}
The functor $(X,D)^r$ may be viewed as a special case of the
construction of relative stable maps of Jun Li \cite{Li1}.
Indeed, $(X,D)^r$ is naturally isomorphic to the moduli stack of relative stable maps to $(X,D)$ of degree $0$ with the requirement that every domain curve has $r$ connected components each of genus $0$ and carrying $3$ markings. By the main result of \cite{Li1}
the functor $(X,D)^r$ is hence a proper Deligne-Mumford stack.
Moreover, since $\CT$ is smooth and the relative cotangent complex of the classifying morphism $(X,D)^{r} \to \CT$ is concentrated in degree $0$, the stack $(X,D)^r$ is smooth (an overkill for this argument is to use the deformation theory of relative stable maps discussed in \cite{Li2} and \cite{GV}. In particular it is easy to see
that the obstruction space relative to $\CT$ as given in \cite[(2)]{GV} vanishes.) Its also immediate to see that any object in $(X,D)^r$ has automorphisms only the identity.

To construct the forgetful morphism, let $(\CX_S \to S, p_1, \ldots, p_{r+1})$ be an object of $(X,D)^{r+1}$.
By forgetting the last section we obtain an $r$-pointed family
but which may have infinitely many automorphism over geometric points. As explained in \cite[Sec.3.1]{Li1} there exists a line bundle
(called the standard line bundle associated to the stable map in \cite{Li1}) associated to the tuple $(p_1, \ldots, p_r)$ whose restriction to a fiber $X[\ell]$ is ample on $X$ and all the bubbles $\p$ containing a marking, while for components without a marking it is pulled back from an ample line bundle on $D$.
For the contraction associated to this line bundle,
\begin{equation} \label{contraction map} q: \CX_S \to \overline{\CX}_S, \end{equation}
the tuple $(\overline{\CX}_S, q \circ p_1, \ldots, q \circ p_r)$ is an element of $(X,D)^r$.
This yields the forgetful functor $\pi_{r+1}(X,D)^{r+1} \to (X,D)^r$.

To prove the rest of the proposition the idea is to construct a natural isomorphism
\begin{equation} f : (X,D)^{r+1} \xrightarrow{\cong} \CX_r \label{iso} \end{equation}
such that $\pi_{r+1} = \pi \circ f$.
Indeed, this implies that $(X,D)^r$ is a variety whenever $(X,D)^{r-1}$ is one (since $\CX \to \CT$ is representable), so we may apply induction for the first part.
Moreover, since $\pi$ is flat and $f$ is an isomorphism, we also see that $\pi_{n+1}$ is flat.

To obtain the isomorphism \eqref{iso} we construct a morphism $f : (X,D)^{r+1} \to \CX_r$ as follows. Let $\pi_{r+1} : (X,D)^{r+1} \to (X,D)^r$ be the forgetful morphism and consider the natural commutative diagram
\[
\begin{tikzcd}
\CX_{r+1} \ar{d} \ar{r}{\tilde{q}} & \CX_r \ar{d} \\
(X,D)^{r+1} \ar{r}{\pi_{n+1}} \ar[bend left]{u}{p_{n+1}} & (X,D)^r.
\end{tikzcd}
\]
where $\tilde{q}$ is the composition
of the contraction morphism
$q: \CX_{r+1} \to \pi_{n+1}^{\ast}(\CX_r)$ 
of \eqref{contraction map}
with the natural map $\pi_{n+1}^{\ast}(\CX_r) \to \CX_r$.
We define 
\[ f := \tilde{q} \circ p_{n+1} : (X,D)^{r+1} \to \CX_r. \]
It is easy to see that $f$ is an isomorphism on closed points.
Indeed, consider a $\BC$-point $(X[k], p_1, \ldots, p_r)$ of $(X,D)^r$,
and let $a \in X[k]$ be a point.
If $a$ does not lie on a relative divisor $D[k]$ or the singular locus, then $a \in (\CX_r)_{(X[k], p_1, \ldots, p_r)}$ is the image of the unique point $(X[k],p_1, \ldots, p_r,a)$ under the morphism $f$.
If $a$ lies on the relative divisor $D[k]$, then
it is the image of the unique point
$(X[k+1], p_1, \ldots, p_r, \tilde{a})$
where $\tilde{a}$ is the unique point (up to scaling) on the $(k+1)$-th bubble which does not lie on the zero or infinite section
and projects to $a$ on $D$.
The case $a$ lying in the singular locus is similar (we glue in an extra component).
Finally, to check that $f$ is also scheme-theoretically an isomorphism
can be checked by a local computation.
Alternatively, it also follows by the case considered in \cite{LM}.
Indeed, near the relative divisor, the pair $(X,D)$ 
is \'etale locally isomorphic to an open in $(\BA^1,0) \times U$ for some smooth variety $U$, and hence \'etale locally $(X,D)^r$ is isomorphic to (products of) $(\BA^1,0)^{r} \times U^r$.
In this case, the isomorphism of $(\BA^1,0)^{r}$ with the
universal family over $(\BA^1,0)^{r-1}$ was shown in \cite[Sec.1]{LM}
by writing down an explicit inverse.
\end{proof}

\begin{example}
As a variety $(X,D)^1$ is isomorphic to $X$ (via the canonical projection $\CX \to X$),
and $(X,D)^2$ is the blow-up $\mathrm{Bl}_{D \times D}(X \times X)$.
We refer to \cite[Sec.3.4]{ABPZ} for more discussion and a beautiful self-explaining figure illustrating this case.
\end{example}
\begin{example}
Let $Y$ be a smooth projective variety.
Let $C$ be a smooth curve and let $z=(z_1, \ldots, z_N)$ be a tuple of distinct points $z_i \in C$
such that
\[ (C, z_1, \ldots, z_N) \]
is stable. Consider the relative pair
\[ (X,D) = (Y \times C, Y_z), \quad Y_{z} := \bigsqcup Y \times \{ z_i \}. \]

Following \cite{GV}, see also \cite{ACFW}, the stack of target degenerations $\CT$ of this pair is described very concretely.
Let $\mathrm{st} : \mathfrak{M}_{g,N} \to \Mbar_{g,N}$ be the stabilization map from the
moduli space of $N$-marked genus $g$ nodal curves to the moduli space of stable curves.
The fiber
\[ \mathfrak{M}_{(C,z)} = \mathrm{st}^{-1}([C,z]) \]
parametrizes nodal degenerations of $(C,z)$. Let
\[ \mathfrak{M}_{(C,z)}^{\text{ss}} \subset \mathfrak{M}_{(C,z)} \]
be the open substack parametrizing semi-stable curves, i.e. those marked curves
such that every connected component of the nomralization of genus $0$ carries at least $2$ special points.
Concretely, we allow chains of $\p^1$'s attached to $C$ but the last component in the chain must carry one of the marked points $z_i$.
Then one has
\[ \CT = \mathfrak{M}_{(C,z)}^{\text{ss}}. \]

Moreover, let $\Mbar_{C, (1^{N}, \epsilon^{r})}$ be the moduli space 
parametrizing nodal curves $\Sigma$ together with $N+r$ markings $p_1, \ldots, p_N, q_1, \ldots, q_r$ such that:
\begin{itemize}
\item The stabilization $(\Sigma, p_1, \ldots, P_N)$ is $(C, z_1, \ldots, z_N)$
\item The points $p_i, q_i$ lie in the smooth locus, the $p_i$ are pairwise distinct and distinct from $q_j$,
but the $q_j$ may coincide,
\item Stability condition: Every connected component of the normalization of $\Sigma$ has at least $3$ special points (a special point is either the preimage of a marking or a node).
\end{itemize}
This is a special case of Hassett's construction of moduli space of weighted s table curves \cite{Hassett}.
Then one has
\[ (X,D)^r \cong \Mbar_{C, (1^{N}, \epsilon^{r})} \times Y^r. \]
\qed
\end{example}

\subsubsection{Forgetful morphism}
By Proposition~\ref{prop:moduli of ordered points}, for any $I \subset \{ 1, \ldots, r \}$
we have contraction maps 
\[ \pi_I : (X,D)^r \to (X,D)^{|I|} \]
given by  forgetting all markings except those labeled by $I$ and contracting the unstable components.
Given a class $\gamma \in (X,D)^{|I|}$ we write
\[ \gamma_I = \pi_I^{\ast}(\gamma) \]
for the pullback.
A special case is given by the joint projection:
\[ \pi_1 \times \cdots \times \pi_r : (X,D)^r \to X^r. \]
Given a class $\alpha \in H^{\ast}(X^r)$ we will
denote the pullback $(\prod_i \pi_i)^{\ast}(\alpha)$ simply by $\alpha$,
that is any cohomology class on $X^r$ will be viewed as naturally defining a class on $(X,D)^r$.

\subsubsection{Diagonal}
There exists also a natural diagonal morphism
\[ X \cong (X,D)^1 \to (X,D)^r \]
given on points by $(X[k], p) \mapsto (X[k], p, \ldots, p)$.
We will be particularly interested in the class of the relative diagonal
\[ \Delta_{(X,D)}^{\text{rel}} \subset (X,D)^2. \]
Under the isomorphism $(X,D)^2 \cong \mathrm{Bl}_{D \times D}(X \times X)$
it is the proper transform of the usual diagonal $\Delta_X \subset X \times X$.
Consider the blow-up diagram
\[
\begin{tikzcd}
\p(N) \ar{d}{g} \ar{r}{j} & (X,D)^2 \ar{d}{\pi} \\
D \times D \ar{r}{i} & X^2.
\end{tikzcd}
\]
where $N = \CO(D)|D \oplus \CO(D)|D$ is the normal bundle of $D \times D \subset X \times X$.
The following describes the class of the relative diagonal:
\begin{lemma} \label{lemma:rel splitting}
In $A^{\ast}((X,D)^2)$ we have $\Delta_{(X,D)}^{\text{rel}} = \pi^{\ast}( \Delta_X ) - j_{\ast} g^{\ast}( \Delta_{D} )$.
\end{lemma}
\begin{proof}
This is immediate from the classical blow-up formula \cite[Thm.6.7]{Fulton}.
\end{proof}

\subsection{Rubber target} \label{subsec:moduli of points rubber}
Let $D \subset X$ be our smooth connected divisor 
and recall the projective bundle:
\[ \p = \p(N_{D/X} \oplus \CO_{D}). \]
We consider the expanded target expansions associated to the rubber target
\[ (\p, D_{0,\infty})^{\sim}, \quad D_{0,\infty} = D_0 \sqcup D_{\infty}, \]
where we write $\sim$ to denote the rubber.

The easiest way to define the corresponding stack of target degenerations is as the closed substack
\[ \CT^{\rub} \subset \CT, \]
where $\CT$ is the stack of target expansions of $(X,D)$,
parametrizing expansions which are non-trivially expanded. We let
\[ (\CX^{\rub}, \CD_0 \sqcup \CD_{\infty}) \to \CT^{\rub} \]
be the universal family. Here $\CX^\rub$ can be defined as the complement of the open substack
\[ (X \setminus D) \times \CT^{\rub} \subset \CX|_{\CT^{\rub}}. \]
The infinite divisor is simply $\CD_{\infty} = \CD|_{\CT^\rub}$
and $\CD_{0}$ is the intersection $\CX^{\rub} \cap (X \times \CT^\rub)$.

The moduli space of $r$ ordered points on the rubber $(\p, D_{0,\infty})^{\sim}$
\[ (\p, D_{0,\infty})^{r,\sim} \]
parametrizes tuples
\[ (\p_{\ell}, p_1, \ldots, p_r) \]
where $\p_{\ell} = \p \cup \ldots \cup \p$ is a chain of copies of $\p$ glued as usual,
and $p_1, \ldots, p_r$ are point on $p_{\ell}$ not incident to the singular locus
or the divisors $D_0[\ell] = D_{1,0}$ (the zero section in the first copy) and $D_{\infty}[\ell] = D_{\ell, \infty}$ (the infinite section in the last copy).
Two tuples $(\p_{\ell}, p_1, \ldots, p_r)$ and $(\p_{\ell}, p_1', \ldots, p_r')$
are identified if they differ by an element of $\BG_m^{\ell}$,
where the $i$-th copy of $\BG_m$ acts on the $i$-th copy of $\p$ by the natural scaling action on the fibers of $\p \to D$.
The definition of a family of $r$-ordered points on the rubber over a scheme $S$ is parallel.

As before the moduli space $(\p, D_{0,\infty})^{r,\sim}$ is smooth and proper, it admits forgetful and diagonal morphism.
Because of the scaling action it is of dimension $r \dim(\p) - 1$.
There exists a closed embedding
\[ (\p, D_{0,\infty})^{r,\sim} \hookrightarrow (X,D)^r \]
which defines a non-singular divisor on $(X,D)^r$.

\begin{example}
We have $(\p, D_{0,\infty})^{1,\sim} \cong D$ where the isomorphism is given by the projection $\p \to D$.
For two markings under the isomorphism 
\[ (X,D)^2 \cong \mathrm{Bl}_{D \times D}(X \times X) \]
the space $(\p, D_{0,\infty})^{2,\sim}$
is the exceptional divisor of the blow-up:
\[ (\p, D_{0,\infty})^{2,\sim} \cong \p_{D \times D}(N_{D/X} \oplus N_{D/X}) \cong D \times D \times \p^1. \]
The $\p^1$ records the relative position of the two markings on the bubble.
\end{example}
\begin{example}
In the case of the relative pair $(X,D) = (Y \times C, Y_z)$ we have
\[ \p = Y \times \p^1, \quad D_0 = Y \times \{0 \}, \quad D_{\infty} = Y \times D_{\infty}. \]
The stack of target expansions is $\CT^{\rub} \cong \mathfrak{M}_{0,2}^{\textup{ss}}$.
To describe the space
$(Y \times \p^1, Y_{0,\infty})^{r,\sim}$,
we can use the rigidification arguments of \cite[Sec.3.2]{MOblomDT} and \cite[Sec.4.9]{OPLocal}.
The idea is as follows: Given a tuple $(\p_{\ell}, p_1, \ldots, p_{r})$ we can single out a distinguished component $\p_{i_1} \subset \p_{\ell}$ of the chain $\p_{\ell}$,
namely the one that carries the marking $p_1$.
By the $\BG_m$-action we can further assume that on the distinguished component $\p_{i_1} = \p^1 \times Y$
the point $p_1$ lies over the point $1 \in \p^1$.
We hence view it as an element
\[
((Y \times \p^1)[k_1,k_2], p_1, \ldots, p_r) \in (Y \times \p^1, Y_{0,\infty})^{r}
\]
where $p_1$ is fixed over $1$ and
\[ (Y \times \p^1)[k_1, k_2] = \underbrace{\p \cup \ldots \cup \p}_{k_1 \text{ times}}\ \cup\ \underbrace{Y \times \p^1}_{\p_{i_1}}\ \cup\ \underbrace{\p \cup \ldots \cup \p}_{k_2 \text{ times}}. \]
Fixing the point over $1 \in \p^1$ essentially fixes a slice of the $\BG_m$-automorphism on $p_{i_1}$,
hence also the isomorphisms data are the same.
This identifies 
$(Y \times \p^1, Y_{0,\infty})^{r,\sim}$
with the closed substack of $(Y \times \p^1, Y_{0,\infty})^{r}$ where the first marked point $p_1'$ lies on 
the main component $Y \times \p^1$ over the point $1 \in \p^1$.
In other words,
we have the fiber diagram
\[
\begin{tikzcd}
(Y \times \p^1, Y_{0,\infty})^{r,\sim} \ar{r} \ar{d} & (Y \times \p^1, Y_{0,\infty})^{r} \ar{d}{\mathrm{pr}_1} \\
Y \times \{ 1 \} \ar{r} & Y \times \p^1
\end{tikzcd}
\]
where the right hand vertical arrow is given by the composition of the forgetful morphism
$\pi_1$ with the canonical projection to $Y \times \p^1$. \qed 
\end{example}

The moduli of ordered points on the pair $(\p, D_{0,\infty})$ and the moduli spaec of rubber points $(\p, D_{0, \infty})^{r,\sim}$ are naturally related.
In the former the points lie also on a  chain $\p_{\ell}$ but with a distinguished component
for which we do not identify points by $\BG_m$-automorphism; in the latter there is no distinguised component, all are treated equally.
Forgetting the distinguished component yields a morphism
\[ f : (\p, D_{0,\infty})^{r} \to (\p, D_{0,\infty})^{r,\sim}. \]
More precisely, for a point $x=((Y \times \p^1)[k_1,k_2], p_1, \ldots, p_r)$ it is defined as follows:
If one of the markings $p_i$ lies on the distinguished component $Y \times \p^1$,
we send $x$ to its isomorphism class under the natural $\BG_m$-action on this component
(and hence forgetting the distinguished component). On this locus $f$ is a quotient map by $\BG_m$.
The second case is if none of the markings $p_i$ lie on $Y \times \p^1$,
or equivalently, that $x$ is fixed by the $\BG_m$-action;
the image point is then obtained by contracting the distinguished component.

We have the following rigidification lemma:
\begin{lemma} \label{lemma:point rigidification}
\begin{align*} 
[ (\p, D_{0,\infty})^{r,\sim} ] & = f_{\ast}( \pi_1^{\ast}([D_0]) \cap [(\p, D_{0,\infty})^{r}] )\\
& = f_{\ast}( \pi_1^{\ast}([D_{\infty}]) \cap [(\p, D_{0,\infty})^{r}] )
\end{align*}
where $\pi_1 : (\p, D_{0,\infty})^r \to (\p,D_{0,\infty}) \cong \p$ is the forgetful morphism.
\end{lemma}
\begin{proof}
In both cases the right hand side is of degree $0$ hence a multiple of the fundamental class.
Integrating over the fiber of a generic point yields the result.
\end{proof}

\begin{rmk}
A second connection between $(\p, D_{0,\infty})$ and $(\p, D_{0,\infty})^{r,\sim}$ is given by presenting the second as a $\BG_m$ quotient. Let
\[ U \subset (\p, D_{0,\infty})^{r} \]
be the open subset corresponding to points where there is no expansion over $D_{\infty}$,
and which are not fixed by the $\BG_m$-translation action on the main component $\p$.
Then one has $(\p, D_{0,\infty})^{r,\sim} \cong U / \BG_m$.
(The inverse is provided by viewing $(\p[k], p_1, \ldots, r)$ in $(\p, D_{0,\infty})^{r, \sim}$
as marked points on $\p[k-1,0]$, with the main component in $\p[k-1,0]$ taking the role of the $k$-th component in $\p_{k}$.) \qed
\end{rmk}

We can prove also the following generalization of Lemma~\ref{lemma:rel splitting} which will be important later on.
Let $(X,D)^I$ denote the moduli space of ordered point on $(X,D)$ labelled by a finite set $I$, and similarly in the rubber case.
Consider the fiber diagram:
\[
\begin{tikzcd}
W \ar{r} \ar{d} & (X,D)^r \ar{d}{\pi_{12}} \\
\p(N) \ar{r}{\Delta^{\rel}} & (X,D)^2 
\end{tikzcd}
\]
where $\p(N) = \p(N_{D^2/X^2}) \subset \mathrm{Bl}_{D \times D}(X \times X) \cong (X,D)^2$ parametrizes the locus where both points lie on a bubble.

For any decomposition $\{1, \ldots, r \} = I \sqcup J$ consider the gluing morphism
\[ \xi_{I,J} : (\p, D_{0,\infty})^{I,\sim} \times (X,D)^{J} \to (X,D)^r. \]

\begin{lemma} \label{lemma:2diagonal}
The gluing morphism
\[ \bigsqcup \xi_{I,J} : \bigsqcup_{\substack{ \{1, \ldots, r \} = I \sqcup J \\ 1,2 \in I }}
(\p, D_{0,\infty})^{I,\sim} \times (X,D)^{J} \xrightarrow{\quad \quad} W \subset (X,D)^r \]
is birational.
\end{lemma}
\begin{proof}
This can be checked locally \cite{LM} and is also implicit in \cite{Li1, Li2, LiWu}.
\end{proof}

Recall the class of the locus where the first two markings coincide:
\[ \Delta_{12}^{\rel} = \pi_{12}^{\ast}( \Delta^{\rel} ). \]
We also have the absolute diagonal,
\[ \Delta_{12} = (\pi_1 \times \pi_2)^{\ast}( \Delta_X ) \]
where $\pi_i : (X,D)^r \to (X,D)^1 \cong X$ forgets all but the $i$-th point.

In $(\p, D_{0,\infty})^{s,\sim}$ consider the diagonal
\[ \Delta_{D, 12} = (\pi_1 \times \pi_2)^{\ast}(\Delta_D) \]
where $\pi_i : (\p, D_{0,\infty})^{s,\sim} \to (\p, D_{0,\infty})^{1,\sim} \cong D$.
\begin{cor} \label{cor:diagonal splitting general}
 We have
\[ \Delta_{12}^{\rel} = \Delta_{12} - 
\sum_{\substack{ \{1, \ldots, r \} = I \sqcup J \\ 1,2 \in I }} \xi_{\ast}( \Delta_{D,12} ). \]
\end{cor}
\begin{proof}
This follows from Lemma~\ref{lemma:rel splitting} and Lemma~\ref{lemma:2diagonal}.
\end{proof}

For later on we also record the following immediate computation:
\begin{lemma} \label{lemma:diagonal pullback}
For any decomposition $\{ 1, \ldots, r \} = I \sqcup J$ and elements $a,b \in \{ 1, \ldots, r \}$ we have
\[
\xi_{I,J}^{\ast}( \Delta_{ab}^{\rel}) =
\begin{cases}
0 & \text{ if } (a \in I, b \in J) \text{ or } (a \in J, b \in I) \\
\mathrm{pr}_1^{\ast} \Delta_{(\p,D_{0,\infty})^{\sim}, ab}^{\rel} & \text{ if } a,b \in I \\
\mathrm{pr}_2^{\ast} \Delta_{(X,D),ab}^{\rel} & \text{ if } a,b \in J.
\end{cases}
\]
\end{lemma}

\subsection{Simple degenerations}
Let $W$ be a smooth variety, $B$ be a smooth curve with a distinguished point $0 \in B$.
A simple degeneration (also denoted $W_t \rightsquigarrow X_1 \cup X_2$ for some $t \in B \setminus 0$) is a
flat projective morphism
\[ \epsilon : W \to B \]
such that the fibers $W_b = \epsilon^{-1}(b)$ satisfy:
\begin{itemize}
\item[(i)] $W_b$ is smooth for all $b \in B \setminus \{ 0 \}$,
\item[(ii)] $W_0$ is the union of two smooth irreducible components $X_1, X_2$ glued along a smooth connected divisor $D$.
\end{itemize}

An $k$-step expanded degeneration of the central fiber $W_0$ is the variety
\[ W_0[k] = X_1 \cup \p_{k} \cup X_2 \]
where $\p_k = \p \cup \ldots \cup \p$ is the chain of $k$ copies of $\p = \p(N_{D/X_1} \oplus \CO)$ from before,
$X_1$ is glued along $D$ to the zero section of the first component of $\p_k$,
and $X_2$ is glued to the infinity section of the last component of $\p_k$.
We have
\[ N_{D/X_1} \cong N_{D_2/X}^{\vee} \]
so that for every singular divisor, the two normal directions are normal to each other,
in particular the infinite section of the last copy has normal bundle $N_{D_{\infty}/\p} = N_{D/X_1} \cong N_{D/X_2}^{\vee}$
which is dual to $N_{D/X_2}$.
By \cite{Li1, Li2, LiWu} there exists a stack of target expansions
\[ \CT_{\epsilon} \to B \]
and a universal expansion $\CW \to \CT_{\epsilon}$,
such that for a morphism $S \to \CT_{\epsilon}$ the pullback family $\CW_S$ has fibers $W_s$
whenever $s$ does not lie over $0\in B$ and $W_0[k]$ for some $k$ if $s$ lies over $0$.

We let $(\CW/B)^r$ be the moduli space of $r$ ordered points on $\CW \to \CT_{\epsilon}$.
concretely, an object in $(\CW/B)^r$ over a scheme $S$ is a tuple
\[ (f : S \to \CT_{\epsilon}, p_1, \ldots, p_r) \]
where $p_1, \ldots, p_r$ are sections of the pullback family $\CW_{S} \to S$,
\[
\begin{tikzcd}
\CW_S \ar{r} \ar{d} & \CW \ar{d} \\
S \ar{r} \ar[bend left]{u}{p_1, \ldots, p_r} & \CT_{\epsilon} \ar{r} & B
\end{tikzcd}
\]
with the property that for every geometric point $s \in S$ lying over a $0 \in B$ 
and with fiber $\CW_{s} \cong W_0[k]$ we have:
\begin{itemize}[itemsep=0pt]
\item[(i)] $p_1(s), \ldots, p_r(s)$ do not meet the singular locus or the relative divisor,
\item[(ii)] the automorphism $\Aut(W_0[k], p_1(s), \ldots, p_r(s))$ is finite,
where we the automorphism group consists of the elements in $\BG_m^k$ acting on $\p_k$ which fix the markings.
\end{itemize}

Then $(\CW/B)^r$ is a smooth proper variety, and $\epsilon_r : (\CW/B)^r \to B$ is flat.
The fiber over a point $t \in B \setminus \{ 0 \}$ is simply
\[ \epsilon_r^{-1}(t) = W_t^r. \]
Consider the fiber of $\epsilon_r$ over $0$,
\[
\begin{tikzcd}
(\CW/B)^r_0 \ar{r} \ar{d} &  (\CW/B)^r \ar{d} \\
0 \ar{r} & B.
\end{tikzcd}
\]
Then we have a natural gluing morphism, 
\[ \bigsqcup_{\{1, \ldots, r \} = I \sqcup J}
(X_1, D)^{I} \times (X_2, D)^J \to (\CW/B)^r_0
\]
which is finite.
The diagonal and forgetful morphisms of $(\CW/B)^r$ are defined similarly to before.

\subsection{Relative Hilbert scheme of points} \label{subsec:rel hilb scheme}
Let $S$ be a smooth projective surface and let $D \subset S$ be a (connected) non-singular divisor.
For an integer $d \geq 0$ let
\[ (S,D)^{[d]} \]
be the Hilbert scheme of $d$ points on the relative pair $(S,D)$
which parametrizes length $d$ zero-dimensional subschemes on expansions $S[k]$,
which do not meet the relative divisor or the singular locus,
and have finite automorphism, with automorphisms given by the scaling-action of $\BG_m^k$
on the fibers of the bubbles.
The Hilbert scheme $(S,D)^{[d]}$ is a smooth Deligne-Mumford stack
and is a special case of the construction in \cite{LiWu}.
We also refer to \cite{Setayesh} for a study of its cohomology.

\begin{example}
We have $(S,D)^{[1]} \cong (S,D)^1 \cong S$. For two points consider the $\BZ_2$-action on
\[ \mathrm{Bl}_{\Delta^{\rel}_{(S,D)}}( \mathrm{Bl}_{D \times D}(S \times S) ) \]
that switches the two factors. It has fixed locus equal to the exceptional divisor of the blow-up (a $\p^1$-bundle over $\Delta_{(S,D)}^{\rel}$)
as well as the (the preimage of the) dimension $1$ locus $\Sigma$ in $\p(N)$ corresponding to pairs of points which are interchanged by
the action of the involution
$(-1) \in \BG_m$ on the bubble (the $\BZ_2$ action restricted to $\p(N) \cong \p^1 \times D^2$ is the product of the action on $D^2$-switching the two factors
and an involution on $\p^1$. Any involution on $\p^1$ has two fixed loci,
which correspond here to the intersection of the relative diagonal with $\p(N)$ and the locus $\Sigma \cong D$.)
The Hilbert scheme $(S,D)^{[2]}$ is then given as the quotient
\[ (S,D)^{[2]} \cong ( \mathrm{Bl}_{\Delta^{\rel}_{(S,D)}}( \mathrm{Bl}_{D \times D}(S \times S) ) / \BZ_2 \]
where we take the ordinary scheme-theoretic quotient along the exceptional divisor, and the stack quotient along $\Sigma$.
In particular, $(S,D)^{[2]}$ is a smooth Deligne-Mumford stack with singular coarse moduli space. \qed
\end{example}

We recall the construction of Nakajima cycles on the relative Hilbert schenme $(S,D)^{[d]}$
following \cite[Sec.5.4]{PaPix_GWPT}.
Given a partition $\lambda = (\lambda_1, \ldots, \lambda_{\ell})$ let
\[ Z_{(S,D), \lambda} \subset (S,D)^{\ell} \times (S,D)^{[d]} \]
be the closure of the locus of distinct points in $(S,D)^{\ell}$ carrying punctual subschemes of length $\lambda_1, \ldots, \lambda_{\ell}$
For any $\gamma \in H^{\ast}((S,D)^{\ell})$ we then define
\[ \mathsf{Nak}_{\lambda}(\gamma) = \mathrm{pr}_{2 \ast}\left( \mathrm{pr}_{1}^{\ast}(\gamma) \cdot [Z_{(S,D), \lambda}] \right), \]
where $\pr_i$ are the projection of $(S,D)^{\ell} \times (S,D)^{[d]}$ to the factors.

\begin{example} \label{example:Nakajima}
Assume that $D$ is empty, so that $(S,D)^{[d]}$ is simply the ordinary Hilbert scheme of points $S^{[d]}$ on the surface $S$.
For $i > 0$ and $\alpha \in H^{\ast}(S)$ recall the Nakajima operators
\[ \Fq_i(\alpha) : H^{\ast}(S^{[a]}) \to H^{\ast}(S^{[a+i]}), \ \gamma \mapsto \rho_{2 \ast}( \rho_1^{\ast}(\gamma) \cdot q^{\ast}(\alpha)) \]
with $\rho_1, q, \rho_2$ the projections to the factors of the incidence scheme:
\[ S^{[a,a+i]} = \{ (I_1, x, I_2) \in S^{[a]} \times S \times S^{[a+i]} | I_1 \subset I_2, \quad \mathrm{Supp}(I_2/I_1) = \{ x \} \}. \]
Then by \cite{Nak} (see also \cite{deCM}) the cycle $[(S,\varnothing)^{(\lambda)}]$ is precisely the class of $\Fq_{\lambda_1} \cdots \Fq_{\lambda_{\ell}}( - )$ viewed as
a cycle in $S^{\ell} \times S^{[d]}$.
Hence in this case, with $\gamma = \gamma_1 \otimes \ldots \otimes \gamma_{\ell} \in X^{\ell}$ we have
\[ \mathsf{Nak}_{\lambda}(\gamma) = \prod_{i=1}^{\ell} \Fq_{\lambda_i}(\gamma_i) \vacuum. \]
\qed
\end{example}

Given a simple degeneration $\epsilon : W \to B$ of the surface $S$ into the union $S_1 \cup_D S_2$,
we also have the relative Hilbert scheme of points on the fibers of $\epsilon$
\[ (\CW/B)^{[d]} \to B \]
which parametrizes $0$-dimensional subschemes on the fibers $W_t[k]$ for $t \in B$
which do not lie on the singular locus, where we identify subschemes modulo the usual scaling action on the bubbles.
We can again construct relative Nakajima cycles in a straightforward way as follows:
Let 
\[ Z_{\CW/B, \lambda} \subset (\CW/B)^{\ell} \times (\CW/B)^{[d]} \]
be the closure of the locus of distinct points in $(\CW/B)^{\ell(\lambda)}$
carrying punctual subschemes of length $\lambda_1, \ldots, \lambda_{\ell}$.
Given $\gamma \in H^{\ast}( (\CW/B)^{\ell} )$ we define
\[ \mathsf{Nak}_{\lambda}(\gamma) = \mathrm{pr}_{2 \ast}\left( \mathrm{pr}_{1}^{\ast}(\gamma) \cdot [ Z_{\CW/B, \lambda}] \right) 
\quad \in H^{\ast}( (\CW/B)^{[d]} ).
\]

Over $0 \in B$ for any $i+j=d$ we have a natural finite gluing morphism:
\[ \xi_{i} : (S_1, D)^{[i]} \times (S_2,D)^{[j]} \to \left( (\CW/B)^{[d]} \right)_0. \]

\begin{lemma} \label{lemma:rel hilb gluing}
We then have:
\[
\xi_{i}^{\ast}( \mathsf{Nak}_{\lambda}(\gamma))
=
\sum_{\substack{
\{1 ,\ldots , \ell(\lambda)\} = I \sqcup J \\
|\lambda_I|=i, |\lambda_J|=j}} \Nak_{\lambda_{I}}( \delta_{I,s} ) \otimes \Nak_{\lambda_{J}}( \delta_{I,s}' ),
\]
where 
\begin{itemize}
\item $\lambda_I = (\lambda_i)_{ i \in I}$ is the partition formed from the $I$-parts of $\lambda$,
\item we have used the K\"unneth decomposition
\[ \xi_I^{\ast}( \gamma ) = \sum_{s} \delta_{I,s} \otimes \delta_{I,s}' \in H^{\ast}( (S_1,D)^I \times (S_2,D)^J ) \]
where $\xi_I : (S_1,D)^I \times_{B} (S_2,D)^J \to (\CW/B)^{\ell(\lambda)}$ is the gluing morphism.
\end{itemize}
\end{lemma}
\begin{proof}
We have the diagram
\[
\begin{tikzcd}
\bigsqcup_{I,J} (S_1,D)^I \times (S_2,D)^J \ar{d}{\xi_I} & 
\bigsqcup_{I,J} Z_{(S_1,D), \lambda_I} \times Z_{(S_2,D), \lambda_J} \ar{l}{\pr_1} \ar{r}{\pr_2}\ar{d} & (S_1, D)^{i} \times (S_2,D)^{j} \ar{d}{\xi_i} \\
(\CW/B)^{\ell(\lambda)} & Z_{\CW/B,\lambda} \ar{l}{\pr_1} \ar{r}{\pr_2} & (\CW/B)^{[d]}
\end{tikzcd}
\]
where $I,J$ runs over the same data as in the claim.
The left square is commutative and one checks that the right square is fibered.
The claim follows by a diagram chase.
\end{proof}

\subsection{Logarithmic tangent bundle}
\label{subsec:log tangent bundle}
For use later on we also introduce the logarithmic tangent bundle. Let $D \subset X$ be a simple normal crossing divisor in a smooth variety $X$,
that is $D = \sum_i D_i$ for smooth irreducible divisors $D_i$
and, \'etale locally $D$ is the union of hypersurfaces with local equation $z_1 \cdots z_k = 0$,
where $z_1, \ldots, z_n$ are the local coordinates of $X$.
Let $\Omega_X[D]$ be the locally free sheaf of differential forms with logarithmic poles along $D$.
If $D$ is locally given by $z_1 \cdots z_k = 0$, then the stalk $\Omega_{X}[D]$ at $z_1 = \ldots = z_k=0$ is
\[ \Omega_{X,p}[D] = \CO_{X,p} \frac{dz_1}{z_1} \oplus \cdots \oplus \CO_{X,p} \frac{dz_k}{z_k} \oplus \CO_{X,p} dz_{k+1} \oplus \cdots \oplus \CO_{X,p} dz_{n}. \]

In particular, we have the exact sequence:
\begin{equation} 0 \to \Omega_X \to \Omega_X[D] \to \bigoplus_{i} \iota_{i \ast} \CO_{D_i} \to 0 \label{log tangent ses} \end{equation}
and the restriction to the $i$-th divisor is
\[
\Omega_X[D]|_{D_i} = \CO_{D_i} \oplus \Omega_{D_i}\left[ \sum_{j \neq i} (D_i \cap D_j) \right].
\]

The logarithmic tangent bundle is defined to be
\[ T_X[-D] = \Omega_X[D]. \]

\section{Relative Pandharipande-Thomas theory} \label{sec:relative PT}
Let $X$ be a smooth projective threefold and let $D \subset X$ be a smooth connected divisor.
(Again, the case of non-disconnected $D$ is not more difficult.)

\subsection{Moduli spaces} \label{subsec:moduli space of r-marked stable pairs}
Let $(\CX,\CD) \to \CT$ be the universal family over the stack of expanded degenerations.
For a morphism $S \to \CT$ let $(\CX_S,\CD_S) = (\CX \times_{\CT} S, \CD \times_{\CT} S) \to S$ denote the pull-back family.

\begin{defn}[\cite{LiWu}]
A relative stable pair on $(X,D)$ over a scheme $S$ is a triple
$((\CX_S,\CD_S) \to S, F, \sigma)$ where 
\begin{itemize}
\item $(\CX_S,\CD_S) \to S$ is an expanded degeneration of $(X,D)$ over $S$,
\item $F$ is a coherent sheaf on $\CX_S$ flat over $S$,
\item $\sigma \in H^0(\CX_S, F)$ is a section,
\end{itemize}
such that on every geometry fiber $\CX_{S,s} \cong X[\ell]$
the restriction $(F_s, \sigma_s)$ satisfies the following conditions:
\begin{enumerate}
\item[(i)] $F_s$ is pure $1$-dimensional,
\item[(ii)] the cokernel of $\sigma_s$ is zero-dimensional,
\item[(iii)] the sheaf $F_s$ is normal to the relative divisor $D[\ell] \subset X[\ell]$ and the singular locus,
and the cokernel of $\sigma_s$ is disjoint from the relative divisor $D[\ell]$ and the singular locus,
\item[(iv)] the automorphism group $\Aut_{\CX_s}(F_s,\sigma_s)$ is finite.
\end{enumerate}
An isomorphism between relative stable pairs
$((\CX_S,\CD_S) \to S, F, \sigma)$ and $((\CX_S,\CD_S) \to S, F', \sigma')$
is an automorphism $\varphi : (\CX_S,\CD_S) \to (\CX_S,\CD_S)$
together with isomorphisms $\varphi^{\ast}(F) \cong F'$ and $\varphi^{\ast}(\sigma) = \sigma'$.
In particular, we identify relative stable pairs if they differ by a scaling automorphism of the bubble.
\end{defn}

Let $\beta \in H_2(X,\BZ)$ is a curve class and $n \in \BZ$. We will write
\[ P_{\Gamma}(X,D), \quad \Gamma = (n,\beta) \]
for the moduli space of relative stable pairs $\CO_{X[\ell]} \xrightarrow{\sigma} F$ with numerical conditions
\[ \beta = p_{\ast} \ch_2(F), \quad n = \chi(F) = \ch_3(F) + \frac{1}{2} \int_{\beta} c_1(T_X), \]
where $p : X[\ell] \to X$ is the canonical projection map contracting the bubbles.
By \cite{LiWu} $P_{\Gamma}(X,D)$ is a proper Deligne-Mumford stack.

We will be interested to import cohomology classes from $(X,D)^r$ to the moduli space $P_{\Gamma}(X,D)$.
For that we require a version of $P_\Gamma(X,D)$ which involves \emph{marked points}.

\begin{defn}
A $r$-marked relative stable pair on $(X,D)$ over a scheme $S$ is a tuple
$(S \to \CT, F, \sigma, p_1, \ldots, p_r)$ where
\begin{itemize}
\item $(\CX_S,\CD_S) \to S$ is an expanded degeneration of $(X,D)$ over $S$,
\item $F$ is a coherent sheaf on $\CX_S$ flat over $S$
\item $\sigma \in H^0(\CX_S, F)$ is a section,
\item $p_1, \ldots, p_r : S \to \CX_S$ are sections
\end{itemize} 
such that on every geometry fiber
$\CX_{S,s} \cong X[\ell]$
the restriction $(F_s, \sigma_s, p_1(s), \ldots, p_r(s))$ satisfies conditions (i-iii) above as well as
\begin{enumerate}
\item[(iv)] $p_1(s), \ldots, p_r(s)$ do not lie on the relative divisor or in the singular locus of $X[\ell]$,
\item[(v)] The automorphism group respecting the markings $\Aut_{\CX_s}(F_s,\sigma_s, p_1(s), \ldots, p_r(s))$ is finite.
\end{enumerate}
Isomorphisms are defined as isomorphisms of relative stable pairs as before but with the additional assumption that
$\varphi$ commutes with the sections: $\varphi_i \circ p_i = p_i$.
\end{defn}

We write $P_{\Gamma,r}(X,D)$ for the moduli stack of $r$-marked relative stable pairs
with numerical conditions fixed by $\Gamma$ as before.
The universal target 
\[ \CX_r := P_{\Gamma,r}(X,D) \times_{\CT} \CX \to P_{\Gamma,r}(X,D) \]
over the moduli space fits into the diagram
\begin{equation} \label{universal family P Gamma r}
\begin{tikzcd}
\CX_r \ar{r} \ar{d} & \CX \ar{d} \\
(X,D)^r \ar[bend left]{u}{p_1, \ldots, p_r} \ar{r} & \CT.
\end{tikzcd}.
\end{equation}

\begin{prop} \label{prop:moduli of r marked stable pairs}
(i) The functor $P_{\Gamma,r}(X,D)$ is a proper Deligne-Mumford stack.\\
(ii) For any $r \geq 0$, there exists flat morphisms
\[ \pi_i : P_{\Gamma,r+1}(X,D) \to P_{\Gamma,r}(X,D) \]
given by forgetting the $i$-th marking and contracting unstable components.
\end{prop}

\begin{proof}
Similar arguments already appeared in \cite[Sec.4.9]{OPLocal} or \cite[Sec.3.2]{MOblomDT}
in the context of rigidification.
To construct the forgetful morphism
one argues precisely as in
the proof of Proposition~\ref{prop:moduli of ordered points}:
Given an $(r+1)$-marked stable pair over a scheme $S$,
\[ (\CX_S \to S, F, \sigma, p_1, \ldots, p_{r+1}), \]
consider the tuple where the $(r+1)$-th marking is forgotten. By using the standard line bundle of \cite{LiWu} twisted by the markings
one obtains then an associated contraction 
\[
q : \CX_S \to \overline{\CX}_S
\]
such that on geometric fibers the map $q$ contracts all bubbles which do not contain any markings and on which the stable pair is $\BG_m$-equivariant. 
It follows that on $\CX_S$ we have
\[ (F,\sigma) \cong (q^{\ast} \overline{F}, q^{\ast}( \overline{\sigma}) ) \]
for a unique stable pair $(\overline{F}, \overline{\sigma})$ on $\overline{\CX}_S$. The morphism
\[ \pi_{n+1} : P_{\Gamma,r+1}(X,D) \to P_{\Gamma,r}(X,D) \]
is defined by
\[ (\CX_S \to S, F, \sigma, p_1, \ldots, p_{r+1}) 
\mapsto
( \overline{\CX}_S \to S, \overline{F}, \overline{\sigma}, q \circ p_1, \ldots, q \circ p_{r}). \]

Consider the diagram
\[
\begin{tikzcd}
\CX_{r+1} \ar{r}{\tilde{q}} \ar{d} & \CX_{r} \ar{d} \\
P_{\Gamma,r+1}(X,D) \ar{r}{\pi_{r+1}} & P_{\Gamma,r}(X,D).
\end{tikzcd}
\]
Then as in the proof of Proposition~\ref{prop:moduli of ordered points} one argues that
\[ f := \tilde{q} \circ p_{n+1} : P_{\Gamma,r+1}(X,D) \to \CX_r \]
is an isomorphism (it is an isomorphism on closed points
and the fiber of closed points are closed; for the latter the stable pair does not play a role and hence this can be argued as in Proposition~\ref{prop:moduli of ordered points}).
Since $\CX_r \to P_{\Gamma,r}(X,D)$ is flat, proper and representable
(as a base change of $\CX \to \CT$) we conclude Claim (ii) directly,
and Claim (i) by induction on $r$.
\end{proof}

\begin{example}
In case that $D$ is empty, we have $P_{\Gamma,r}(X,D) = P_{\Gamma}(X) \times X^r$.
\end{example}
\begin{rmk}
Contrary to the moduli space of $r$-marked stable maps here we do not require the markings to be distinct here.
The intuitive reason is that they correspond on the Gromov-Witten side to the images of the marked points in the target,
which also do not have to be distinct.\qed
\end{rmk} 

\begin{rmk}
Consider the classifying morphism of the universal target over $P_{\Gamma,r}(X,D)$,
\[ f_r : P_{\Gamma,r}(X,D) \to \CT. \]
Note that these maps usually do \emph{not} commute with the morphisms forgetting the $i$-th marking, that is
$f_{r -1} \circ \pi_i$ and $f_r$ are not isomorphic in general.
The reason is that the forgetful morphism is defined by contracting unstable components which changes the universal target (and hence the map $f_r$).
The natural morphism
\[
\pi_{1} \times \cdots \pi_r : P_{\Gamma,r}(X,D) \to P_{\Gamma}(X,D) \times_{\CT} \underbrace{\CX \times_{\CT} \ldots \times_{\CT} \CX}_{r \text{ times }} \]
is almost always not an isomorphism.
This is similar to the case of Gromov-Witten theory. The morphism to the Artin stacks of prestable curves do not commute with the forgetful morphisms on $\Mbar_{g,n}(X,\beta)$ since we contract  curves. \qed
\end{rmk}

By \cite{LiWu} the moduli space $P_{\Gamma}(X,D)$ carries a virtual fundamental class $[ P_{\Gamma}(X,D) ]^{\vir}$ of dimension $\int_{\beta} c_1(T_X)$.
Let $\pi : P_{\Gamma,r}(X,D) \to P_{\Gamma}(X,D)$ be the map forgetting all the markings and contracting unstable components,
which by Proposition~\ref{prop:moduli of r marked stable pairs} are flat and proper.
We define the virtual class of $P_{\Gamma,r}(X,D)$ by
\[
[ P_{\Gamma,r}(X,D) ]^{\vir} = \pi^{\ast} [ P_{\Gamma}(X,D) ]^{\vir} \in A_{\ast}( P_{\Gamma,r}(X,D) )
\]
which is of dimension $\int_{\beta} c_1(T_X) + 3r$.

\begin{rmk}
This is similar to the case of Gromov-Witten theory, where for the forgetful morphism $\pi : \Mbar_{g,n}(X,\beta) \to \Mbar_{g}(X,\beta)$ we have that
\[ [ \Mbar_{g,n}(X,\beta) ]^{\vir} = \pi^{\ast} [\Mbar_g(X,\beta) ]^{\vir}. \]
\end{rmk}

\subsection{Marked relative Invariants} \label{subsec:relative descendents}
For every $r \geq 0$ we have an evaluation map at the relative divisor
\[ \ev_{D}^{\rel} : P_{\Gamma,r}(X,D) \to D^{[\ell]}, \quad \ell = D \cdot \beta \]
defined by sending a point $(X[k], F, \sigma, p_1, \ldots, p_r)$ to the intersection $\CO_{D[k]} \xrightarrow{\sigma|_{D[k]}} F|_{D[k]}$
viewed as an element in the Hilbert scheme of $\ell$ points on $D$.

On the moduli spaces of $r$-marked stable pairs we also have interior evaluation maps:
\[ \ev : P_{\Gamma,r}(X,D) \to (X,D)^r \]
whose image of a point $(X[k], F, s, p_1, \ldots, p_r)$ is obtained by forgetting $(F,s)$ and then contracting all unstable bubbles (precisely those $\p$ that do not contain any marking).
Let 
\[ \CX_r \to P_{\Gamma,r}(X,D) \]
be the universal target over the moduli space,
and let 
\[ p_1, \ldots, p_r : P_{\Gamma,r}(X,D) \to \CX_r. \]
be the universal sections, see  \eqref{universal family P Gamma r}.
Let $(\BF,\sigma)$ be the universal stable pair on $\CX_r$.

For any $\lambda \in H^{\ast}(D^{[\ell]})$ and $\gamma \in H^{\ast}((X,D)^r)$
and integers $k_1, \ldots, k_r \in \BZ$ we define:

\begin{defn} 
The marked relative Pandharipande-Thomas invariants are defined by
\begin{multline*}
\left\langle \, \lambda \middle| \tau_{k_1} \cdots \tau_{k_r}(\gamma) \right\rangle^{(X,D), \PT, \textup{marked}}_{\Gamma} \\
=
\int_{ [ P_{\Gamma,r}(X,D) ]^{\vir} }
\ev^{\mathrm{rel} \ast}_D(\lambda) \cdot
p_1^{\ast}(\ch_{2+k_1}(\BF)) \cdots p_{r}^{\ast}( \ch_{2+k_r}(\BF)) \cdot \ev^{\ast}(\gamma).
\end{multline*}
\end{defn}

\vspace{5pt}

Consider a partition
\[ \{ 1 , \ldots, r \} = I_1 \sqcup \ldots \sqcup I_s \]
and write $I_{j} = (i_{j,b})_{b=1}^{\ell(I_j)}$. If $\gamma$ is of the form
\[ \gamma = \prod_{j} \pi_{I_j}^{\ast}(\gamma_j) \]
for some classes $\gamma_j \in H^{\ast}( (X,D)^{I_j} )$ then 
we will also write
\[
\tau_{k_1} \cdots \tau_{k_r}(\gamma) = 
\left( \tau_{k_{i_{1,1}}} \cdots \tau_{k_{i_{1,\ell(I_1)}}}( \gamma_1 ) \right) \cdots \left( \tau_{k_{i_{s,1}}} \cdots \tau_{k_{i_{s,\ell(I_s)}}}( \gamma_s ) \right).
\]
For example, if $\gamma = \pi_1^{\ast}(\gamma_1) \cdots \pi_r^{\ast}(\gamma_r) \in H^{\ast}((X,D)^r)$ for some $\gamma_i \in H^{\ast}(X)$, we write
\[ \tau_{k_1} \cdots \tau_{k_r}(\gamma) = \tau_{k_1}(\gamma_1) \cdots \tau_{k_r}(\gamma_r). \]

\subsection{Comparision with standard definition} \label{sec:comparision with std definition}
Recall the usual definition of descendent invariants in relative Pandharipande-Thomas theory.
Let $P_{\Gamma}(X,D)$ be the moduli space of (unmarked) stable pairs on $(X,D)$,
and let $(\BF,s)$ be the universal stable pair on $P_{\Gamma}(X,D) \times_{\CT} \CX$.
Consider the diagram
\[
\begin{tikzcd}
& P_{\Gamma}(X,D) \times_{\CT} \CX \ar{dr}{\pi_X} \ar{dl}[swap]{\rho} &  \\
P_{\Gamma}(X,D) & & X
\end{tikzcd}
\]
where $\pi_X$ is the projection to $\CX$ followed by the universal contraction morphism $\CX \to X$.

Given $k \geq 0$ and a class $\gamma \in H^{\ast}(X)$ we define the descendents
\[ \tau_{k}(\gamma) = \rho_{\ast}( \ch_{2+k}(\BF ) \cup \pi_X^{\ast}(\gamma) ) \in H^{\ast}(P_{\Gamma}(X,D)), \]
viewed here in operational cohomology.

Let $\lambda \in H^{\ast}(D^{[\beta \cdot D]})$ be a cohomology class, and
let $k_i \geq 0$ and $\gamma_i \in H^{\ast}(X)$.
\begin{defn} 
The descendent Pandharipande-Thomas invariants are defined by
\[
\left\langle \, \lambda \middle| \tau_{k_1}(\gamma_1) \cdots \tau_{k_r}(\gamma_r) \right\rangle^{(X,D), \PT, \textup{std.-desc.}}_{\Gamma} \\
=
\int_{ [ P_{\Gamma}(X,D) ]^{\vir} }
(\ev_D^{\rel})^{\ast}(\lambda) \cup \prod_{i=1}^{r} \tau_{k_i}(\gamma_i).
\]
\end{defn}

We have the following comparision result which says that for
\[ \gamma = \pi_1^{\ast}(\gamma_1) \cdots \pi_r^{\ast}(\gamma_r) \in H^{\ast}((X,D)^r) \]
for $\pi_i : (X,D)^r \to (X,D) \cong X$ the forgetful morphism to the $i$-th point,
the marked relative invariants specialize to the usual invariants:

\begin{prop} \label{prop: comparision} 
If $\gamma = \pi_1^{\ast}(\gamma_1) \cdots \pi_r^{\ast}(\gamma_r) \in H^{\ast}((X,D)^r)$, then
\[
\left\langle \, \lambda\, \middle|\, \tau_{k_1} \cdots \tau_{k_r}( \gamma ) \right\rangle^{(X,D), \PT, \textup{marked}}_{\Gamma} 
=
\left\langle \, \lambda\, \middle| \, \tau_{k_1}(\gamma_1) \cdots \tau_{k_r}(\gamma_r) \right\rangle^{(X,D), \PT, \textup{std.-desc.}}_{\Gamma}
\]
\end{prop}

By the Proposition the marked relative invariants generalize the classical descendent invariants.
Hence we will drop "marked" and "std.-desc." from the notation.

\begin{proof}
Let $P_r = P_{\Gamma,r}(X,D)$ with universal target $\CX_r \to P_r$.
Consider the diagram
\[
\begin{tikzcd}
P_r \ar{rr}{q} \ar[swap]{dr}{\pi}&  & \CX_0 \times_{P_{0}} \cdots \times_{P_0} \CX_0 \ar{dl}{\rho} \\
& P_{0}. &
\end{tikzcd}
\]
where the map to the $i$-th factor is given by $P_r \xrightarrow{p_i} \CX_r \xrightarrow{\tilde{q}} \CX_0$.
By construction we have that $[P_r]^{\vir}=\pi^{\ast}[P_0]^{\vir}$ and hence that
\[ q_{\ast} [ P_r ]^{\vir} = \rho^{\ast} [P_0]^{\vir}. \]

We show that also the integrand appearing in the marked-relative invariants is pulled-back by $q$.
Clearly, $(\ev_{D}^{\rel})^{\ast}(\lambda)$ is pulled back from $P_0$, and similarly, we have
\[ \ev^{\ast}( \pi_i^{\ast}(\gamma_i)) = \tilde{\pi}^{\ast}\big( \pr_i^{\ast} \pi_X^{\ast}(\gamma_i) \big) \]
where $\pr_i : \CX_0 \times_{P_{0}} \cdots \times_{P_0} \CX_0 \to \CX_0$ is the projection to the $i$-th factor.
To deal with the Chern characters of the universal stable pair,
consider the commutative diagram
\[
\begin{tikzcd}
\CX_r \ar{r}{\tilde{q}} & \CX_0 \\
P_r \ar{u}{p_i} \ar{r}{\tilde{\pi}} & \CX_0 \times_{P_0} \cdots \times_{P_0} \CX_0 \ar{u}{\pr_i}.
\end{tikzcd}
\]
By construction we have that the universal stable pair $(\BF_r,\sigma_r)$ on $\CX_r \to P_r$ is
given by
\[ (\BF_r,\sigma_r) = \tilde{q}^{\ast}(\BF,\sigma). \]
Hence we find that
\[ p_i^{\ast}( \ch_{2+k_i}(\BF_r)) = \tilde{\pi}^{\ast} \pr_i^{\ast}( \ch_{2+k_i}(\BF) ). \]

In conclusion we find that:
\begin{align*}
& \left\langle \, \lambda\, \middle|\, \tau_{k_1} \cdots \tau_{k_r}( \gamma ) \right\rangle^{(X,D), \PT, \textup{marked}}_{\Gamma} \\
= & \int_{ [ P_{\Gamma,r}(X,D) ]^{\vir} }
\ev^{\mathrm{rel} \ast}_D(\lambda) \cdot
p_1^{\ast}(\ch_{k_1}(\BF)) \cdots p_{r}^{\ast}( \ch_{k_r}(\BF)) \cdot \ev^{\ast}(\gamma)  \\
=& 
\int_{ \CX_0 \times_{P_0} \cdots \times_{P_0} \CX_0}
\rho^{\ast}\left( [P_0]^{\vir}  (\ev_{D}^{\rel})^{\ast}(\lambda) \right)
\prod_{i=1}^{r} \pr_i^{\ast}( \ch_{2+k_i}(\BF) \pi_X^{\ast}(\gamma_i) )  \\
=& \int_{[P_0]^{\vir}} (\ev_{D}^{\rel})^{\ast}(\lambda) \rho_{r \ast}\left( \prod_{i=1}^{r} \pr_i^{\ast}( \ch_{2+k_i}(\BF) \pi_X^{\ast}(\gamma_i) ) \right) \\
=& \int_{[P_0]^{\vir}} (\ev_{D}^{\rel})^{\ast}(\lambda) \prod_{i=1}^{r} \tau_{k_i}(\gamma_i)
\end{align*}
which completes the proof.
\end{proof}

\subsection{Degeneration formulas} \label{subsec:degeneration formula PT}
Consider a simple degeneration $\epsilon : W \to B$ which degenerates a smooth fiber $X$ into the union $X_1 \cup_D X_2$
where $D$ is a smooth connected divisor,
\[ X \rightsquigarrow X_1 \cup_D X_2. \]
Let $\beta \in H_2(W,\BZ)$ be a curve class, and write
\[ \iota : X \to W, \quad \iota_{1} : X_1 \to W, \quad \iota_2 : X_2 \to W \]
for the inclusions.

Recall the moduli of $r$-ordered points $\epsilon_r : (\CW/B)^r \to B$,
and fix a class
\[ \gamma \in H^{\ast}( (\CW/B)^r ). \]
Since for $t \in B \setminus \{ 0 \}$ we have
\[ \epsilon_r^{-1}(t) = W_t^r \]
the class $\gamma$ induces a class on $W_t^r$ by restriction.
Over the origin we have the gluing map
\[ \xi = \bigsqcup_I \xi_I \colon \bigsqcup_{I} (X_1, D)^{I} \times (X_2, D)^J \to (\CW/B)^r_0. \]
where $I$ runs over all subsets of $\{1, \ldots, r\}$ with complement $J$.
The K\"unneth decomposition of the pullback to the $I$-th component will be denoted by
\[ \xi_I^{\ast}(\gamma) = \sum_{\ell} \delta_{I,\ell} \otimes \delta'_{I,\ell} 
\quad \in H^{\ast}((X_1, D)^{I} \times (X_2, D)^J). \]

\begin{prop}[Degeneration formula for marked relative invariants] 
\label{prop:degeneration formula PT}
For any $\beta \in H_2(W,\BZ)$,
\begin{multline*}
\sum_{\substack{ \beta' \in H_2(X,\BZ) \\ \iota_{\ast} \beta' = \beta}}
\left\langle \, \tau_{k_1} \cdots \tau_{k_r}( \gamma|_{X^r} ) \right\rangle^{X, \PT}_{n,\beta'} = \\
\sum_{\substack{ \beta_i \in H_2(X_i,\BZ) \\ \iota_{1 \ast} \beta_1 + \iota_{2 \ast} \beta_2 = \beta \\ \beta_1 \cdot D = \beta_2 \cdot D }}
\sum_{\substack{ n = n_1 + n_2 - (D \cdot \beta_1) \\  \{ 1, \ldots, r \} = I \sqcup J \\ \ell }}
\left\langle \Delta_1 \middle| \left( {\textstyle {\prod_{i \in I}} \tau_{k_i} }\right) (\delta_{I,\ell}) \right\rangle^{(X_1,D), \PT}_{n_1,\beta_1} 
\left\langle \Delta_2 \middle| \left( {\textstyle {\prod_{i \in J}} \tau_{k_j} } \right) (\delta'_{I,\ell}) \right\rangle^{(X_2,D), \PT}_{n_2,\beta_2}
\end{multline*}
where (following our convention of Section~\ref{sec:convention})
$\Delta_1, \Delta_2$ stands for summing over the K\"unneth decomposition of the class of the diagonal $\Delta_{D^{[\ell]}} \subset (D^{[\ell]})^2$
with $\ell = \beta_1 \cdot D$.
\end{prop}

\begin{proof} 
The argument is a straightforward generalization of \cite[Sec.6]{LiWu}.
For $\Gamma = (n,\beta)$,
let $\widetilde{\epsilon} : P_{\Gamma,r}(\CW/B) \to B$ be the moduli space of $r$-marked stable pairs on the degeneration $\epsilon : W \to B$,
which over $t \neq 0$ simply parametrizes $r$-marked stable pairs on $W_t$,
and over $0$ parametrizes relative stable pairs on expansions $W_0[k]$ together with $r$ marked points $p_1, \ldots, p_r \in W_0[k]$
not lying on the singular locus.
In particular, for $t \neq 0$ we have
\[ \widetilde{\epsilon}^{-1}(t) = \bigsqcup_{ \substack{ \beta' \in H_2(X_t,\BZ) \\ \iota_{t \ast} \beta' = \beta}} P_{n,\beta',r}(W_t), \]
and over $0 \in B$ we have a natural gluing morphism
\begin{equation} \label{xi def}
\xi: \bigsqcup_{\substack{ \beta_i \in H_2(X_i,\BZ) \\ \iota_{1 \ast} \beta_1 + \iota_{2 \ast} \beta_2 = \beta \\ \beta_1 \cdot D = \beta_2 \cdot D \\ n = n_1 + n_2 - (D \cdot \beta_1)\\ \{ 1, \ldots, r \} = I \sqcup J }}
P_{n_1,\beta_1,I}(X_1, D) \times_{D^{[\ell]}} P_{n_2, \beta_2,J}(X_2, D) \to P_{\Gamma,r}(\CW/B)_0.
\end{equation}

By pulling back the virtual class and splitting formulas of \cite{LiWu} via the forgetful morphism $P_{\Gamma,r}(\CW/B) \to P_{\Gamma,0}(\CW/B)$,
the moduli space $P_{\Gamma,r}(\CW/B)$ carries a virtual fundamental class $[ P_{\Gamma,r}(\CW/B) ]^{\vir}$ with the property
that for the inclusion $j_t : \{ t \} \to B$ we have over $t \neq 0$ the intersection
\[
j_t^{!} [ P_{\Gamma}(\CW/B) ]^{\vir} = \sum_{ \iota_{t \ast} \beta' = \beta} [ P_{n,\beta',r}(X_t) ]^{\vir},
\]
and over $t \in B$ we have
\begin{equation} \label{3wsef}
j_0^{!} [ P_{\Gamma,r}(\CW/B) ]^{\vir} =
\sum 
\xi_{I \ast} \Delta_{D^{[\ell]}}^{!}\Big( [ P_{n_1,\beta_1,I}(X_1, D) ]^{\vir} \times [P_{n_2, \beta_2,J}(X_2, D)]^{\vir} \Big),
\end{equation}
where the sum is over the same data as in \eqref{xi def} and we 
used the diagram:
\[
\begin{tikzcd}
P_{n_1,\beta_1,I}(X_1, D) \times_{D^{[\ell]}} P_{n_2, \beta_2,J}(X_2, D) \ar{r} \ar{d} & P_{n_1,\beta_1,I}(X_1, D) \times P_{n_2, \beta_2,J}(X_2, D) \ar{d}{\ev^{\rel} \times \ev^{\rel}} \\
D^{[\ell]} \ar{r}{\Delta_{D^{[\ell]}}}  & D^{[\ell]} \times D^{[\ell]}.
\end{tikzcd}
\]

We also have an evaluation morphism $\ev : P_{\Gamma,r}(\CW/B) \to (\CW/B)^r$ to the moduli space of $r$ ordered points on the fibers $\CW \to B$
which fits into the commutative diagram
\begin{equation} \label{ev diagram}
\begin{tikzcd}
P_{n_1,\beta_1,I}(X_1, D) \times_{D^{[\ell]}} P_{n_2, \beta_2,J}(X_2, D) \ar{r}{\xi} \ar{d}{\ev_I \times \ev_J} & P_{\Gamma,r}(\CW/B)_0 \ar{d}{\ev} \\
(X_1,D)^{I} \times (X_2, D)^{J} \ar{r}{\xi_I} & \left( (\CW/B)^{r} \right)_0.
\end{tikzcd}
\end{equation}

It remains to consider the Chern characters of the universal stable pairs. Let
\[ \CX_r \to P_{\Gamma,r}(\CW/B), \quad \CX_1 \to P_{n_1,\beta_1,I}(X_1,D), \quad \CX_2 \to P_{n_2,\beta_2,J}(X_2,D) \]
be the universal targets and consider the commutative diagram:
\[
\begin{tikzcd}
\pr_1^{\ast} \CX_1 \sqcup \pr_2^{\ast}(\CX_2) \ar{d} \ar{r}{\nu} & \CX_r \ar{d} \\
P_{n_1,\beta_1,I}(X_1, D) \times_{D^{[\ell]}} P_{n_2, \beta_2,J}(X_2, D) \ar{r}{\xi} & P_{\Gamma,r}(\CW/B).
\end{tikzcd}
\]
where $\nu$ is induced by the natural inclusion $\pr_1^{\ast}(\CX_1) \hookrightarrow \xi^{\ast}(\CX_r)$.
Then for $i \in I$ we have
\[ p_i \circ \xi = \nu \circ \big( (p_i \circ \pr_1) \times \pr_2 \big) \]
and since $\nu^{\ast}(\BF_r)|_{\pr_1^{\ast}\CX_1} = \pr_1^{\ast}(\BF_1)$ (with $\BF_r, \BF_1$ the universal sheafs on $\CX_r, \CX_1$)
we see that
\begin{equation} \xi^{\ast}( p_i^{\ast} \ch_{2+k}(\BF_r) ) = \pr_1^{\ast}( p_i^{\ast}( \ch_{2+k}(\BF_1) )). \label{chF pullback} \end{equation}
The case $j \in J$ is similar.

Consider now the class
\[
\widetilde{\epsilon}\left( 
p_1^{\ast}(\ch_{k_1}(\BF)) \cdots p_{r}^{\ast}( \ch_{k_r}(\BF)) \cdot \ev^{\ast}(\gamma)
\cap [ P_{\Gamma,r}(\CW/B) ]^{\vir}
\right) \in A^{\ast}(B).
\]
Restricting this class over $t \in B$ yields the left hand side of Proposition~\ref{prop:degeneration formula PT}.
Restricting this class over $0 \in B$ and using \eqref{3wsef},
the commutativity of \eqref{ev diagram},
and \eqref{chF pullback}, then yields precisely the right hand side. This completes the proof.
\end{proof}

\subsection{Rubber stable pairs}
Consider the projective bundle,
\[ \p = \p(N_{D/X} \oplus \CO_{D}), \]
and the rubber target
\[ (\p, D_{0,\infty})^{\sim}, \quad D_{0,\infty} = D_0 \sqcup D_{\infty}. \]

An $r$-marked rubber stable pair on $(\p, D_{0,\infty})$ is a tuple $(F,\sigma, p_1, \ldots, p_r)$ where
$F$ is a pure 1-dimensional sheaf on a chain $\p_{\ell}$, the element $\sigma \in H^0(F)$ is a section with zero-dimensional cokernel,
and $p_1, \ldots, p_r \in \p_{\ell}$ are points satisfying the following properties:
\begin{itemize}[itemsep=0pt]
\item $F$ is normal to the relative divisors $D_{0}[\ell]$, $D_{\infty}[\ell]$ and the singular locus,
\item $p_1, \ldots, p_r \in \p_{\ell}$ do not lie on the relative divisor or the singular locus,
\item there exist only finitely many $\varphi \in \BG_m^{\ell}$ acting fiberwise on $\p_{\ell}$ such that
$\varphi^{\ast}(F,s) \cong (F,s)$ and $\varphi(p_i) = p_i$.
\end{itemize}
And we have the rubber condition: Two such tuples $(F,\sigma,p_1, \ldots, p_r)$, $(F', \sigma',p'_1, \ldots, p'_r)$  are considered isomorphic if they differ
by an element $\varphi \in \BG_m^r$.

Given $\Gamma = (n,(\alpha,d))$ for a curve class $\alpha \in H_2(D,\BZ)$, $d \geq 0$, and $n \in \BZ$ we let
\[ P_{\Gamma,r}(\p, D_{0,\infty})^{\sim} \]
be the moduli space of rubber stable pairs $(F,\sigma)$ on $(\p, D_{0,\infty})^{\sim}$
satisfying the numerical conditions
\begin{equation} \chi(F) = n, \quad \pi_{D \ast} \ch_2(F) = \alpha, \quad \ell(F \cap D_{\infty}[\ell]) = d, \label{3sdfYYY} \end{equation} 
where $\pi_D : \p_{\ell} \to D$ is the projection.
In the case $r=0$ we recover the usual moduli space of rubber stable pairs \cite{LiWu}.

Since on $\p$ we have that
\[ D_{\infty} = D_{0} + \pi_D^{\ast}( c_1(N_{D/X}) )  \in A^{\ast}(\p) \]
we see that \eqref{3sdfYYY} implies that
\[ d_0 := \ell( F \cap D_{0}[\ell] ) = \ell( F \cap D_{\infty} )- \int_{\alpha} c_1(N_{D/X}) = d - \int_{\alpha} c_1(N_{D/X}). \]

The moduli space $P_{\Gamma,r}(\p, D_{0,\infty})^{\sim}$ carries a virtual fundamental class,
and there are relative evaluation maps both over the divisor $D_0$, $D_{\infty}$
\begin{gather*} 
\ev^{\rel}_{D_0} : P_{\Gamma,r}(\p, D_{0,\infty})^{\sim} \to D_0^{[d_0]} \\
\ev^{\rel}_{D_{\infty}} : P_{\Gamma,r}(\p, D_{0,\infty})^{\sim} \to D_{\infty}^{[d]}
\end{gather*} 
given by send $(F,\sigma, p_1, \ldots, p_{\infty})$ on some $\p_{\ell}$ to $F \cap D_0[\ell]$ or $F \cap D_{\infty}[\ell]$ respectively.
We also have interior evaluation maps
\[
\ev : P_{\Gamma,r}(\p, D_{0,\infty})^{\sim} \to (\p, D_{0,\infty})^{r, \sim}
\]
where the moduli space $(\p, D_{0,\infty})^{r, \sim}$ was introduced in Section~\ref{subsec:moduli of points rubber}. Let
\[ \CX_r \to P_{\Gamma,r}(\p, D_{0,\infty})^{\sim} \]
be the universal target over the moduli space, let $(\BF,\sigma)$ be the universal stable pair,
and let 
\[ p_1, \ldots, p_r : P_{\Gamma,r}(\p, D_{0,\infty})^{\sim} \to \CX_r \]
be the universal sections.
(By definition, the tuple $(\BF, \sigma, p_1, \ldots, p_r)$ is canonical only up the action of $\BG_m^r$ but we can always make one choice of it.
The invariant below does not depend on the choice.)
Let $\gamma \in H^{\ast}( (\p, D_{0,\infty})^{r, \sim} )$, as well as $\lambda \in H^{\ast}(S^{[d_0]})$ and $\mu \in H^{\ast}(S^{[d]})$
be cohomology classes for some $a,b \geq 0$.
We define the marked-relative rubber invariants:
\begin{multline*}
\big\langle \, \lambda\, , \mu \, \big|\, \tau_{k_1} \cdots \tau_{k_r}(\gamma) \big\rangle^{(\p,D_{0,\infty}), \PT, \sim}_{\Gamma} \\
=
\int_{ [ P_{\Gamma,r}(\p,D_{0,\infty})^{\sim} ]^{\vir} }
(\ev^{\rel}_{D_0})^{\ast}(\lambda) (\ev^{\rel}_{D_\infty})^{\ast}(\mu) \cdot
p_1^{\ast}(\ch_{k_1}(\BF)) \cdots p_{r}^{\ast}( \ch_{k_r}(\BF)) \cdot \ev^{\ast}(\gamma).
\end{multline*}

\subsection{Rigidification}
We can compare the rubber invariants with the (usual) relative stable pair invariants of the pair $(\p, D_{0,\infty})$.
We will use the identification
\[ H_2(\p, \BZ) = H_2(D,\BZ) \oplus \BZ \]
given by sending $(\alpha,d) \in H_2(D,\BZ) \oplus \BZ$ to
\[ \iota_{D_0 \ast}(\alpha) + d F \]
where $F$ is the class of a fiber of $\pi_D : \p \to D$,
and $\iota_{D_0} : D_0 \to \p$ is the inclusion.
Note that this matches nicely the convention for the rubber stable pairs:
\[ (\alpha,d) \cdot D_{\infty} = d, \quad (\alpha,d) \cdot D_0 = d - \int_{\alpha} c_1(N_{D/X}). \]

Given $\Gamma = (n, (\alpha,d))$ we have a natural morphism
\[ f : P_{\Gamma,r}(\p, D_{0,\infty}) \to P_{\Gamma,r}(\p, D_{0,\infty})^{\sim} \]
that for a relative stable pair $(F,\sigma, p_i)$ on $(\p[k_1, k_2], D_{0,\infty})$ does the following:
If the pair is not fixed by the natural $\BG_m$-action induced from the fiberwise scaling of $\p$,
we simply view it as a stable pair on the chain $\p_{\ell} = \p[k_1, k_2]$ (in the process hence forgetting the distinguished component,
and identifying the stable pair with all its $\BG_m$-translates).
If the pair is $\BG_m$-fixed, 
the stable pair $(F,\sigma, p_i)$ is isomorphic to the pullback of a stable pair $(F', \sigma', p_i')$ under the 
map $c : \p[k_1, k_2] \to \p_{k_1+k_2}$ that contracts the central fiber.
The map then sends $(F,\sigma, p_i)$ to $(F',\sigma',p_i')$.

We have the following comparision of virtual classes:
\begin{prop}[Rigidification] \label{prop:rigidification}
\begin{align*}
[ P_{\Gamma,r}(\p, D_{0,\infty})^{\sim} ]^{\vir} & = f_{\ast} \left( \ev_1^{\ast}(D_0) [ P_{\Gamma,r}(\p, D_{0,\infty})^{\sim} ]^{\vir} \right) \\
& = f_{\ast} \left( \ev_1^{\ast}(D_\infty) [ P_{\Gamma,r}(\p, D_{0,\infty})^{\sim} ]^{\vir} \right)
\end{align*}
\end{prop} 
\begin{proof}
The proof is identical to \cite[Sec.1.5.3]{MP}. We sketch the idea for convenience.
The map $f$ is $\BG_m$-equivariant with respect to the induced action by the fiberwise scaling action of $\BG_m$ on $\p$,
and the trivial action on the target.
One computes the pushforward of $\ev_1^{\ast}(D_0) [ P_{\Gamma,r}(\p, D_{0,\infty})^{\sim} ]^{\vir}$ by virtual localization \cite{GP}
(the other case is similar).
The $\BG_m$-fixed locus parametrizes relative stable pairs with expansions at $D_0$ and/or $D_\infty$.
If we have expansion at both $D_0$ and $D_\infty$, then the virtual dimension of the fixed virtual class
is one less than the dimension of
\[  f_{\ast} \left( \ev_1^{\ast}(D_0) [ P_{\Gamma,r}(\p, D_{0,\infty})^{\sim} ]^{\vir} \right) . \]
Hence it does not contribute.
If there is no expansion over $D_0$, the marking $p_1$ has to lie in the expansion at $D_\infty$,
but then the restriction of $\ev_1^{\ast}(D_0)$ to this component vanishes.
One finds the remaining contribution coming from the expansion over $D_{0}$ to be
\[ \frac{t - c_1(N_{D/X})}{t - \psi_{D_{\infty}}} \cap [ P_{\Gamma,r}(\p, D_{0,\infty})^{\sim} ]^{\vir} \]
where $\psi_{D_{\infty}}$ is the cotangent line class\footnote{The class $\psi_{D_{\infty}}$ is the pull-back of the first chern class of the line bundle 
$\BL_{D_{\infty}}$ on $\CT^{\rub}_{(\p, D_{0} \sqcup D_{\infty})}$. Note that the line bundle $\BL_{D}$ on $\CT_{(X,D)}$
restricts to $\BL_{D_0} \cong \BL_{D_{\infty}}^{\vee}$ on $\CT^{\rub}_{(\p, D_{0} \sqcup D_{\infty})}$.
The term in the denominator is the contribution from the virtual normal bundle which is here the product of the two normal bundles at the divisor where the stable pairs splits.}
of the relative divisor $D_{\infty}$ in the rubber geometry $(\p, D_{0,\infty})^{\sim}$,
and $t$ is the tangent weight of the $\BG_m$-action on the fiber of $N_{D_0/\p}|_{x}$ for some $x \in D_0$
(in particular, we have the equivariant Chern class $c_1(\CO_{D_0}(D_0)) = c_1(N_{D_0/\p}) + t$).
The result now follows from pushforward.
\end{proof}

\begin{cor} \label{cor:rigidification}
Let $r \geq 1$. For any $\gamma \in (\p,D_{0,\infty})^{r, \sim}$ we have
\[
\big\langle \, \lambda\, , \mu \, \big|\, \tau_{k_1} \cdots \tau_{k_r}(\gamma) \big\rangle^{(\p,D_{0,\infty}), \PT, \sim}_{\Gamma} \\
= 
\big\langle \, \lambda\, , \mu \, \big|\, \tau_{k_1} \cdots \tau_{k_r}\big(\pi_1^{\ast}(D_0) f^{\ast}(\gamma) \big) \big\rangle^{(\p,D_{0,\infty}), \PT}_{\Gamma}
\]
where $f : (\p,D_{0,\infty})^r \to (\p,D_{0,\infty})^{r, \sim}$ is the natural morphism.
\end{cor}
\begin{proof}
We have the commutative diagram
\[
\begin{tikzcd}
P_{\Gamma,r}(\p, D_{0,\infty}) \ar{d} \ar{r}{\ev} \ar[bend left]{rr}{\ev_1} & (\p,D_{0,\infty})^r \ar{d}{f} \ar{r}{\pi_1} & \p \\
P_{\Gamma,r}(\p, D_{0,\infty})^{\sim} \ar{r} & (\p,D_{0,\infty})^{r, \sim}
\end{tikzcd}
\]
from which the claim follows by Proposition~\ref{prop:rigidification} and the definition.
\end{proof}

\subsection{Splitting formula}
For any decomposition $\{1, \ldots, r \} = I \sqcup J$ consider the gluing morphism
\[ \xi_{I} : (\p, D_{0,\infty})^{I,\sim} \times (X,D)^{J} \to (X,D)^r, \]
where we suppress $J$ from the notation because it is determined by $I$ via $J = I^{c}$.

Let $\pi_{12} : (X,D)^r \to (X,D)^2$ the morphism that forgets all but the first two points,
and recall that we write
\[ \Delta_{12}^{\rel} = \pi_{12}^{\ast}( \Delta^{\rel} ) \]
for the relative diagonal, which is the class of the locus where the first two points coincide.
We also have the absolute diagonal,
\[ \Delta_{12} = (\pi_1 \times \pi_2)^{\ast}( \Delta ) \in H^{\ast}((X,D)^r) \]
where $\pi_i : (X,D)^r \to (X,D)^1 \to X$ is the projection to the $i$-th factor.

We prove here the following splitting formula (which is an analogue of \cite[Thm.3.10]{ABPZ}):

\begin{prop} 
Let $d = \beta \cdot D$. We have
\label{prop:splitting relative diagonal}
\begin{multline*}
\big\langle \, \lambda \, \big| \, \tau_{k_1} \cdots \tau_{k_r}( \Delta_{12}^{\rel} \cdot \gamma) \big\rangle^{(X,D), \PT}_{n,\beta}
=
\big\langle \, \lambda \, \big| \, \tau_{k_1} \cdots \tau_{k_r}( \Delta_{12} \cdot \gamma) \big\rangle^{(X,D), \PT}_{n,\beta} - \\
\sum_{\substack{\iota_{\ast} \alpha + \beta' = \beta \\ n_1 + n_2 = n + d_0}}
\sum_{ \substack{ \{ 1, \ldots, r \} = I \sqcup J \\ \text{with } 1,2 \in I \\ \ell }}
\big\langle \Delta_1, \lambda \big| \left( {\textstyle \prod_{i \in I} \tau_{k_i} } \right) ( \Delta_{D} \delta_{I,\ell}) \big\rangle^{(\p,D_{0, \infty}), \PT, \sim}_{n_1,(\alpha,d)}
\big\langle \Delta_2 \big| \left( {\textstyle \prod_{i \in J} \tau_{k_i} } \right)( \delta'_{I,\ell} ) \big\rangle^{(X,D), \PT}_{n_2,\beta'} 
\end{multline*}
where 
$\Delta_1, \Delta_2$ runs over the K\"unneth decomposition of the diagonal $\Delta_{D^{[d_0]}}$ in $( D^{[d_0]} )^2$
with $d_0 = d - \int_{\alpha} c_1(N_{D/X})$ and we used the K\"unneth decomposition
\[ \xi_{I}^{\ast}( \gamma ) = \sum_{\ell} \delta_{I,\ell} \otimes \delta'_{I,\ell}  \quad \in H^{\ast}( (\p, D_{0,\infty})^{I, \sim} \times (X,D)^{J} ). \]
\end{prop}

\begin{proof}
By Lemma~\ref{lemma:rel splitting} in $H^{\ast}((X,D)^2)$ we have 
\[ \Delta_{(X,D)}^{\text{rel}} = (\pi_1 \times \pi_2)^{\ast}( \Delta_X ) - j_{\ast}(g^{\ast}(\Delta_D)), \]
where $j : \p(N) \to (X,D)^2$ is the natural inclusion and $g : \p(N) \to D \times D$ is the projection.

Consider the space
\[
R_{I} P_{\Gamma,r}(X,D) =
\bigsqcup_{\substack{\iota_{\ast} \alpha + \beta' = \beta \\ n_1 + n_2 = n + \beta \cdot D}}
P_{(n_1,(\alpha,d)),I}(\p, D_{0,\infty})^{\sim} \times P_{(n_2,\beta'),J}(X,D).
\]
and its virtual class, defined by
\[
[ R_{I} P_{\Gamma,r}(X,D) ]^{\vir} 
=
\sum_{\substack{\iota_{\ast} \alpha + \beta' = \beta \\ n_1 + n_2 = n + \beta \cdot D}}
\Delta_{D^{[\ell]}}^{!}\left( [ P_{(n_1,(\alpha,d)),I}(\p, D_{0,\infty})^{\sim} ] \otimes [ P_{(n_2,\beta'),J}(X,D) ]^{\vir} \right).
\]

We then have a commutative diagram
\[
\begin{tikzcd}
\bigsqcup_{ \substack{ \{ 1, \ldots, r \} = I \sqcup J \\ 1,2 \in I }} R_{I} P_{\Gamma,r}(X,D) \ar{r}{\widetilde{\xi}} \ar{d}{\ev_I \times \ev_J} & \widetilde{W} \ar{r} \ar{d} & P_{\Gamma,r}(X,D) \ar{d} \\
\bigsqcup_{ \substack{ \{ 1, \ldots, r \} = I \sqcup J \\ 1,2 \in I }} (\p, D_{0,\infty})^{I, \sim} \times (X,D)^{J} \ar{r}{\xi = \sqcup \xi_{I}} & W \ar{r} \ar{d} & (X,D)^r \ar{d}{\pi_{12}} \\
& \p(N) \ar{r}{j} & (X,D)^2.
\end{tikzcd}
\]
where the square on the right are fiber (and so define $W$ and $\widetilde{W}$)
and $\xi$ is birational (by Lemma~\ref{lemma:2diagonal}).
By the arguments of \cite{LiWu} (compare also \cite[Thm.3.9]{ABPZ}) we have
\[
j^{!} [ P_{\Gamma,r}(X,D) ]
=
\sum_{\substack{I,J \\ 1,2 \in I}} \widetilde{\xi}_{\ast} [ R_{I} P_{\Gamma,r}(X,D) ]^{\vir}.
\]

We obtain:
\begin{align*}
& \big\langle \, \lambda \, \big| \, \tau_{k_1} \cdots \tau_{k_r}( (j_{\ast} g^{\ast}\Delta_D )_{12} \cdot \gamma) \big\rangle^{(X,D), \PT}_{n,\beta}  \\
& =
\int_{[ P_{\Gamma,r}(X,D) ]^{\vir} } \ev_{\rel}^{\ast}(\lambda) \ev^{\ast}(\gamma) \ev_{12}^{\ast}( j_{\ast} g^{\ast}(\Delta_D))  \cdot \prod_i p_i^{\ast}(\ch_{k_i}(\BF)) \\
& =
\sum_{ \substack{ I,J \\ 1,2 \in I }}
\int_{ [ R_{I} P_{\Gamma,r}(X,D) ]^{\vir} }
\ev_{D_{\infty}}^{\ast}(\lambda) ( \ev_I \times \ev_J )^{\ast}( \xi^{\ast}(\gamma)) \ev_{12}^{\ast}(\Delta_D) \cdot \prod_i p_i^{\ast}(\ch_{k_i}(\BF)).
\end{align*}
Using a similar observation as \eqref{chF pullback} this completes the proof.
\end{proof}

\section{Relative Gromov-Witten theory} \label{sec:GW theory}
Let $X$ be a smooth projective variety and let $D \subset X$ be a smooth 
connected divisor.
\subsection{Moduli space}
Let $\beta \in H_2(X,\BZ)$ and let
$\vec{\lambda} = (\lambda_1, \lambda_2, \ldots, \lambda_{\ell}$ be an ordered partition of size and length
\[ |\vec{\lambda}| := \sum_i \lambda_i = \beta \cdot D, \quad \ell(\lambda) := \ell. \]
We consider the moduli space introduced by Jun Li \cite{Li1,Li2}
\[ \Mbar_{g,r,\beta}((X,D), \vec{\lambda}) \]
which parametrizes $r$-pointed genus $g$ degree $\beta$ relative stable maps from connected curves to the pair $(X,D)$
with ordered ramification profile $\vec{\lambda}$ along the divisor $D$.
By definition, an element of the moduli space is a map
\[ f : C \to X[k] \]
to an expansion of $(X,D)$ 
such that (i) no component is mapped entirely into the singular locus,
(ii) $f$ is predeformable, (iii) the relative multiplicities with divisor $D$ are as specified (the intersection points are marked),
and (iv) has finite automorphism. We refer to \cite{LiICTP} or the recent \cite{ABPZ} for an introduction.
The degree
of the map $f$ is fixed to be $p_{\ast} f_{\ast}[C] = \beta$
where $p : X[k] \to X$ is the canonical map that contracts the expansion.

The moduli space has relative evaluation maps
\[ \ev_{D,i}^{\rel} : \Mbar_{g,r,\beta}((X,D), \vec{\lambda}) \to D \]
which send a stable map to the $i$-th intersection point with the divisor $D$ (according to the fixed ordering).
We also have an interior evaluation map:
\[
\ev : \Mbar_{g,r,\beta}((X,D), \vec{\lambda}) \to (X,D)^r
\]
Given a subset $I \subset \{ 1, \ldots, r \}$ we write $\ev_I = \pi_I \circ \ev$ where $\pi_I : (X,D)^r \to (X,D)^I$
is the morphism which forgets all points except those labeled by $I$.

\subsection{Cohomology weighted partitions}
A $H^{\ast}(D)$-weighted partition $\lambda$ (or simply cohomology-weighted partition if $D$ is clear from context) is an ordered list of pairs
\begin{equation} \label{weighted partition} \big( (\lambda_1, \delta_1) , \ldots, (\lambda_{\ell}, \delta_{\ell} ) \big), \quad \delta_i \in H^{\ast}(D), \quad \lambda_i \geq 1  \end{equation}
such that $\vec{\lambda} = (\lambda_1, \lambda_2, \ldots, \lambda_{\ell})$ is a partition
(called the partition underlying $\lambda$).

While the $\delta_i$ can be arbitrary cohomology classes on $D$,
we often take them to be elements of a fixed basis $\CB$ of $H^{\ast}(D)$.
In this case we also talk of a $\CB$-weighted partition.
Given a $\CB$-weighted partition $\lambda$, the automorphism group $\Aut(\lambda)$ consists of the permutation symmetries of $\lambda$,
i.e. those $\sigma \in S_r$ such that $\lambda^{\sigma} = \lambda$.

\subsection{Gromov-Witten invariants} \label{subsec:defn GW invariants}
Given a $H^{\ast}(D)$-weighted partition $\lambda$ and a class $\gamma \in H^{\ast}((X,D)^r)$
we define relative Gromov-Witten invariants by
integration over the virtual fundamental class of the moduli space:
\begin{equation} \label{GWbracket}
\big\langle \, \lambda \, \big| \, \tau_{k_1} \cdots \tau_{k_r}(\gamma) \big\rangle^{(X,D), \GW}_{g,\beta} \\
=
\int_{ [ \Mbar_{g,r,\beta}((X,D), \vec{\lambda}) ]^{\vir} }
\ev^{\ast}(\gamma) \prod_{i=1}^{r} \psi_i^{k_i} \prod_{i=1}^{\ell(\lambda)} \ev^{\text{rel}}_{D,i}( \delta_{i} ),
\end{equation}
where $\psi_i$ are the cotangent line classes at the interior markings.

The discussion also applies
when we allow the source curve of our relative stable map to be disconnected.
More precisely, we let 
\[ \Mbar_{g,r,\beta}^{\bullet}((X,D), \vec{\lambda}) \]
denote the moduli space of relative stable maps to $(X,D)$ as above
except that we allow disconnected domain curves
subject to the following condition:

($\bullet$) For any relative stable map $f : C \to X[k]$ to an expansion of $(X,D)$
the stable map $f$ has non-zero degree on every of the connected components of its domain.

We define Gromov-Witten invariants in the disconnected case parallel as in \eqref{GWbracket}.
The brackets on the left hand side will be denoted with a supscript $\bullet$, as in $\langle .. \rangle^{(X,D), \GW \bullet}$.

\subsection{Degeneration formula}
We state for completeness the degeneration formula in Gromov-Witten theory.
We use the same notation as in Section \ref{subsec:degeneration formula PT}.
\begin{prop} [\cite{Li1,Li2}] \label{prop:degeneration formula GW}
For any $\beta \in H_2(W,\BZ)$,
\begin{multline*}
\sum_{\substack{ \beta' \in H_2(X,\BZ) \\ \iota_{\ast} \beta' = \beta}}
\left\langle \, \tau_{k_1} \cdots \tau_{k_r}( \gamma|_{X^r} ) \right\rangle^{X, \GW,\bullet}_{g,\beta'} = 
\sum_{\substack{ \beta_i \in H_2(X_i,\BZ) \\ \iota_{1 \ast} \beta_1 + \iota_{2 \ast} \beta_2 = \beta \\ \beta_1 \cdot D = \beta_2 \cdot D }}
\sum_{\substack{ g_1+g_2=g + 1 - \ell(\mu) \\  \{ 1, \ldots, r \} = I \sqcup J }}
\sum_{\mu, \ell} \\
\frac{\prod_i \mu_i}{|\Aut(\mu)|}
\left\langle \mu \middle| \left( {\textstyle {\prod_{i \in I}} \tau_{k_i} }\right) (\delta_{I,\ell}) \right\rangle^{(X_1,D), \GW, \bullet}_{g_1,\beta_1} 
\left\langle \mu^{\vee} \middle| \left( {\textstyle {\prod_{i \in J}} \tau_{k_j} } \right) (\delta'_{I,\ell}) \right\rangle^{(X_2,D), \GW, \bullet}_{g_2,\beta_2}
\end{multline*}
where $\mu$ runs over all cohomology weigted partitions $\mu = \{ (\mu_i, \gamma_{s_i}) \}$ of size $\beta_1 \cdot D$
with weights from a fixed basis $\CB=\{ \gamma_i \}$ of $H^{\ast}(D)$,
and we let $\mu^{\vee} = \{ (\mu_i, \gamma_{s_i}^{\vee}) \}$ be the dual partition
with weights from the basis $\{ \gamma_i^{\vee} \}$ which is dual to $\{ \gamma_i \}$.
Moreover, we used the K\"unneth decomposition
\[ \xi_I^{\ast}(\gamma) = \sum_{\ell} \delta_{I,\ell} \otimes \delta'_{I,\ell} 
\quad \in H^{\ast}((X_1, D)^{I} \times (X_2, D)^J). \]
\end{prop}

\subsection{Rubber moduli space}
Recall the projective bundle $\p = \p(N_{D/X} \oplus \CO_D)$ with sections $D_{0}, D_{\infty} \subset \p$.
Let
\[ \Mbar_{g,r,\alpha}^{\sim}((\p,D_0 \sqcup D_{\infty}), \vec{\lambda}, \vec{\mu}) \]
be the moduli space of genus $g$ degree $\alpha \in H_2(D,\BZ)$ rubber stable maps with target $(\p, D_{0,\infty})$
with oredered ramification profiles $\vec{\lambda}$, $\vec{\mu}$ over the divisors $D_0$ and $D_{\infty}$ respectively.
Elements of the moduli space are maps $f : C \to \p_{l}$
satisfying the usual list of conditions
(finite automorphism, predeformability, no components mapping entirely mapped to the singular fibers, relative multiplicities as specified),
and where two maps are considered to be isomorphic if they differ by an action of the natural scaling automorphism $\BG_m^{l}$ of $\p_{l}$.
The degree of the map is fixed to be $\pi_{D \ast} f_{\ast} [C] = \alpha$ where $\pi_D : \p_l \to D$ is the natural projection.
In our definition above the source curve is assumed to be connected.
If we allow disconnected domains subject to condition ($\bullet$), we decorate the moduli space and invariants with the supscript $\bullet$.

We have evaluation maps at the relative markings over both $D_0$ and $D_{\infty}$,
\begin{gather*}
 \ev^{\rel}_{D_0,i} : \Mbar_{g,r,\alpha}^{\sim}((\p,D_0 \sqcup D_{\infty}), \vec{\lambda}, \vec{\mu}) \to D_0, \ i=1,\ldots, \ell(\vec{\lambda}) \\
 \ev^{\rel}_{D_{\infty},i} : \Mbar_{g,r,\alpha}^{\sim}((\p,D_0 \sqcup D_{\infty}), \vec{\lambda}, \vec{\mu}) \to D_{\infty}, \ i=1,\ldots, \ell(\vec{\mu})
\end{gather*}
and an interior evaluation map
\[ \ev:  \Mbar_{g,r,\alpha}^{\sim}((\p,D_0 \sqcup D_{\infty}), \vec{\lambda}, \vec{\mu}) \to (\p,D_{0,\infty})^{r,\sim}. \]
Given $H^{\ast}(D)$-weighted partitions $\lambda = (\lambda_i, \delta_i)_{i=1}^{\ell(\lambda)}$, $\mu = (\mu_i, \delta'_i)_{i=1}^{\ell(\mu)}$
and $\gamma \in H^{\ast}((\p,D_{0,\infty})^r)$ we define:
\begin{multline*}
\big\langle \, \lambda, \mu \, \big| \, \tau_{k_1} \cdots \tau_{k_r}(\gamma) \big\rangle^{(\p,D_{0,\infty}), \GW, \sim}_{g,\alpha} \\
=
\int_{ [ \Mbar_{g,r,\alpha}^{\sim}((\p,D_0 \sqcup D_{\infty}), \vec{\lambda}, \vec{\mu}) ]^{\vir} }
\ev^{\ast}(\gamma) \prod_{i=1}^{r} \psi_i^{k_i} \prod_{i=1}^{\ell(\lambda)} (\ev_{D_0,i}^{\rel})^{\ast}(\delta_i) \prod_{i=1}^{\ell(\mu)} (\ev_{D_{\infty},i}^{\rel})^{\ast}(\delta_i').
\end{multline*}

There is a rigidification statement parallel to Proposition~\ref{prop:rigidification},
see \cite[Sec.1.5.3]{MP}.

\subsection{Splitting formula}
We state the splitting formulas we will need.
Let $\iota : D \to X$ denote the inclusion.
\begin{prop} \label{prop:splitting relative diagonal GW}
Let $d = \beta \cdot D$. We have
\begin{multline*}
\big\langle \, \lambda \, \big| \, \tau_{k_1} \cdots \tau_{k_r}( \Delta_{12}^{\rel} \cdot \gamma) \big\rangle^{(X,D), \GW, \bullet}_{g,\beta}
=
\big\langle \, \lambda \, \big| \, \tau_{k_1} \cdots \tau_{k_r}( \Delta_{12} \cdot \gamma) \big\rangle^{(X,D), \GW, \bullet}_{g,\beta} \\
- \sum_{\substack{\iota_{\ast} \alpha + \beta' = \beta \\ g_1+g_2=g + 1 - \ell(\mu)}}
\sum_{ \substack{ \{ 1, \ldots, r \} = I \sqcup J \\ \text{with } 1,2 \in I }}
\sum_{\mu, \ell}\frac{\prod_i \mu_i}{|\Aut(\mu)|}
\big\langle \Delta_1, \lambda \big| \left( {\textstyle \prod_{i \in I} \tau_{k_i} } \right) ( \Delta_{D} \delta_{I,\ell}) \big\rangle^{(\p,D_{0, \infty}), \GW, \bullet, \sim}_{g_1,(\alpha,d)} \\
\big\langle \Delta_2 \big| \left( {\textstyle \prod_{i \in J} \tau_{k_i} } \right)( \delta'_{I,\ell} ) \big\rangle^{(X,D), \GW, \bullet}_{g_2,\beta'}
\end{multline*}
where $\mu$ runs over all cohomology weigted partitions $\mu = \{ (\mu_i, \gamma_{s_i}) \}$ of size $d_0 = d - \int_{\alpha} c_1(N_{D/X})$,
with weights from a fixed basis $\CB=\{ \gamma_i \}$ of $H^{\ast}(D)$
and we used the K\"unneth decomposition
\[ \xi_{I}^{\ast}( \gamma ) = \sum_{\ell} \delta_{I,\ell} \otimes \delta'_{I,\ell}  \quad \in H^{\ast}( (\p, D_{0,\infty})^{I, \sim} \times (X,D)^{J} ) \]
where $\xi_{I} : (\p, D_{0,\infty})^{I,\sim} \times (X,D)^{J} \to (X,D)^r$ is the gluing morphism.
\end{prop}
\begin{proof}
This is a special case of \cite[Theorem 3.10]{ABPZ}.
Essentially, the argument is to start with Lemma~\ref{lemma:rel splitting}
and then one observes that $\p(N) \subset (X,D)^2$ 
is the pullback of the Cartier divisor on the stack of target degenerations $\CT$
which parametrizes non-trivial expansions.
The one uses the splitting of the virtual class of \cite{Li1,Li2}.
\end{proof}

\section{GW/PT correspondences} \label{sec:GW/PT corr}
Let $X$ be a smooth projective threefold and let $D \subset X$ be a smooth connected divisor,
which is hence a smooth surface.
We discuss here GW/PT correspondence for the pair $(X,D)$
and in particular consider the case of marked relative invariants.

\subsection{Cohomology weighted partitions}
Consider a $H^{\ast}(D)$ weighted partition
\[ \lambda = \big( (\lambda_1, \delta_1), \ldots, (\lambda_{\ell(\lambda)}, \delta_{\ell(\lambda)}) \big), \quad \delta_i \in H^{\ast}(D) \]
We can associate to $\lambda$ a cohomology class in $H^{\ast}(D^{[|\lambda|]})$ as follows.
For $i > 0 $ let
\[ \Fq_i(\alpha) : H^{\ast}(D^{[n]}) \to H^{\ast}(D^{[n+i]}) \]
be the $i$-th Nakajima creation operator with cohomology weight $\alpha \in H^{\ast}(D)$, see \cite{Nak}
(we refer to Example~\ref{example:Nakajima} for the convention that we follow here).

\begin{defn}
The class in $H^{\ast}(D^{[|\lambda|]})$ associated to $\lambda$ is
\begin{equation} \lambda = \frac{1}{\prod_i \mu_i} \prod_{i} \Fq_i(\delta_i) \vacuum. \label{identification} \end{equation}
\end{defn}

A given class in the Hilbert scheme can have several representations as a cohomology weighted partitions.
Nevertheless, we will also write $\lambda$ for the associated cohomology class,
because the formulas where it will appear will only depend on the class.

\begin{lemma} In $H^{\ast}(D^{[n]} \times D^{[n]})$ we have the K\"unneth decomposition of the diagonal
\begin{equation} \label{Kunneth diagonal}
\Delta_{D^{[n]}} = \sum_{\mu} (-1)^{n - \ell(\mu)}
\frac{\prod_i \mu_i}{|\Aut(\mu)|} \cdot \mu \boxtimes \mu^{\vee}.
\end{equation}
where $\mu$ runs over all cohomology weigted partitions $\mu = \{ (\mu_i, \gamma_{s_i}) \}$ with weights from a fixed basis $\CB = (\gamma_1, \ldots, \gamma_{b})$ of $H^{\ast}(D)$,
and $\mu^{\vee} = \{ (\eta_i, \gamma_{s_i}^{\vee}) \}$ is the dual partition
with weights from the basis $( \gamma_i^{\vee} )$ which is dual to $\CB$.
Moreover we have used \eqref{identification} to associate a cohomology class to a cohomology weighted partition.
\end{lemma}
\begin{proof}
For $\CB$-weighted partitions $\mu, \nu$ one has
\[ \int_{D^{{n]}}} \mu \cdot \nu^{\vee} = \delta_{\mu \nu} (-1)^{n+\ell(\mu)} \frac{|\Aut(\mu)|}{\prod_i \mu_i}. \]
\end{proof}

\subsection{Partition functions} \label{sec:partition functions}
Let $\beta \in H_2(X,\BZ)$ be a curve class and define the integer
\[ \d_{\beta} = \int_{\beta} c_1(T_X). \]
Let $\lambda$ be a $H^{\ast}(D)$-weighted partitions of size $\beta \cdot D$, let $\gamma \in H^{\ast}(X,D)^r$, and let $k_1, \ldots, k_r \geq 0$.

\begin{defn}
The partition function of Pandharipande-Thomas invariants is defined by
\[ 
Z^{(X,D)}_{\PT, \beta}\left( \lambda \middle| \tau_{k_1} \cdots \tau_{k_r}(\gamma) \right)
=
\sum_{m \in \frac{1}{2} \BZ} i^{2m} p^m
\big\langle \, \lambda \, \big| \, \tau_{k_1} \cdots \tau_{k_r} (\gamma) \big\rangle^{(X,D), \PT}_{m + \frac{1}{2} d_{\beta},\beta},
\]
where $i = \sqrt{-1}$.
\end{defn}

We index here the stable pairs series by
the third Chern character of the stable pair, that is if $\ch_3(F) = m$, then $\chi(F) = m + d_{\beta}/2$.

\begin{defn} The partition function of Gromov-Witten invariants is defined by
\begin{multline} \label{defn:ZGW}
Z^{(X,D)}_{\GW, \beta}\left( \lambda \middle| \tau_{k_1} \cdots \tau_{k_r}(\gamma) \right) \\
=
(-i)^{\d_{\beta}} (-1)^{ \ell(\lambda) - |\lambda| }
z^{\d_{\beta} + \ell(\lambda) - |\lambda| }
\sum_{g \in \BZ} (-1)^{g-1} z^{2g-2}
\left\langle \, \lambda \, \middle| \, \tau_{k_1} \cdots \tau_{k_r}(\gamma) \right\rangle^{(X,D), \GW, \bullet}_{g, \beta}.
\end{multline}
\end{defn}

Our variables $z$ and $p$ here are related to the standard genus and Euler characteristic variables $u$ and $q$ of \cite{PaPix_GWPT} by
the variable change $z = iu$ and $q=-p$.

Since the Gromov-Witten bracket is invariant under permutations of relative markings that preserve the ramification profile (i.e. under $\Aut(\vec{\lambda})$),
the partition function \eqref{defn:ZGW}
only depends on the associated class $\lambda \in H^{\ast}(D^{[|\lambda|]})$. Hence the above defines a morphism:
\[
Z^{(X,D)}_{\GW, \beta}\left( - \middle| \tau_{k_1} \cdots \tau_{k_r}(\gamma) \right) : 
H^{\ast}(D^{[\beta \cdot D]}) \to \BQ((z)),
\]

We will also require the partition functions of rubber invariants.
Let $\alpha \in H_2(D,\BZ)$ be a curve class, 
and let $\lambda, \mu$ be $H^{\ast}(D)$-weighted cohomology partition of size\footnote{By our convention, $\lambda$ records the ramification conditions with $D_0$,
and the fiberwise degree is measured against $D_{\infty}$.}
\[ |\lambda| = d_0 := d - \int_{\alpha} c_1(N_{D/X}), \quad |\mu| = d. \]
Let $\gamma \in H^{\ast}( (\p, D_{0,\infty})^{r,\sim})$ be a class.
We have
\begin{align*}
\mathsf{d}_{(\alpha,d)} & = \int_{(\alpha,d)} c_1(T_{\p})  \\
& = d + d_0 + \int_{\alpha} c_1(T_D).
\end{align*}

\begin{defn} The partition functions of rubber PT and GW invariants is defined by
\[ 
Z^{(\p, D_{0,\infty}), \sim}_{\PT, \alpha}\left( \lambda, \mu \middle| \tau_{k_1} \cdots \tau_{k_r}(\gamma) \right)
=
\sum_{m \in \frac{1}{2} \BZ} i^{2m} p^m
\big\langle \, \lambda, \mu \, \big| \, \tau_{k_1} \cdots \tau_{k_r}(\gamma) \big\rangle^{(\p, D_{0,\infty}), \PT, \sim}_{m + \frac{1}{2} d_{\alpha},(\alpha,d)},
\]
and
\begin{multline}
Z^{(\p, D_{0,\infty}),\sim}_{\GW, \alpha}\left( \lambda, \mu \middle| \tau_{k_1} \cdots \tau_{k_r}(\gamma) \right) 
=
(-i)^{\d_{(\alpha,d)}} (-1)^{ \ell(\lambda) - |\lambda| + \ell(\mu) - |\mu| } \\
\cdot z^{\d_{(\alpha,d)} + \ell(\lambda) - |\lambda| + \ell(\mu) - |\mu| }
\sum_{g \in \BZ} (-1)^{g-1} z^{2g-2}
\left\langle \, \lambda, \mu \, \middle| \, \tau_{k_1} \cdots \tau_{k_r}(\gamma) \right\rangle^{(\p, D_{0,\infty}), \GW, \bullet, \sim}_{g, \alpha}.
\end{multline}
\end{defn}

Here we choose the signs and prefactor in the rubber generating series to match
those of the non-rubber pair $(\p, D_{0,\infty})$.
This will yield a clean statement of rigidification.

\subsection{Properties}
The GW and PT partition functions are choosen so that
the degeneration formula, the rigidification and the splitting formulas
takes for them exactly the same form and moreover that there are no extra factors appearing.
For convenience we present here the form of these formulas for the partitions functions $Z^{(X,D)}_{\beta, \GW/\PT}( \cdots )$.
We will use the following convention: If we write
\[ Z^{(X,D)}_{\beta}\left( \ldots \right) \]
without specifying $\PT$ or $\GW$, the statement will hold for both GW and PT partition functions.
We then have the following (with the notation of the corresponding sections).

\vspace{6pt}
\noindent
\textbf{Degeneration formula} (Propositions~\ref{prop:degeneration formula PT} and \ref{prop:degeneration formula GW}):
We have
\begin{multline*}
\sum_{\substack{ \beta' \in H_2(X,\BZ) \\ \iota_{\ast} \beta' = \beta}}
Z^{X}_{\beta'}\left( \, \tau_{k_1} \cdots \tau_{k_r}( \gamma|_{X^r} ) \right) = \\
\sum_{\substack{ \beta_i \in H_2(X_i,\BZ) \\ \iota_{1 \ast} \beta_1 + \iota_{2 \ast} \beta_2 = \beta \\ \beta_1 \cdot D = \beta_2 \cdot D }}
\sum_{\substack{ \{ 1, \ldots, r \} = I \sqcup J \\ \ell }}
Z^{(X_1,D)}_{\beta_1}\left( \Delta_1 \middle| \left( {\textstyle {\prod_{i \in I}} \tau_{k_i} } \right) (\delta_{I,\ell}) \right)
Z^{(X_2,D)}_{\beta_2}\left( \Delta_2 \middle| \left( {\textstyle {\prod_{i \in J}} \tau_{k_j} } \right) (\delta'_{I,\ell}) \right)
\end{multline*}
where $\Delta_1, \Delta_2$ runs over the K\"unneth decomposition of $\Delta_{D^{[\beta_1 \cdot D]}} \subset (D^{[\beta_1 \cdot D]})^2$.

\begin{proof}
Let $W \to B$ be the total space of the degeneration and consider the logarithmic tangent bundle $T_W[-W_0]$, see
Section~\ref{subsec:log tangent bundle}. For $t \in B \setminus \{ 0 \}$ we have the restriction
\[ T_W[-W_0]|_{W_t} = T_{W_t} \oplus \CO \]
and over $0$ we have the restrictions
\[ T_W[-W_0]|_{X_1} = T_{X_1}[-D], \quad T_W[-W_0]|_{X_2} = T_{X_2}[-D]. \]
Hence given $\beta', \beta_1, \beta_2$ as in the claim we find that
\begin{align*}
\d_{\beta'} & =  \int_{\beta'} c_1(T_W[-W_0]) \\
& = \int_{\beta} c_1(T_W[-W_0]) \\
& = \int_{\beta_1} c_1( T_{X_1}[-D] ) + \int_{\beta_2} c_1(T_{X_2}[-D]) \\
& = \d_{\beta_1} + \d_{\beta_2} - D \cdot \beta_1 - D \cdot \beta_2.
\end{align*}
where we use the exact sequence $0 \to \Omega_{X_i} \to \Omega_{X_i}[D] \to \iota_{\ast} \CO_{D} \to 0$ in the last step.

The statement on the PT side follows from the fact that for $n=n_1 + n_2 - D \cdot \beta$ we have 
\[ n-\frac{1}{2} \d_{\beta'} = \left( n_1 - \frac{1}{2} \d_{\beta_1} \right) + \left( n_2 - \frac{1}{2} \d_{\beta_2} \right). \]

The GW side follows since for $g = g_1 + g_2 + \ell(\mu) - 1$ we have
\begin{align*}
(-i)^{\d_{\beta'}} z^{\d_{\beta'}} (-1)^{g-1} z^{2g-2}
& = (-i)^{\d_{\beta_1}} (-1)^{\ell(\mu) - |\mu|} (-1)^{g_1-1} z^{2g_1-2 + \ell(\mu) - |\mu| + \d_{\beta_1}} \\
& \phantom{=} \cdot (-i)^{\d_{\beta_2}} (-1)^{\ell(\mu) - |\mu|} (-1)^{g_2-1} z^{2g_2-2 + \ell(\mu) - |\mu| + \d_{\beta_2}} \\
& \phantom{=}\cdot (-1)^{\ell(\mu) - |\mu|}
\end{align*}
and the last sign together with the splitting factor $\prod_i \mu_i / |\Aut(\mu)|$ yields precisely the K\"unneth decomposition of the diagonal as in Lemma~\ref{Kunneth diagonal}.
\end{proof}

\vspace{6pt}
\noindent
\textbf{Rigidification} (Corollary~\ref{cor:rigidification}).
For any $\gamma \in (\p,D_{0,\infty})^{r, \sim}$ for $r \geq 1$,
\[
Z^{(\p, D_{0,\infty}), \sim}_{\alpha}\left( \lambda, \mu \middle| \tau_{k_1} \cdots \tau_{k_r}(\gamma) \right)
=
Z^{(\p, D_{0,\infty})}_{\U, \alpha}\left( \lambda, \mu \middle| \tau_{k_1} \cdots \tau_{k_r}\big(\pi_1^{\ast}(D_0) f^{\ast}(\gamma) \big) \right)
\]
where $f : (\p,D_{0,\infty})^r \to (\p,D_{0,\infty})^{r, \sim}$ is the natural morphism.

\vspace{10pt}
\noindent
\textbf{Splitting formula} (Propositions~\ref{prop:splitting relative diagonal} and \ref{prop:splitting relative diagonal GW}).
\begin{multline*}
Z^{(X,D)}_{\beta} \left( \, \lambda \, \big| \, \tau_{k_1} \cdots \tau_{k_r}( \Delta_{12}^{\rel} \cdot \gamma) \right)
=
Z^{(X,D)}_{\beta} \left( \, \lambda \, \big| \, \tau_{k_1} \cdots \tau_{k_r}( \Delta_{12} \cdot \gamma) \right) - \\
\sum_{\iota_{\ast} \alpha + \beta' = \beta}
\sum_{ \substack{ \{ 1, \ldots, r \} = I \sqcup J \\ 1,2 \in I \\ \ell }}
Z^{(\p,D_{0,\infty}), \sim}_{(\alpha, d)}\left( \Delta_1, \lambda \big| \left( {\textstyle \prod_{i \in I} \tau_{k_i} } \right) ( \Delta_{D} \delta_{I,\ell}) \right)
Z^{(X,D)}_{\beta'}\left( \Delta_2 \big| \left( {\textstyle \prod_{i \in J} \tau_{k_i} } \right)( \delta'_{I,\ell} ) \right)
\end{multline*}
where $d = \beta \cdot D$ and $d_0 = d - \int_{\alpha} c_1(N_{D/X})$ and 
$\Delta_1, \Delta_2$ runs over the K\"unneth decomposition of the diagonal $\Delta_{D^{[d_0]}}$.

\begin{proof}
One argues as before.
To show that $\d_{\beta} = \d_{\beta'} + \d_{(\alpha,d)} - 2d_0$ 
one can use the same argument as before for the degeneration of $X$ to the normal cone of $D$.
\end{proof}

\subsection{Standard correspondence}
By \cite{PP_GWPT_Toric3folds} there exists a universal correspondence matrix\footnote{The universal GW/PT correspondence matrix is denoted by $\tilde{K}_{\alpha, \widehat{\alpha}}$ in \cite{PP_GWPT_Toric3folds, PaPix_GWPT},
and related to our matrix $\K_{\alpha, \widetilde{\alpha}}$ by the variable change:
\[ \K_{\alpha,\widetilde{\alpha}} = \widetilde{\mathsf{K}}_{\alpha, \widetilde{\alpha}}\Big|_{u = -i z}. \]}
\[ \K_{\alpha,\widetilde{\alpha}}\in \mathbb{Q}[i,c_1,c_2,c_3]((z)) \]
indexed by partitions $\alpha$ and $\widetilde{\alpha}$ of positive size.
The $c_i$ are formal variables that in the formulas below will be specialized to
the Chern classes of the logarithmic tangent bundle, 
\[ c_i := c_i(T_X[-D]). \]
We have the basic vanishing 
\begin{equation} \widetilde{\mathsf{K}}_{\alpha,\widetilde{\alpha}}=0 \text{ for all } |\alpha| < |\widetilde{\alpha}|. \label{vanishingaa} \end{equation}
This ensures that in the sums below all except finitely many terms are zero.

Let $\alpha = (\alpha_1, \alpha_2 , \ldots, \alpha_r)$ be a partition, and write
\[ \tau_{[\alpha]} = \tau_{\alpha_{1}-1} \cdots \tau_{\alpha_r - 1}. \]
Let $P$ be a set partition of the index set $\{1, \ldots, r \}$.
For any part $T \in P$ given by a subset $T \subset \{ 1, \ldots, r \}$ we
can form a partition from the $T$-indices of $\alpha$:
\[ \alpha_T := (\alpha_i )_{i \in T} \]

\begin{defn}[\cite{PaPix_GWPT}, Section 1.3] \label{defn:desc tra 1}
For any classes $\gamma_1, \ldots, \gamma_r \in H^{\ast}(X)$ define the descendent transformation:
\begin{equation*}
\overline{\tau_{\alpha_1-1}(\gamma_1)\cdots
\tau_{\alpha_{\ell}-1}(\gamma_{\ell})}
=
\sum_{\substack{P \textup{ set partitions}\\\textup{ of }\{1,\ldots,r\}}}\ \prod_{T\in P}\ 
\left[ \sum_{\tilde{\alpha}} \tau_{[\tilde{\alpha}]}\Bigg(\Delta^{\text{rel}}_{1, \ldots, \ell(\tilde{\alpha})} \cdot \pi_1^{\ast}\big( \mathsf{K}_{\alpha_T,\tilde{\alpha}} \cdot \prod_{i \in T} \gamma_i \big) \Bigg) \right].
\end{equation*}
where $\tilde{\alpha}$ runs over all partitions,
$\pi_1 : (X,D)^{\ell(\tilde{\alpha})} \to (X,D) \cong X$ is the forgetful morphism
and $\Delta^{\text{rel}}_{1, \ldots, \ell(\tilde{\alpha})}$ denotes the (class of the) small diagonal in $(X,D)^{\ell(\tilde{\alpha})}$.
\end{defn}

We state the main GW/PT correspondence for relative geometries in the form of \cite{PaPix_GWPT}.

\begin{conj}[{\cite[Conjecture 4]{PaPix_GWPT}}] \label{conj:GWPT} 
Let $\gamma_1, \ldots, \gamma_r \in H^{\ast}(X)$. We have that
\[ Z^{(X,D)}_{\PT, \beta}\big( \lambda \big| \tau_{\alpha_1-1}(\gamma_1) \cdots \tau_{\alpha_r-1}(\gamma_r) \big) \]
is the Fourier expansion of a rational function in $p$, and
\begin{equation} \label{eqn_correspondence}
Z^{(X,D)}_{\PT, \beta}\big(\, \lambda \, \big| \,\tau_{\alpha_1-1}(\gamma_1) \cdots \tau_{\alpha_r-1}(\gamma_r) \big) \\
=
Z^{(X,D)}_{\GW, \beta}\left(\, \lambda\, \middle|\, \overline{\tau_{\alpha_1-1}(\gamma_1) \cdots \tau_{\alpha_r-1}(\gamma_r) } \right)
\end{equation}
under the variable change $p=e^{z}$.
\end{conj}

\begin{rmk} \label{rmk:GW in terms of PT}
From \cite[Sec.7]{PP_GWPT_Toric3folds} it follows that
\begin{equation} \overline{\tau_{\alpha_1-1}(\gamma_1) \cdots \tau_{\alpha_r-1}(\gamma_r) } = 
z^{\ell(\alpha) - |\alpha|} \tau_{\alpha_1-1}(\gamma_1) \cdots \tau_{\alpha_r-1}(\gamma_r) + ( \ldots ) \label{dsfAAA} \end{equation}
where the dots stand for terms $\tau_{[\tilde{\alpha}]}(\cdots )$ with $|\tilde{\alpha}| < |\alpha|$.
Hence, by an induction on $|\alpha|$, the correspondence of Conjecture~\ref{conj:GWPT} can be
be inverted and can be used to express all PT invariants in terms of GW invariants.

(To prove \eqref{dsfAAA} one uses \eqref{vanishingaa} and shows that $\K_{\alpha, \tilde{\alpha}} = \delta_{\alpha, \tilde{\alpha}} \delta_{\ell(\alpha),1} z^{\ell(\alpha)-|\alpha|}$
whenever $|\alpha| = |\tilde{\alpha}|$.
The latter follows from \cite[Prop.24]{PP_GWPT_Toric3folds} and a direct evaluation of the non-vansihg term via \cite[Eqn. (59)]{PP_GWPT_Toric3folds}
and basic properties of the Sterling numbers of the second kind.) \qed
%
\end{rmk}

\subsection{Generalized correspondence} \label{sec:generalized correspondence}
The conjectural GW/PT correspondence of Conjecture~\ref{conj:GWPT} apples to the usual descendent PT invariants,
and hence by the comparision of Proposition~\ref{prop: comparision}
to all marked relative invariants for insertions
of the form $\gamma = \pi_1^{\ast}(\gamma_1) \cdots \pi_r^{\ast}(\gamma_r)$ on $(X,D)^r$.
We will need a generalized form of the correspondence, valid for all $\gamma \in H^{\ast}((X,D)^r)$

Let $\alpha = (\alpha_1, \alpha_2 , \ldots, \alpha_r)$ be again a partition,
and let $P$ be a set partition of the index set $\{1, \ldots, r \}$.
For any $T \in P$ let a partition $\hat{\alpha}_T$ be given. We write
\[ \hat{\alpha} =  (\hat{\alpha}_T)_{T \in P}, \quad \ell(\hat{\alpha}) = \sum_T \ell(\hat{\alpha}_T). \]
Consider any set partition
\[ \left\{ 1, \ldots, r, r+1, \ldots, r+\ell(\hat{\alpha}) \right\} = \{ 1, \ldots, r \} \sqcup \bigsqcup_{T} I_T. \]

We define the cycle
\[ \Gamma_{\alpha,P, \hat{\alpha}} \in A^{\ast}( (X,D)^{r + \ell(\hat{\alpha}) }) \]
by
\[
\Gamma_{\alpha,P,\hat{\alpha}}
=
\prod_{T \in P} \pi_{T \sqcup I_T}^{\ast}\left( \pi_1^{\ast}(\K_{\alpha_{T}, \hat{\alpha}_T}) \cdot \Delta_{T \sqcup I_T}^{\rel} \right),
\]
where
\[
\pi_{T \sqcup I_T} : (X,D)^{r + \ell(\hat{\alpha})} \to (X,D)^{T \sqcup I_T}
\]
is the forgetful morphism to the indices labeled $T \sqcup I_T$,
\[ \pi_1 : (X,D)^{T \sqcup I_T}  \to (X,D) \cong X \]
is the map forgetting all but the first marking, and
\[ \Delta_{T \sqcup I_T}^{\rel} \subset (X,D)^{T \sqcup I_T} \]
is the small diagonal (the image of the diagonal $(X,D) \to (X,D)^{T \sqcup I_T}$).

We view the cycle $\Gamma_{\alpha,P,\hat{\alpha}}$ as defining a morphism on cohomology in the usual way:
\begin{gather*}
\Gamma_{\alpha,P, \hat{\alpha}} : H^{\ast}( (X,D)^r ) \to H^{\ast}( (X,D)^{\ell(\widehat{\alpha})} ) \\
\gamma \mapsto \rho_{2 \ast}( \rho_1^{\ast}( \gamma) \cdot \Gamma_{\alpha,P, \hat{\alpha}})
\end{gather*}
where the maps $\rho_1, \rho_2$ are the natural projection maps:
\[
\begin{tikzcd}
&  (X,D)^{r + \ell(\hat{\alpha})} \ar{dl}[swap]{\rho_1} \ar{dr}{\rho_2} & \\
(X,D)^r & & (X,D)^{\ell(\hat{\alpha})}.
\end{tikzcd}
\]
%

The main definition is the following:
\begin{defn} \label{defn:desc tra 2} For any $\gamma \in H^{\ast}((X,D)^r)$ we have
\[
\overline{\tau_{\alpha_{1}-1} \cdots \tau_{\alpha_r - 1}(\gamma)}
=
\sum_{P \textup{ set partition of }\{1,\ldots,\}} 
\sum_{ \substack{\textup{for all }T \in P \\ \textup{a partition } \hat{\alpha}_T }}
\left( { \textstyle \prod_{T} \tau_{[\hat{\alpha}_T]} }\right) ( \Gamma_{\alpha,P, \hat{\alpha}}(\gamma) )
\]
\end{defn}

In the special case where $\gamma$ is pulled back from $X^r$ we recover Definition~\ref{defn:desc tra 1}.

\begin{lemma} \label{lemma:correspondence standard case}
Assume that $\gamma = \pi_1^{\ast}(\gamma_1) \cdots \pi_r^{\ast}(\gamma_r) \in H^{\ast}((X,D)^r)$ 
for some $\gamma_i \in H^{\ast}(X)$. Then the transformation of Definition~\ref{defn:desc tra 2} yields:
\begin{equation*}
\overline{\tau_{\alpha_1-1}(\gamma_1)\cdots
\tau_{\alpha_{\ell}-1}(\gamma_{\ell})}
=
\sum_{\substack{P \textup{ set partitions}\\\textup{ of }\{1,\ldots,r\}}}\ \prod_{T\in P}\ 
\left[ \sum_{\tilde{\alpha}} \tau_{[\tilde{\alpha}]}\Bigg(\Delta^{\text{rel}}_{1, \ldots, \ell(\tilde{\alpha})} \cdot \pi_1^{\ast}\big( \mathsf{K}_{\alpha_T,\tilde{\alpha}} \cdot \prod_{i \in T} \gamma_i \big) \Bigg) \right].
\end{equation*}
where $\tilde{\alpha}$ runs over partitions,
and $\Delta_{1, \ldots, \ell(\widetilde{\alpha})}^{\rel}$ is the small diagonal in $(X,D)^{\ell(\widetilde{\alpha})}$.
\end{lemma}
\begin{proof}
If $\gamma = \pi_1^{\ast}(\gamma_1) \cdots \pi_r^{\ast}(\gamma_r)$, then we have
\begin{align*}
\rho_1^{\ast}(\gamma) \cdot \pi_{T \sqcup I_T}^{\ast}\left( \pi_1^{\ast}(\K_{\alpha_{T}, \hat{\alpha}_T}) \cdot \Delta_{T \sqcup I_T}^{\rel} \right) 
& = \pi_1^{\ast}(\gamma_1) \cdots \pi_r^{\ast}(\gamma_r) \cdot \pi_{T \sqcup I_T}^{\ast}\left( \pi_1^{\ast}(\K_{\alpha_{T}, \hat{\alpha}_T}) \cdot \Delta_{T \sqcup I_T}^{\rel} \right) \\
& = \prod_{i \notin T} \pi_i^{\ast}(\gamma_i) \cdot 
\pi_{T \sqcup I_T}^{\ast}\left( \prod_{i \in T} \pi_i^{\ast}(\gamma_i) \cdot \pi_1^{\ast}(\K_{\alpha_{T}, \hat{\alpha}_T}) \cdot \Delta_{T \sqcup I_T}^{\rel} \right) \\
& = \prod_{i \notin T} \pi_i^{\ast}(\gamma_i) \cdot 
\pi_{T \sqcup I_T}^{\ast}\left( \pi_1^{\ast}\Big( \K_{\alpha_{T}, \hat{\alpha}_T} \cdot \prod_{i \in T} \gamma_i \Big) \cdot \Delta_{T \sqcup I_T}^{\rel} \right).
\end{align*}
Taking the product over $T$ we hence find that
\[
\rho_1^{\ast}(\gamma) \Gamma_{\alpha, P, \widehat{\alpha}}
=
\prod_{T \in P}
\pi_{T \sqcup I_T}^{\ast}\left( \pi_1^{\ast}\Big( \K_{\alpha_{T}, \hat{\alpha}_T} \cdot \prod_{i \in T} \gamma_i \Big) \cdot \Delta_{T \sqcup I_T}^{\rel} \right).
\]
This implies the claim by pushforward by $\rho_2$ (Use that if $\delta \in H^{\ast}(X)$ is a class,
and given any partition $\{ 1, \ldots, s \} = I \sqcup J$, then in $H^{\ast}((X,D)^J)$ we have
\[
\pi_{I \ast}( \pi_1^{\ast}(\delta) \Delta_{I \sqcup J}^{\rel}) = \pi_1^{\ast}(\delta) \cdot \Delta_J^{\rel}. \ )
\]
\end{proof}

We can now state the general form of the GW/PT correspondence for marked relative insertions.

\begin{conj} \label{conj:relGWPT}
For all $\gamma \in H^{\ast}((X,D)^r)$ we have that
\[ Z^{(X,D)}_{\PT, \beta}\big(\, \lambda \, \big|\, \tau_{\alpha_1-1} \cdots \tau_{\alpha_r-1}(\gamma) \big) \]
is the Fourier expansion of a rational function in $p$, and
\[
Z^{(X,D)}_{\PT, \beta}\big(\, \lambda \, \big|\, \tau_{\alpha_1-1} \cdots \tau_{\alpha_r-1}(\gamma) \big)
=
Z^{(X,D)}_{\GW, \beta}\left( \, \lambda \, \middle|  \, \overline{\tau_{\alpha_1-1} \cdots \tau_{\alpha_r-1}(\gamma)} \right)
\]
under the variable change $p=e^{z}$.
\end{conj}

\subsection{Compatibility with the splitting formula} \label{subsec:compatibility splitting formula}
We derive the compatibility with the splitting formula.
Consider the subring\footnote{It would be interesting to know whether this subring is equal to the full cohomology ring $H^{\ast}((X,D)^r)$.}
\[ DH^{\ast}( (X,D)^r ) \subset H^{\ast}((X,D)^r) \]
generated by all
\begin{itemize} 
\item big diagonals $\Delta_{ab}^{\rel} = \pi^{\ast}(\Delta_{ab})$ for all $a,b \in \{ 1, \ldots, r \}$,
\item classes $\pi_i^{\ast}(\gamma)$ for all $i \in \{ 1, \ldots, r \}$ and $\gamma \in H^{\ast}(X)$,
where $\pi_i : (X,D)^r \to (X,D)^1 \cong X$ is the map forgetting all but the $i$-th marking.
\end{itemize}

\begin{prop} \label{prop:relGWPT}
Assume that Conjecture~\ref{conj:GWPT} holds for
$(X,D)$ and $(\p(N_{D/X} \oplus \CO), D_{0, \infty})$.
Then Conjecture~\ref{conj:relGWPT} holds for $(X,D)$ and all $\gamma \in DH^{\ast}((X,D)^r)$.
\end{prop}
\begin{proof}
For $r=0$ the claim is trivial,
and for $r=1$ it holds since $(X,D)^1 \cong X$ and by
Proposition~\ref{prop: comparision}. In general,
we write $\gamma$ as a monomials in diagonal classes $\Delta^{\rel}_{ab}$ and pullbacks $\pi_i^{\ast}(\gamma_i)$,
and argue by induction on the number of diagonal factors.
If there are no diagonal factors the claim follows by Proposition~\ref{prop: comparision}, this is the base.
In general, let $\gamma$ be written as a product of $L>0$ diagonal factors and assume the statement whenever there are less than $L$ factors.
Without loss of generality, we may write
\[ \gamma = \Delta^{\rel}_{12} \gamma' \]
where $\gamma'$ is written as a product of $(L-1)$-diagonal factors.
Write also $\alpha = (k_1+1, \ldots, k_r+1)$.

\vspace{5pt}
\noindent \textbf{Rationality.}
We consider the PT side and apply the splitting formula. The result is:
\begin{multline} \label{3sdf11}
Z^{(X,D)}_{\PT, \beta} \left( \, \lambda \, \big| \, \tau_{k_1} \cdots \tau_{k_r}( \Delta_{12}^{\rel} \cdot \gamma) \right)
=
Z^{(X,D)}_{\PT, \beta} \left( \, \lambda \, \big| \, \tau_{k_1} \cdots \tau_{k_r}( \Delta_{12} \cdot \gamma) \right)  \\[5pt]
- \sum_{ \substack{\iota_{\ast} \alpha + \beta' = \beta \\ \{ 1, \ldots, r \} = A \sqcup B \\ 1,2 \in A \\ \ell }}
Z^{(\p,D_{0,\infty}), \sim}_{\PT,(\alpha, d)}\left( \Delta_{D^{[d_0]},1}, \lambda \big| \left( {\textstyle \prod_{i \in A} \tau_{k_i} } \right) ( \Delta_{D} \delta_{A,\ell}) \right)
Z^{(X,D)}_{\PT,\beta'}\left( \Delta_{D^{[d_0]},2} \big| \left( {\textstyle \prod_{i \in B} \tau_{k_i} } \right)( \delta'_{A,\ell} ) \right)
\end{multline}
where $d = \beta \cdot D$ and $d_0 = d - \int_{\alpha} c_1(N_{D/X})$ and 
$\Delta_{D^{[d_0]},1}, \Delta_{D^{[d_0]},2}$ runs over the K\"unneth decomposition of the diagonal $\Delta_{D^{[d_0]}}$.
Moreover, we used the K\"unneth decomposition
\begin{equation} \xi_{A}^{\ast}( \gamma ) = \sum_{\ell} \delta_{A,\ell} \otimes \delta'_{A,\ell}  \quad \in H^{\ast}( (\p, D_{0,\infty})^{A, \sim} \times (X,D)^{B} ). \label{Kunneth11}
\end{equation}
where for any decomposition $\{1, \ldots, r \} = A \sqcup B$ we have the gluing morphism
\[ \xi_{A} : (\p, D_{0,\infty})^{A,\sim} \times (X,D)^{B} \to (X,D)^r. \]

By Lemma~\ref{lemma:diagonal pullback} all terms on the right are again product of big diagonals and $\pi_i^{\ast}(\gamma_i)$ for some $\gamma_i$.
We can further apply rifidification to conclude that:
\[
Z^{(\p,D_{0,\infty}), \sim}_{\PT,(\alpha, d)}\left( \Delta_1, \lambda \big| \left( {\textstyle \prod_{i \in A} \tau_{k_i} } \right) ( \Delta_{D} \delta_{A,\ell}) \right)
=
Z^{(\p,D_{0,\infty})}_{\PT,(\alpha, d)}\left( \Delta_1, \lambda \big| \left( {\textstyle \prod_{i \in A} \tau_{k_i} } \right) ( \pi_1^{\ast}(D_0) \Delta_{D} f^{\ast}(\delta_{A,\ell})) \right).
\]
where
\[ f : (\p, D_{0,\infty})^{r} \to (\p, D_{0,\infty})^{r,\sim} \]
is the forgetful morphism.

By induction applied to both $(X,D)$ and $(\p, D_{0,\infty})$ we find that
all terms on the right of \eqref{3sdf11} are the Fourier expansion of a rational functions in $p$,
hence so is the left hand side.

\vspace{5pt}
\noindent \textbf{GW/PT relation.}
Recall the diagram
\[
\begin{tikzcd}
\p(N) \ar{d}{g} \ar{r}{j} & (X,D)^2 \ar{d}{\pi} \\
D \times D \ar{r}{i} & X^2.
\end{tikzcd}
\]
where $N = \CO(D)|D \oplus \CO(D)|D$ is the normal bundle of $D \times D \subset X \times X$,
and the splitting
\[ \Delta_{(X,D)}^{\text{rel}} = \pi^{\ast}( \Delta_X ) - j_{\ast} g^{\ast}(\Delta_D). \]
Inserting we hence find
\begin{multline*}
Z^{(X,D)}_{\PT, \beta} \left( \, \lambda \, \big| \, \tau_{k_1} \cdots \tau_{k_r}( \Delta_{12}^{\rel} \cdot \gamma') \right) 
= \\
Z^{(X,D)}_{\PT, \beta} \left( \, \lambda \, \big| \, \tau_{k_1} \cdots \tau_{k_r}( \Delta_{12} \cdot \gamma') \right)
-
Z^{(X,D)}_{\PT, \beta} \left( \, \lambda \, \big| \, \tau_{k_1} \cdots \tau_{k_r}( (j_{\ast} g^{\ast}(\Delta_D))_{12} \cdot \gamma') \right)
\end{multline*}
By induction we have
\[
Z^{(X,D)}_{\PT, \beta} \left( \, \lambda \, \big| \, \tau_{k_1} \cdots \tau_{k_r}( \Delta_{12} \cdot \gamma') \right)
=
Z^{(X,D)}_{\GW, \beta} \left( \, \lambda \, \big| \, \overline{ \tau_{k_1} \cdots \tau_{k_r}( \Delta_{12} \cdot \gamma')} \right)
\]
hence it suffices to prove that:
\begin{equation} \label{ABC}
Z^{(X,D)}_{\PT, \beta} \left( \, \lambda \, \big| \, \tau_{k_1} \cdots \tau_{k_r}( (j_{\ast} g^{\ast}(\Delta_D))_{12} \cdot \gamma') \right)
=
Z^{(X,D)}_{\GW, \beta} \left( \, \lambda \, \big| \, \overline{ \tau_{k_1} \cdots \tau_{k_r}( (j_{\ast} g^{\ast}(\Delta_D))_{12} \cdot \gamma') } \right)
\end{equation}

We consider the term $\overline{ \tau_{k_1} \cdots \tau_{k_r}( (j_{\ast} g^{\ast}(\Delta_D))_{12} \cdot \gamma') }$.
For any set partition $P$ of $\{ 1, \ldots, r \}$ and for every list of partitions $(\hat{\alpha}_T)_{T \in P}$ 
we need to understand
\[ \Gamma_{\alpha, P ,\hat{\alpha}}\big( (j_{\ast} g^{\ast}(\Delta_D))_{12} \cdot \gamma' \big). \]

\begin{lemma} \label{lemma:splitting} In $H^{\ast}((X,D)^{\ell(\hat{\alpha})})$ we have the equality
\[
\Gamma_{\alpha, P ,\hat{\alpha}}\big( (j_{\ast} g^{\ast} \Delta_D)_{12} \cdot \gamma' \big) =
\sum_{A} \xi_{A \ast} (f \times \id)_{\ast}
\Big( \Gamma_{\alpha_A, P_A, \widehat{\alpha}_A}\big( \Delta_{D,12} \pi_1^{\ast}(D_0) f^{\ast}(\delta_{A,\ell}) \big)
\boxtimes \Gamma_{\alpha_B, P_B, \widehat{\alpha}_B}( \delta'_{A,\ell} ) \Big)
\]
where the sum runs over all decompositions $\{ 1, \ldots, r \} = A \sqcup B$ such that
\begin{itemize}[itemsep=0pt]
\item[(i)] $1,2 \in A$ 
\item[(ii)] $T \subset A$ or $T \subset B$ for all $T \in P$.
\end{itemize}
Moreover, we used the K\"unneth decomposition \eqref{Kunneth11} and
\begin{itemize}
\item $\alpha_A$ are the parts of $\alpha$ labeled by $I$,
\item $P_A = \{ T \in P | T \subset A\}$ is the partition of $A$ induced by $P$,
\item and $\widehat{\alpha}_A = (\widehat{\alpha}_T)_{T \subset A}$,
\end{itemize}
and similarly for $B$. The pushforward is along the map:
\[
\begin{tikzcd}
(\p, D_{0,\infty})^{\sqcup_{T \subset A} I_T, \sim} \times (X,D)^{\sqcup_{T \subset B} I_T} \ar{r}{f \times \id} &
(\p, D_{0,\infty})^{\sqcup_{T \subset A} I_T} \times (X,D)^{\sqcup_{T \subset B} I_T} \ar{r}{\xi} &
(X,D)^{\ell(\hat{\alpha})}
\end{tikzcd}
\]
\end{lemma}
\begin{proof}[Proof of Lemma~\ref{lemma:splitting}]
Consider the diagram
\[
\begin{tikzcd}
\bigsqcup_{I}
(\p, D_{0,\infty})^{\tilde{A},\sim} \times (X,D)^{\tilde{B}} \ar{r}{\sqcup_{\tilde{A}} \xi_{\tilde{A}}} & W \ar{r} \ar{d} &  (X,D)^{r + \ell(\hat{\alpha})} \ar{d} \\
& \p(N) \ar{r}{j} & (X,D)^2 
\end{tikzcd}
\]
where the right square is fibered (defining $W$),
and $\tilde{A}$ runs over subsets $\{ 1, \ldots, r+ \ell(\hat{\alpha}) \}$
containing $1,2$ with complement $\tilde{B}$.
By Lemma~\ref{lemma:2diagonal} the gluing map $\xi$ is birational, so that
\[ [W] = \sum_{1,2 \in \tilde{A}} \xi_{\ast}[ (\p, D_{0,\infty})^{\tilde{A},\sim} \times (X,D)^{\tilde{B}} ]. \]

For every $T \in P$ we have
\begin{equation} \label{Cases} \xi_{\tilde{A}}^{\ast}( \Delta_{I \sqcup I_T}^{\rel} ) = 
\begin{cases}
\Delta_{I \sqcup I_T}^{\rel} & \text{if } T \sqcup I_T \subset \tilde{A} \text{ or } T \sqcup I_T \subset \tilde{B} \\
0 & \text{otherwise}.
\end{cases}
\end{equation} 

The decompositions $\{ 1, \ldots, r+\ell(\hat{\alpha}) \} = \tilde{A} \sqcup \tilde{B}$ with $1,2 \in \tilde{A}$ and satisfying the first condition in \eqref{Cases}
are in $1$-to-$1$ correspondence with decompositions $\{1, \ldots, r \} = A \sqcup B$
satisfying (i) and (ii).
Indeed, to the decomposition $\{ 1, \ldots, r \} = A \sqcup B$ satisfying (i,ii)
we can associate 
\begin{equation} \tilde{A} = A \sqcup \bigsqcup_{\substack{T \in P \\ T \subset A}} I_T. \label{Atilde} \end{equation}
With $\rho_1 : (X,D)^{r + \ell(\hat{\alpha})} \to (X,D)^r$ the projection to the first $r$ factors, 
we hence have
\begin{multline}\label{XXY}
\rho_1^{\ast} \left( (j_{\ast} g^{\ast} \Delta_D)_{12} \cdot \gamma' \right) \cdot \Gamma_{\alpha, P , \widehat{\alpha}}
=\\
\sum_{A}
\sum_{\ell}
\xi_{\tilde{A} \ast} \Bigg( \left[ 
\rho^{\prime \ast}_1( \Delta_{D, 12} \delta_{A,\ell})
\prod_{\substack{ T \in P \\ T \subset A}} 
\pi_{T \sqcup I_T}^{\ast}\left( \pi_1^{\ast}(\K_{\alpha_{T}, \hat{\alpha}_T}\big|_{c_i = c_i(T_D)}) \Delta_{T \sqcup I_T}^{\rel} 
\right) \right] \\
\boxtimes
\left[ 
\rho^{\prime \prime \ast}_1(\delta_{B,\ell})
\prod_{\substack{ T \in P \\ T \subset B}} 
\pi_{T \sqcup I_T}^{\ast}\left( \pi_1^{\ast}(\K_{\alpha_{T}, \hat{\alpha}_T}) \Delta_{T \sqcup I_T}^{\rel}
\right) \right] \Bigg)
\end{multline} 
where $\tilde{A}$ is given by \eqref{Atilde} and we used the forgetful morphisms
\[ \rho'_1 : (\p, D_{0,\infty})^{\tilde{A}} \to (\p,D_{0,\infty})^{A},
\quad \rho^{\prime \prime}_1 : (X,D)^{\tilde{B}} \to (X,D)^B \]
as well as $c_i(T_X[-D]|_{D}) = c_i(T_D)$,
see Section~\ref{subsec:log tangent bundle}, and the K\"unnet decomposition \eqref{Kunneth11}

Write $\tilde{A} = A \sqcup \hat{A}$ and $\tilde{B} = B \sqcup \hat{B}$ and consider the commutative diagram
\[
\begin{tikzcd}
(\p,D_{0,\infty})^{\widehat{A}} \times (X,D)^{\widehat{B}} \ar{r}{f \times \id}  & 
(\p,D_{0,\infty})^{\widehat{A}, \sim} \times (X,D)^{\widehat{B}} \ar{r}{\xi_{\widehat{A}}} &
(X,D)^{\ell(\hat{\alpha})} \\
(\p,D_{0,\infty})^{\widetilde{A}} \times (X,D)^{\widetilde{B}} \ar{r}{f \times \id} \ar{u}{\rho_2' \times \rho_2^{\prime \prime}} \ar[swap]{d}{\rho_1' \times \rho_1^{\prime \prime}} & 
(\p,D_{0,\infty})^{\widetilde{A}, \sim} \times (X,D)^{\widetilde{B}} \ar{r}{\xi_{\widetilde{A}}} \ar{u}{\rho_2' \times \rho_2^{\prime \prime}} \ar[swap]{d}{\rho_1' \times \rho_1^{\prime \prime}} &
(X,D)^{r + \ell(\hat{\alpha})} \ar{u}{\rho_2} \ar[swap]{d}{\rho_1} \\
(\p,D_{0,\infty})^{A} \times (X,D)^{B} \ar{r}{f \times 1}  & 
(\p,D_{0,\infty})^{A,\sim} \times (X,D)^{B} \ar{r}{\xi_{A}} &
(X,D)^{r}.
\end{tikzcd}
\]
By Rigidification (Lemma~\ref{lemma:point rigidification}) and Lemma~\ref{lemma:c_i restriction}
we have that
\begin{multline} \label{rigi statement}
\rho^{\prime \ast}_1( \Delta_{D, 12} \delta_{A,\ell})
\prod_{\substack{ T \in P \\ T \subset A}} 
\pi_{T \sqcup I_T}^{\ast}\left( \pi_1^{\ast}(\K_{\alpha_{T}, \hat{\alpha}_T}\big|_{c_i = c_i(T_D)}) \Delta_{T \sqcup I_T}^{\rel} 
\right) \\
=
f_{\ast}\left[ \rho^{\prime \ast}_1( \pi_1^{\ast}([D_0]) \Delta_{D, 12} f^{\ast}(\delta_{A,\ell}))
\prod_{\substack{ T \in P \\ T \subset A}} 
\pi_{T \sqcup I_T}^{\ast}\left( \pi_1^{\ast}(\K_{\alpha_{T}, \hat{\alpha}_T}) \Delta_{T \sqcup I_T}^{\rel}
\right) \right]
\end{multline}
where we used the standard insertion $c_i = c_i(T_{\p}[-D_{0,\infty}])$ fof $\K_{\alpha_T, \hat{\alpha}_T}$.
The claim of the lemma now follows from \eqref{XXY} and \eqref{rigi statement}
by pushing forward to via $\rho_2$ 
\end{proof}

We now prove the desired equality:
\[
Z^{(X,D)}_{\PT, \beta} \left( \, \lambda \, \big| \, \tau_{k_1} \cdots \tau_{k_r}( (j_{\ast} g^{\ast}(\Delta_D))_{12} \cdot \gamma') \right)
=
Z^{(X,D)}_{\GW, \beta} \left( \, \lambda \, \big| \, \overline{ \tau_{k_1} \cdots \tau_{k_r}( (j_{\ast} g^{\ast}(\Delta_D))_{12} \cdot \gamma') } \right)
\]
The left hand side we computed in the Rationality part of the proof to be
\begin{multline*}
Z^{(X,D)}_{\PT, \beta} \left( \, \lambda \, \big| \, \tau_{k_1} \cdots \tau_{k_r}( (j_{\ast} g^{\ast}(\Delta_D))_{12} \cdot \gamma') \right) = \\
\sum_{ \substack{\iota_{\ast} \alpha + \beta' = \beta \\ \{ 1, \ldots, r \} = A \sqcup B \\ 1,2 \in A \\ \ell }}
Z^{(\p,D_{0,\infty})}_{\PT,(\alpha, d)}\left( \Delta_{D^{[d_0]}, 1}, \lambda \big| \left( {\textstyle \prod_{i \in A} \tau_{k_i} } \right) ( \pi_1^{\ast}(D_0) \Delta_{D} f^{\ast}(\delta_{A,\ell})) \right) \cdot \\
\cdot Z^{(X,D)}_{\PT,\beta'}\left( \Delta_{D^{[d_0]},2} \big| \left( {\textstyle \prod_{i \in B} \tau_{k_i} } \right)( \delta'_{A,\ell} ) \right)
\end{multline*}
For the right hand side, we apply the claim and the usual splitting law of \cite{Li1, Li2} which gives that
\begin{multline*}
Z^{(X,D)}_{\GW, \beta} \left( \, \lambda \, \big| \, \overline{ \tau_{k_1} \cdots \tau_{k_r}( (j_{\ast} g^{\ast}(\Delta_D))_{12} \cdot \gamma') } \right) \\
=
\sum_{ \substack{\iota_{\ast} \alpha + \beta' = \beta \\ \{ 1, \ldots, r \} = A \sqcup B \\ 1,2 \in A \\ \ell }}
Z^{(\p,D_{0,\infty})}_{\GW,(\alpha, d)}\left( \Delta_{D^{[d_0]}, 1}, \lambda \big| \overline{ \left( {\textstyle \prod_{i \in A} \tau_{k_i} } \right) ( \pi_1^{\ast}(D_0) \Delta_{D} f^{\ast}(\delta_{A,\ell}))} \right) \cdot \\
\cdot Z^{(X,D)}_{\GW,\beta'}\left( \Delta_{D^{[d_0]},2} \big| \overline{ \left( {\textstyle \prod_{i \in J} \tau_{k_i} } \right)( \delta'_{A,\ell} ) } \right).
\end{multline*}
We conclude the claim by the induction hypothesis.
\end{proof}

\begin{lemma} \label{lemma:c_i restriction} Let $\pi_D : \p \to D$ be the restriction. Then we have
\[ c_i( T_{\p}[-D_{0,\infty}]) = \pi_D^{\ast}( c_i(T_{D}) ). \]
\end{lemma}
\begin{proof}
The claim follows from the computation:
\begin{align*}
& \ch(\Omega_{\p}[-D_{0} - D_{\infty}] \\
& = \ch(\Omega_{\p}) + \ch( \CO_{D_0} ) + \ch( \CO_{D_{\infty}} ) \\
& = \pi_D^{\ast} \ch(\Omega_D) + \ch(\Omega_{\p/D}) + \ch( \CO_{D_0} ) + \ch( \CO_{D_{\infty}} ) \\ 
& = \pi_D^{\ast} \ch(\Omega_D) + \ch(N_{D/X}^{\vee} \otimes \CO_{\p}(-1)) + \ch(\CO_{\p}(-1)) - 1 + \ch( \CO_{D_0} ) + \ch( \CO_{D_{\infty}} ) \\ 
& = \pi_D^{\ast} \ch(\Omega_D) + 1
\end{align*}
where in the first step we used the exact sequence \eqref{log tangent ses},
in the third step the dual of the Euler sequence
$0 \to \CO \to \pi_D^{\ast}(N_{D/X} \oplus \CO_D) \otimes \CO(1) \to T_{\p/D} \to 0$.
For the last step one uses that
\[ \xi = c_1(\CO(1)) = D_0 \]
so that
\[ \ch(\CO_{\p}(-1)) = \ch( \CO  ) - \ch(\CO_{D_0}) = 1 - \ch(\CO_{D_0}), \]
and similarly, 
\[ \ch(N_{D/X}^{\vee} \otimes \CO_{\p}(-1)) = \ch( \CO(-D_{\infty})) = 1 - \ch(\CO_{D_{\infty}}). \]
\end{proof}

\subsection{Compatibility with degeneration formula}
Consider a simple degeneration
\[ X \rightsquigarrow X_1 \cup_D X_2. \]
over a smooth curve $B$ with total space $W \to B$.
Consider a fixed class
\[ \gamma \in H^{\ast}( (\CW/B)^r ). \]
Recall the gluing map
\[ \xi = \sqcup \xi_I : \bigsqcup_{\{1, \ldots, r \} = I \sqcup J} (X_1, D)^{I} \times (X_2, D)^J \to (\CW/B)^r_0. \]
The K\"unneth decomposition of the pullback to the $I$-th component will be denoted by
\[ \xi_I^{\ast}(\gamma) = \sum_{\ell} \delta_{I,\ell} \otimes \delta'_{I,\ell} 
\quad \in H^{\ast}((X_1, D)^{I} \times (X_2, D)^J). \]

We show that Conjecture~\ref{conj:relGWPT} is compatible with the degeneration formula,
in the sense that the following two operations yield the same result:
\begin{itemize}
\item[(i)] Apply the correspondence on $X$ for $\gamma|_{X^r}$ (which reduces to the standard correspondence of Conjecture~\ref{conj:GWPT} by K\"unneth decomposition and Lemma~\ref{lemma:correspondence standard case}) and then apply the degeneration formula.
\item[(ii)]  Apply the degeneration formula and then apply the correspondence of Conjecture~\ref{conj:relGWPT} to both $(X_1, D)$ and $(X_2, D)$.
\end{itemize}
This compatibility statement then boils down to the following condition on the GW partition function:
%

\begin{prop}[Compatibility with the degeneration formula]
\label{prop:Compatibility GWPT with degeneration formula}
\begin{multline*}
\sum_{\substack{ \beta' \in H_2(X,\BZ) \\ \iota_{\ast} \beta' = \beta}}
Z^{X}_{\GW, \beta'}\left( \, \overline{ \tau_{k_1} \cdots \tau_{k_r}( \gamma|_{X^r} ) } \right) = \\
\sum_{\substack{ \beta_i \in H_2(X_i,\BZ) \\ \iota_{1 \ast} \beta_1 + \iota_{2 \ast} \beta_2 = \beta \\ \beta_1 \cdot D = \beta_2 \cdot D }}
\sum_{\substack{ \{ 1, \ldots, r \} = I \sqcup J \\ \ell }}
Z^{(X_1,D)}_{\GW, \beta_1}\left( \Delta_1 \middle| \overline{ \left( {\textstyle {\prod_{i \in I}} \tau_{k_i} } \right) (\delta_{I,\ell}) } \right)
Z^{(X_2,D)}_{\GW, \beta_2}\left( \Delta_2 \middle| \overline{ \left( {\textstyle {\prod_{i \in J}} \tau_{k_j} } \right) (\delta'_{I,\ell}) } \right)
\end{multline*}
where $\Delta_1, \Delta_2$ runs over the K\"unneth decomposition of $\Delta_{D^{[\beta_1 \cdot D]}} \subset (D^{[\beta_1 \cdot D]})^2$.
\end{prop}
\begin{proof}
This boils down to defining the correspondence $\Gamma_{\alpha, P , \hat{\alpha}}$ relatively on $(\CW/B)^r$,
and then showing that it restricts accordingly.
The methods are similarly to what was used in the proof of the claim for Proposition~\ref{prop:relGWPT}.
The key point is that the Chern classes of the log tangent bundle restrict as desired. Indeed, we have
\begin{gather*} c_i( T_W[-W_0] )|_{W_t} = c_i(T_{W_t}) \\
c_i(T_W[-W_0])|_{X_i} = c_i( T_{X_i}[-D] )
\end{gather*}
for all $t \neq 0$ and $i=1,2$.
\end{proof}

\begin{rmk}
In case $\gamma = \pi_1^{\ast}(\gamma_1) \cdots \pi_r^{\ast}(\gamma_r)$
Proposition~\ref{prop:Compatibility GWPT with degeneration formula} was obtained in \cite[Sec.1.4]{PaPix_GWPT}.
\end{rmk}

\subsection{Chow-theoretic correspondence}
It would be very interesting to understand the geometric origin of the Gromov-Witten / Pandharipande-Thomas correspondence.
As such one may ask for the most general framework in which it holds.
Is it just a numerical correspondence
or is it obtained by an algebraic correspondence
on some (possibly infinite-type) moduli space defined in terms of the geometry of the target $X$?
In the later case one would expect that the GW/PT correspondence lifts naturally to the cycle level,
that is holds in the Chow group of algebraic cycles modulo rational equivalence.
We formulate here a cycle-theoretic version of the GW/PT correspondence.
It requires further investigation whether such a correspondence should hold
(that the numerical correspondence of \cite{PaPix_GWPT} admits naturally a formulation on a cycle-level
is not automatic and may be seen as a positive sign).

Consider the GW and PT theory of the threefold $X$ (without relative condition) for simplicity.
Define the partition function of Gromov-Witten classes:
\[
\GW_{\beta}^X(\tau_{1} \cdots \tau_{k_r})
=
(-i z)^{\d_{\beta}} \sum_{g \in \BZ} (-1)^{g-1} z^{2g-2}
\ev_{\ast}\left[ \prod_{i=1}^{r} \psi_i^{k_i} \cap [ \Mbar^{\bullet}_{g,r}(X,\beta) ]^{\vir} \right]
\]
Consider also the partition function of Pandharipande-Thomas classes:\footnote{We could have defined this class
also via the moduli space of $r$-marked stable pairs by
\[ 
\PT_{\beta}^X(\tau_1 \cdots \tau_{k_r}) 
=
\sum_{m \in \frac{1}{2} \BZ} i^{2m} p^m
\ev_{\ast}\left[ \prod_{i=1}^{r} p_i^{\ast}(\ch_{2+k_i}(\BF)) [ P_{m + \frac{1}{2} d_{\beta},\beta, \beta, r}(X) ]^{\vir} \right].
\]
}
\[ 
\PT_{\beta}^X(\tau_1 \cdots \tau_{k_r}) 
=
\sum_{m \in \frac{1}{2} \BZ} i^{2m} p^m
\rho_{\ast}\left[ \pi^{\ast}\left( [ P_{m + \frac{1}{2} d_{\beta},\beta}(X) ]^{\vir} \right) \prod_{i=1}^{r} \ch_{2+k_i}(\BF_i) \right]
\]
where we used the diagram
\[
\begin{tikzcd}
P_{n,\beta}(X) \times X^r \ar{r}{\rho} \ar{d}{\pi} & X^r \\
P_{n,\beta}(X)
\end{tikzcd}
\]
and $(\BF_i,\sigma_i)$ is the pullback of the universal stable pair $(\BF,\sigma)$ from $P_{n,\beta}(X) \times X$
via $\pi$ times the projection to the $i$-th factor of $X^r$.

We work here with algebraic cycles, so both series lie in the Chow ring:
\begin{gather*}
\GW_{\beta}^X(\tau_{1} \cdots \tau_{k_r}) \in A^{\ast}(X^r) \otimes \BC((u)) \\
\quad \PT_{\beta}^X(\tau_1 \cdots \tau_{k_r})  \in A^{\ast}(X^r) \otimes \BC((p)).
\end{gather*}

For any set partition $P$ of $\{1, \ldots, r \}$,
and for any set of partitions $\hat{\alpha}_T$ index by $T \in P$,
recall the correspondence cycle defined by
\[
\Gamma_{\alpha,P,\hat{\alpha}}
=
\prod_{T \in P} \pi_{T \sqcup I_T}^{\ast}\left( \pi_1^{\ast}(\K_{\alpha_{T}, \hat{\alpha}_T}) \cdot \Delta_{T \sqcup I_T} \right)
\quad \in A^{\ast}( X^{r + \ell(\hat{\alpha}) }).
\]
We view the cycle as defining a morphism
\begin{gather*}
\Gamma_{\alpha,P, \hat{\alpha}} : A^{\ast}(X^{\ell(\hat{\alpha})}) \to A^{\ast}( X^r ),
\ \gamma \mapsto \rho_{1 \ast}( \rho_1^{\ast}( \gamma) \cdot \Gamma_{\alpha,P, \hat{\alpha}}),
\end{gather*}
where $\rho_1 : X^{r+\ell(\hat{\alpha})} \to X^r$ and $\rho_2 : X^{r+\ell(\hat{\alpha})} \to X^{\ell(\hat{\alpha})}$ are the projections.

\begin{question} \label{Question:Chow theoretic correspondence}
Is it true that we have the following equality
\[
\PT_{\beta}^X(\tau_1 \cdots \tau_{k_r}) 
=
\sum_{P \textup{ set partition of }\{1,\ldots,r\}} 
\sum_{ \substack{\textup{for all }T \in P \\ \textup{a partition } \hat{\alpha}_T }}
\Gamma_{\alpha,P, \hat{\alpha}}\Big( \GW_{\beta}^X\left( { \textstyle \prod_{T} \tau_{[\hat{\alpha}_T]} }\right) \Big)
\]
in $A^{\ast}(X^r)$ ?
\end{question}

For a different type of cycle-theoretic correspondence
comparing homology classes on the Chow variety $\mathrm{Chow}_{\beta}(X)$ of curves on $X$
see \cite[Conjecture 1]{MOOP} (proven there for toric threefolds).
If one is very optimistic, one could even 
combine Question~\ref{Question:Chow theoretic correspondence} with the conjecture of \cite{MOOP}
by asking for a Chow-theoretic correspondence on $\mathrm{Chow}_{\beta}(X) \times X^r$
by pushforward of the relevant cycles along the natural maps:
\begin{gather*}
\Mbar^{\bullet}_{g,r}(X,\beta) \to X^r \times \mathrm{Chow}_{\beta}(X) \\
P_{n, \beta, r}(X) \to X^r \times \mathrm{Chow}_{\beta}(X).
\end{gather*}
Without concrete evidence all this remains pure speculation for now.

\section{Fano complete intersections} \label{sec:Fano complete intersection}
Let $X \subset \p^n$ be a Fano complete intersection.\footnote{The list of these cases is relatively short:
$X$ can be a degree $2,3,4$ hypersurface in $\p^4$, 
a degree $(2,2)$ or $(2,3)$ complete intersection in $\p^5$,
or a degree $(2,2,2)$ complete intersection in $\p^6$.
Nevertheless, the discussion should be viewed as a special case
of a general strategy that can be applied to any situation where vanishing cohomology is present.}
The goal here is to prove the GW/PT correspondence for $X$.
By the Lefschetz hyperplane theorem
the even cohomology of $X$ is generated by the ample class $H = c_1(\CO_X(1))$,
and hence non-vanishing under any projective degeneration.
In particular, we can identify curve classes on $X$ by their degree:
\[ H_2(X,\BZ) \cong \BZ, \quad \beta \mapsto \int_{\beta} H. \]
The odd cohomology is concentrated in the middle degree $H^3(X,\BZ)$.

\subsection{Reduction to diagonal insertions}
\begin{prop} \label{prop:reduction to non-vanishing}
Let $\gamma \in H^{\ast}(X^r)$ and $\beta \geq 0$. The set of all invariants of the form
\begin{equation} \label{Iinvariantsxxx}
Z^{X}_{\PT, \beta}\left( \tau_{k_1} \cdots \tau_{k_r}(\gamma) \right)
\end{equation}
can be effectively reconstructed from the subset of invariants \eqref{Iinvariantsxxx} where
\begin{equation} \gamma = \Delta_{X,12} \ldots \Delta_{X, 2 a-1, 2a} \prod_{i=2a+1}^{r} \pi_i^{\ast}(\gamma_i ) \label{gamma insertion diagonal} \end{equation}
for any {\em even} $\gamma_i \in H^{\ast}(X)$.

Similarly, the set of all invariants
\[ Z^{X}_{\GW, \beta}\left( \overline{ \tau_{k_1} \cdots \tau_{k_r}(\gamma) } \right)  \]
can be effectively reconstructed from the subset of $\gamma$
of the above form and the reconstruction algorithm is the same as for PT invariants.
In particular, the Gromov-Witten/Pandharipande-Thomas correspondence is compatible with this reconstruction,
that is if Theorem~\ref{thm:Fano complete intersection} holds for the invariants \eqref{Iinvariantsxxx}
where $\gamma$ of the form \eqref{gamma insertion diagonal}, then it holds for any $\gamma$.
\end{prop}
\begin{proof}
This proof is exactly as in \cite[Thm.4.27]{ABPZ}
with the obvious modification for notation.
To sketch the idea, let $V := H^3(X,\BC)$ be the middle cohomology endowed with the skew-symmetric inner product
$\alpha \cdot \beta := \int_X \alpha \cup \beta$.
The monodromy group of $X$ (defined here as the group generated by all parallel transport operators)
acts on $V$ and by the results of Deligne \cite{Deligne1,Deligne2} its closure is equal to the symplectic group $\mathrm{Sp}(V)$.
By the deformation invariants of GW and PT invariants and since the Chern classes of $X$ are monodromy invariant
we obtain hence for all $g \in \mathrm{Sp}(V)$ that:
\begin{equation}
\label{mono}
\begin{gathered}
Z^{X}_{\PT, \beta}\left( \tau_{k_1} \cdots \tau_{k_r}(\gamma) \right)
=
Z^{X}_{\PT, \beta}\left( \tau_{k_1} \cdots \tau_{k_r}( (g \otimes \ldots \otimes g)\gamma) \right) \\
Z^{X}_{\GW, \beta}\left( \overline{ \tau_{k_1} \cdots \tau_{k_r}(\gamma) } \right)
=
Z^{X}_{\GW, \beta}\left( \overline{ \tau_{k_1} \cdots \tau_{k_r}( (g \otimes \ldots \otimes g)\gamma) } \right)
\end{gathered}
\end{equation} 

Let $\CB$ be a homogeneous basis of $H^{\ast}(X)$.
Consider a general invariant of the form \eqref{Iinvariantsxxx}
where
$\gamma$ is of the form \eqref{gamma insertion diagonal} but
with $\gamma_i$ arbitrary classes in $\CB$.
We assume that the first $f$ factors of $\gamma_i$ lie in $V$, and the remaining factors lie in $H^{\text{even}}(X)$.
We then consider the invariant \eqref{Iinvariantsxxx} as a function
$Z: V^{\otimes f} \to \BC((p))$
given by
\[ v_1 \otimes \ldots \otimes v_{f} \mapsto 
Z^{X}_{\PT, \beta}\left( \tau_{k_1} \cdots \tau_{k_r}(\gamma) \right)\Big|_{\gamma_{2a+i} = v_i \text{ for } i=1,\ldots,f}.
\]
By the monodromy invariance \eqref{mono} we find that
\[ Z \in ((V^{\ast})^{\otimes f})^{\mathrm{Sp}(V)}. \]
Hence $f=2k$ is even, and by the basic representation theoretic fact explained in \cite[Sec.4.6]{ABPZ}
the class $Z$ is determined by its pairings with all big-diagonal classes:
\[ \sigma( \Delta_V \otimes \Delta_V \otimes \cdots \otimes \Delta_V) \in V^{\otimes 2k}. \]
where $\sigma$ runs over all permutations of $2k$, and
$\Delta_V \in V \otimes V$ is the bi-vector dual to the symplectic pairing $\omega \in V^{\vee} \otimes V^{\vee}$.
We have
\[ \Delta_X = \Delta_V + \frac{1}{\left( {\textstyle \int_X H^3} \right)} \sum_{i=0}^{3} H^i \otimes H^{3-i}. \]
Hence by an induction on $f$ the invariant \eqref{Iinvariantsxxx}
is determined when $\gamma_{2a+1} \otimes \ldots \otimes \gamma_{2a+f}$ runs over the classes
\[ \sigma( \Delta_X \otimes \Delta_X \otimes \cdots \otimes \Delta_X) \in (H^{\ast}(X))^{\otimes 2f}. \]
This completes the first part.
The relations we obtain from this process are identical for the Gromov-Witten invariants,
hence the algorithm is compatible with the GW/PT correspondence.
This proves the second part.
\end{proof}

\subsection{Invertibility}
Let $S$ be a smooth projective surface and let $L \in \Pic(S)$ be a line bundle.
We consider the projective bundle
\[ \p_S = \p( L \oplus \CO) \xrightarrow{\pi_S} S. \]
The projection has two canonical sections
\[ S_{0}, S_{\infty} \subset \p \]
called the zero and infinite section specified by the
condition that the zero section has normal bundle $L^{\vee}$ and
the infinite section has normal bundle $L$.
We use the identification
\[ H_2(\p_S, \BZ) \cong H_2(S,\BZ) \oplus \BZ, \quad \beta \mapsto (\pi_{S \ast}(\beta) , {\textstyle \int_{\beta}} D_{\infty} ). \]
We will need a certainly invertability statement for the GW and PT invariants
of $(\p_S, S_{\infty})$.

Given a $H^{\ast}(S)$-weighted partition 
\[ \mu = (\mu_i, \alpha_i)_{i=1}^{\ell(\mu)}, \quad \mu_i \geq 0, \alpha_i \in H^{\ast}(S) \]
define the monomial of descendents
\[ \tau[\mu] := \prod_{i=1}^{\ell(\mu)} \tau_{\mu_i-1}( \iota_{S_{0} \ast}(\alpha_i) ) \]
where $\iota_{S_0} : S_0 \to \p_S$ is the inclusion. 
Let $\CB$ be a basis of $H^{\ast}(S)$.

\begin{lemma}[{\cite[Proposition 6]{PaPixJap}}] \label{lemma:PaPixJap}
Let $d \geq 1$. The matrix indexed by $\CB$-weighted partitions of $d$ with coefficients
\[ Z^{(\p_S, S_{\infty})}_{\PT, (0,d)}( \lambda | \tau[\mu] ) \]
is invertible.
\end{lemma}
\begin{proof}
The proof in \cite{PaPixJap} is stated for K3 surfaces,
but as noted right before Section 4.2 it works for arbitrary surfaces as well.
\end{proof}

We will require also the analogue of Lemma~\ref{lemma:PaPixJap} on the Gromov-Witten side.

\begin{lemma} \label{lemma:PaPixJap_GW}
The matrix indexed by $\CB$-weighted partitions of $d$ with coefficients
\begin{equation} Z^{(\p_S,S_{\infty})}_{\GW,(0,d)}( \lambda | \tau[\mu] ) \label{matrix_Coeff} \end{equation}
is invertible.
\end{lemma}
\begin{proof}
For $S$ a toric surface or a K3 surface 
the statement follows from Lemma~\ref{lemma:PaPixJap}
and the GW/PT correspondence of \cite[Thm.3]{PaPix_GWPT} or 
\cite[Prop.26]{PP_GWPT_Toric3folds} respectively.
However, it is also not difficult to give a direct argument in the general case, parallel to 
the case of the local geometry $(\BC^2 \times \p^1, \BC^2_{\infty})$ discussed in 
\cite[Sec.2.2]{PP_GWPT_Toric3folds}. 
The idea is to show the matrix is upper-triangular with respect to an ordering on $\CB$-valued partitions
(namely the lexicographic ordering associated to two natural order functions) and then show that the diagonal blocks are invertible.
This proceeds in several steps:

Let us write throughout
\[ \lambda = \big( (\lambda_1, \delta_1) , \ldots, (\lambda_{\ell}, \delta_{\ell} ) \big), \quad \delta_i \in H^{\ast}(S). \]
By assumption $|\lambda| = |\mu| = d$.

\vspace{5pt}
\noindent \emph{Step 1.}
If $\ell(\lambda) < \ell(\mu)$, then the coefficient \eqref{matrix_Coeff} vanishes.

\begin{proof}
Since \eqref{matrix_Coeff} vanish if the degree of the integrand does not match the virtual dimension
we can assume the dimension constraint:
\[ \ell(\lambda) + \ell(\mu) = \sum_i \deg(\gamma_i) + \sum_i \deg(\delta_i). \]

Evaluating a relative stable map $f : C \to \p_S[k]$ at the relative markings
defines an element in $S^{\ell(\lambda)}$. 
For maps $f$ which are incident to topological cycles representing the classes $\delta_i$ and $\gamma_i$,
the $i$-th relative marked point of $f$ has to both lie on the cycle dual to $\delta_i$ as well as on all the cycles dual to $\gamma_{i_j}$, where $i_j$ runs over the interior markings lying on the connected component of the stable map incident to the $i$-th relative marking.
In particular, the point in $S^{\ell(\lambda)}$ must lie on a subspace of dimension
\[ 2 \ell(\lambda) - \sum_i \deg(\delta_i) - \sum_i \deg(\gamma_i) = \ell(\lambda) - \ell(\mu). \]
If $\ell(\lambda) < \ell(\mu)$, then this subspace is empty so \eqref{matrix_Coeff} vanishes.
\end{proof}

\vspace{5pt}
\noindent \emph{Step 2.} If $\ell(\lambda) = \ell(\mu)$ but $\ell_{+}(\lambda) > \ell_{+}(\mu)$, then
\eqref{matrix_Coeff} vanishes.
Here $\ell_{+}(\lambda)$ is the number of parts $\lambda_i$ with $\lambda_i > 1$.

\begin{proof}
By the proof of Step 1 we can represent $\delta_i$ and $\gamma_i$ by topological cycles
such that every stable map incident to the given constraint
has to meet $S_{\infty}$ in a finite number of points $(x_1, \ldots, x_{\ell(\lambda)}) \in S^{\ell(\lambda)}$ where the $x_i$ are pairwise distinct.
It follows that each connected component of the domain of such a stable map
must carry precisely one relative marking, lets say with tangency $k$ to $S_{\infty}$.
The corresponding moduli space (with incidence to the topological cycles imposed)
is then of dimension
$2k - (k-1) - 2 = k-1$. Hence if $k > 1$, the connected component has to carry an interior marked point $q_j$ such that the integrand involves a psi-class $\psi_j$, i.e. such that $\mu_j > 0$.
In particular, there has to be as many $j$ with $\mu_j > 1$ as there are $i$ with $\lambda_i > 1$. This shows the claim.
\end{proof}

\vspace{5pt}
\noindent \emph{Step 3.} 
Let $\CB$ be a homogeneous basis. Let $\lambda$ be a $\CB$ weighted partition and let $\mu$ be a $\CB^{\vee}$-weighted partition such that
$\ell(\lambda) = \ell(\mu)$ and $\ell_{+}(\lambda) = \ell_{+}(\mu)$. Then the series
\[ Z^{(\p_S,S_{\infty})}_{\GW, (0,d)}( \lambda | \tau[\mu] ) 
\]
vanishes unless $\lambda = \sigma(\mu^{\vee})$ for some permutation of $\{ 1, \ldots, n\}$,
in which case it is non-zero.

\begin{proof}
Since every component which carries a relative marking with $\lambda_i > 1$ must also carry an interior marking,
we see for dimension reasons that there exists a permutation $\sigma$ of $\{ 1, \ldots, \ell(\lambda) \}$
such that the connected component carrying the $i$-th relative marking must carry precisely one interior marking $q_{\sigma(i)}$ with $\mu_{\sigma(i)} = \lambda_i$.
Moreover, the number of parts of $\lambda$ labeled $1$ is equal to $\ell(\lambda) - \ell_{+}(\lambda)$,
and hence must be equal to the number of parts of $\mu$. Hence $\sigma( \vec{\mu} ) = \vec{\lambda}$.
This also implies that 
\[ \gamma_{\sigma(i)}^{\vee} = \delta_{i} \]
whenever $\lambda_i > 1$, and by a similar argument also whenever $\delta_i \neq \pt$.
But then for dimension reasons we must have $\lambda = \sigma(\mu^{\vee})$, otherwise the series \eqref{matrix_Coeff} vanishes.
The last step, the non-vanishing now follows from the concrete evaluation of
relative invariants:
\[
\blangle (k, \delta) | \tau_{k-1}(\gamma) \brangle^{(S \times \p^1, S_{\infty})}_{0, (0,n)}
=
\frac{1}{k!} \int_S \delta \cup \gamma,
\]
see \cite[(16)]{PP_GWPT_Toric3folds} and \cite[Lemma 7]{OkPan_Unknot}.
\end{proof}

To conclude the claim, define $\lambda < \mu$ if
\[ \ell(\lambda) < \ell(\mu), \quad \text{ or } \quad \ell(\lambda) = \ell(\mu), \quad \ell_{+}(\lambda) > \ell_{+}(\mu). \]
Then the matrix \eqref{matrix_Coeff} is upper-triangular in this ordering.
By Step 3 the diagonal blocks are conjugate to diagonal matrices with non-zero diagonal entries.
This completes the claim.
\end{proof}

We also recall the following basic result:
\begin{thm}[\cite{PaPix_GWPT}] \label{thm:GWPT for P_S toric K3}
Let $S$ be a toric surface or a K3 surface. Then the GW/PT correspondence
of Conjecture~\ref{conj:GWPT} holds for both
$(\p_S, S_{0,\infty})$ and $(\p_S, S_{\infty})$.
\end{thm}

\begin{proof}
For toric surfaces this is \cite[Theorem 3]{PaPix_GWPT}
and for K3 surfaces this is \cite[Sec.3.8] {PaPix_GWPT}
(note the typo in \cite[Prop.10]{PaPix_GWPT}; the correct statement is  $\gamma_i \in H^{\ast}(\mathbf{P}_S)$).
\end{proof}

\subsection{Proof of Theorem~\ref{thm:Fano complete intersection}}
We only discuss the case where $X$ is a hypersurface in $\p^4$ of degree $d$ for $d \in \{ 1, 2,3,4 \}$.
The general case is completely parallel, see for example \cite[Sec.4.8]{ABPZ}.
We argue by induction on $d$.
In the base case $d=1$ we have $X=\p^3$ which is toric, so the result is known by \cite{PP_GWPT_Toric3folds}.
For $d>1$ following \cite[Sec.0.5.4]{MP} one considers a simple degeneration 
\[ X \rightsquigarrow X_1 \cup_D \tilde{X}_2 \]
where
\begin{enumerate}
\item[(i)] $X_1 \subset \p^4$ is a hypersurface of degree $d-1$,
\item[(ii)] $D \subset \p^3$ is a hypersurface of degree $d-1$, and
\item[(iii)] $\tilde{X}_2$ is the blow-up of $\p^3$ along a complete intersection curve of degree $(d,d-1)$.
\end{enumerate}
We analyze now all terms which appear in the degeneration formula.

\vspace{4pt}
\noindent \textbf{Step 1.}
By induction the GW/PT correspondence is known for $X_1$.
Consider the degeneration to the normal cone
\[ X_1 \rightsquigarrow X_1 \cup_D \p( N_{D/X_1} \oplus \CO). \]
Since there is no vanishing cohomology for this degeneration,
all GW and PT invariants of $X_1$ are determined in terms of the invariants 
of the relative theories of $(X_1,D)$ and $(\p_D, D_{0})$ where $\p_D = \p(N_{D/X_1} \oplus \CO)$.
Following \cite[Sec.7.3]{PaPix_GWPT} we write this schematically as
\[ Z( X_1 ) \quad \rightsquigarrow \quad Z(X_1, D) \text{ and } Z(\p_D, D_0) \]
meaning that it applies both to the GW and PT theory.
By the invertibility statement of Lemma~\ref{lemma:PaPixJap} and Lemma~\ref{lemma:PaPixJap_GW}
and an induction argument
this relation can be inverted (see \cite{PaPix_GWPT}) to give that the GW/PT theories of
$X_1$ and $(\p_D, D_0)$ determine the GW/PT theory of $(X_1, D)$.
We write:
\[ Z(X_1, D) \quad \rightsquigarrow \quad Z(X_1) \text{ and } Z(\p_D, D_0). \]

The GW/PT correspondence is known for $X_1$ by induction.
Since $D$ is either isomorphic to $\p^2$, $\p^1 \times \p^1$ or a cubic surface (which is deformation equivalent to a toric surface)
it is also known for $(\p_D, D_0)$ by Theorem~\ref{thm:GWPT for P_S toric K3}.
By the compatibility of the GW/PT correspondence with the degeneration formula,
we conclude that the GW/PT correspondence holds for $(X_1,D)$ for all cohomology classes.

\vspace{4pt}
\noindent \textbf{Step 2.}
Let $Z_{\PT}^{\Delta}(X_1,D)$ denote the list of all the marked relative PT invariants of the pair $(X_1,D)$
with interior insertions given by big diagonals, i.e. from $DH^{\ast}( (\tilde{X}_2, D)^r )$.
Similarly, let $Z_{\GW}^{\Delta}(X_1,D)$ the corresponding Gromov-Witten invariants.
In Section~\ref{subsec:compatibility splitting formula} we proved that these invariants
are determined by those of $(X_1,D)$ and $(\p_D, D_{0,\infty})$,
\[ Z^{\Delta}(X_1,D) \quad \rightsquigarrow \quad Z(X_1, D) \text{ and } Z(\p_D, D_{0,\infty}). \]
Since the GW/PT correspondence is known for both theories on the right by Step 1 and Theorem~\ref{thm:GWPT for P_S toric K3},
by Proposition~\ref{prop:relGWPT}
we conclude that the generalized GW/PT correspondence of Conjecture~\ref{conj:relGWPT} holds for $Z_{\PT}^{\Delta}(X_1,D)$.

\vspace{4pt}
\noindent \textbf{Step 3.}
By \cite[Sec.7.3.3]{PaPix_GWPT} the GW/PT correspondence (in the form of Conjecture~\ref{conj:GWPT})
is known for $(\tilde{X}_2, D)$ for \emph{all cohomology classes}.
Moreover, by Theorem~\ref{thm:GWPT for P_S toric K3} again the GW/PT correspondence is known for the pair
\[ (\p( N_{D/\tilde{X}_2} \oplus \CO ), D_{0,\infty}). \]
Hence by Proposition~\ref{prop:relGWPT} the generalized GW/PT correspondence of Conjecture~\ref{conj:relGWPT} holds for
$(\tilde{X}_2, D)$ for all insertions from $DH^{\ast}( (\tilde{X}_2, D)^r )$,
\[
Z^{\Delta}(\tilde{X}_2,D) \quad \rightsquigarrow \quad Z(\tilde{X}_2, D) \text{ and } Z(\p( N_{D/\tilde{X}_2} \oplus \CO ), D_{0,\infty}).
\]
We conclude the generalized GW/PT correspondence for $Z^{\Delta}(\tilde{X}_2,D)$.

\vspace{4pt}
\noindent \textbf{Final Step.} 
By applying Proposition~\ref{prop:reduction to non-vanishing} and the degeneration formula for
diagonal insertions given in Proposition~\ref{prop:degeneration formula PT} and~\ref{prop:degeneration formula GW} we obtain the reduction:
\begin{align*}
Z(X) \quad & \rightsquigarrow \quad Z^{\Delta}(X_2) \\
& \rightsquigarrow \quad Z^{\Delta}(X_1, D) \text{ and } Z^{\Delta}(\tilde{X}_2,D).
\end{align*}

Since the GW/PT correspondence is compatible with the first reduction step by Proposition~\ref{prop:reduction to non-vanishing},
with the degeneration formula for marked-relative invariants by Proposition~\ref{prop:Compatibility GWPT with degeneration formula},
and is known for both endpoints of the reduction by the Step 2 and 3 above,
we conclude the GW/PT correspondence for $X$. \qed

\section{Bi-relative residue theory} \label{sec:birelative residue theory}
\subsection{Overview}
Let $R \to B$ be a rational elliptic surface, let $E \subset R$ be a smooth elliptic fiber.
We consider here the relative geometry
\[ (R \times \p^1, R_{\infty}) \]
and the \emph{bi-relative} geometry
\begin{equation} \mathbf{Y} = (R \times \p^1, R \times \{ \infty \} \cup (E \times \p^1) ). \label{3sdfs} \end{equation}

These geometries arise naturally when degenerating $(S \times \p^1, S_{\infty})$ via
the degeneration $S \rightsquigarrow R_1 \cup_{E} R_2$
of an elliptic K3 surface into the union of two rational elliptic surfaces glued along a smooth elliptic fiber,
and will be used in Section~\ref{sec:Proof gwpt k3}.
The full GW and PT theory of the bi-relative geometry \eqref{3sdfs} does not fit the framework of \cite{Li1,Li2,LiWu}.
Fortunately, we only have to consider here capped invariants
for which one only works
with the part of the theory which avoids the intersection $E \times \{ \infty \}$ of the two relative divisors.
The construction of the required moduli spaces for capped invariants
was carried out 'by hand' in Section 5 and 6 of \cite{PaPix_GWPT}
based on the known constructions  \cite{Li1,Li2,LiWu}.
A more general treatment of bi-relative GW and PT theory
uses log geometry. We refer to the recent work \cite{MR}
on log Donaldson-Thomas theory and to \cite{GS,Chen,Rang1} for log Gromov-Witten invariants.

\subsection{Capped descendents} \label{subsec:capped descendents}
Let $S$ be a smooth projective surface (we specialize to a rational elliptic surface later). 
Consider the relative geometry
\[ (S \times \p^1, S_{\infty}). \]
We identify curve classes via the K\"unneth decomposition
\[ H_2(S \times \p^1, \BZ) \cong H_2(S,\BZ) \oplus \BZ [ \p^1 ]. \]
Let $\Gamma = (n, (\beta,d))$ for a curve class $\beta \in H_2(S,\BZ)$. We have
\[ \d_{(\beta,d)} = 2d + \int_{\beta} c_1(T_S). \]

Consider the open subset
\[ U_{\Gamma} \subset P_{\Gamma}( S \times \p^1, S_{\infty}) \]
corresponding to stable pairs which do not carry components of positive $S$-degree in the rubber over $S_{\infty}$.
The open set $U_{\Gamma}$ is invariant under the action of 
the torus
\[ T = \BG_m. \]
which we let act on $\p^1$ with fixed points $0, \infty$ and tangent weight $t$ at $0$.
(Here $t$ is the weight of the standard representation, i.e. $t = c_1(\CO_{B \BG_m}(-1))$.)
The $T$-fixed locus of $U_{\Gamma}^{\BG_m}$ is proper, because it is the union of components of the fixed locus $P_{\Gamma}(S \times \p^1, S_{\infty})^{\BG_m}$
with no $S$-degree over an expansion at $S_{\infty}$.

Given a partition and classes
\[ \mu = (\mu_1, \ldots, \mu_{r}), \quad \Gamma = \gamma_1 \otimes \cdots \otimes \gamma_{r} \in H^{\ast}(S)^{\otimes r}, \]
we define the \emph{$T$-equivariant} insertion:
\begin{align*}
\tau_{\mu}(\Gamma_{0}) 
& = \prod_{i=1}^{\ell(\mu)} \tau_{\mu_i-1}( \gamma_i \cdot [S_0]) \\
& = \left( \tau_{\mu_1-1} \ldots \tau_{\mu_{\ell}-1} \right)\left( \gamma \cdot \pi_1^{\ast}([S_0]) \cdots \pi_{\ell(\mu)}^{\ast}([S_0]) \right)
\end{align*}
where $[S_0] \in H_{T}^{\ast}(S \times \p^1)$ is the class of the fiber over $0 \in \p^1$ in $T$-equivariant cohomology.

\begin{defn}
The capped descendent Pandharipande-Thomas invariants are defined by
\[
C^{(S \times \p^1,S_{\infty})}_{\PT, (\beta,d)}
\left( \lambda \middle| \tau_{\mu}(\Gamma_{0})  \right) \\
=
\sum_{m \in \frac{1}{2} \BZ} i^{2m} p^m
\int_{[ U_{m+\frac{1}{2} \d_{(\beta,d)}, (\beta,d)} ]^{\vir} } (\ev^{\text{rel}}_{S_{\infty}})^{\ast}(\lambda) \tau_{\mu}(\Gamma_{0}) 
\quad \in \BQ(t)((p)).
\]
where $\lambda \in H^{\ast}(S^{[d]})$ and the integral is defined by the virtual localization formula.
\end{defn}

Similarly, consider the open subset of the moduli space of relative stable maps
\[
\widetilde{U}_{g,\vec{\lambda},(\beta,d)} \subset \Mbar_{g,r}(S \times \p^1, S_{\infty}, \vec{\lambda})
\]
corresponding to relative stable maps
for which no connected component of the domain is mapped into a bubble over $S_{\infty}$
with positive degree over $S$.
\begin{defn}
The capped Gromov-Witten invariants lie in $\BQ(t)((u))$ and are defined by
\begin{multline*}
C^{(S \times \p^1, S_{\infty})}_{\GW, (\beta,d)}\left( \lambda \middle| \tau_{\mu}(\Gamma_{0}) \right)  \\
=
(-1)^{\ell(\lambda)} z^{d + \ell(\lambda)}
\sum_{g \in \BZ} (-1)^{g-1} z^{2g-2}
\int_{
[ \widetilde{U}_{g,\vec{\lambda},(\beta,d)} ]^{\vir} }
\prod_i \ev_i^{\ast}(\gamma_i [S_0]) \psi_{i}^{\mu_i-1}  
\prod_{j} ( \ev^{\rel}_{\infty,i} )^{\ast}(\delta_i).
\end{multline*}
where $\lambda = ( \lambda_i, \delta_i )_{i=1}^{\ell(\lambda)}$ is a $H^{\ast}(S)$-weighted partition of size $d$.
\end{defn}

A main idea behind these capped invariants
is that they are the correct linear combination of localization terms for which we have the GW/PT correspondence, see \cite{MOOP}.
We state the result for the rational elliptic surface that we need.

\begin{prop}[{\cite[Prop.3]{PaPix_GWPT}}] \label{prop:R capped GWPT}
Let $R$ be a rational elliptic surface.
We have that the series $C^{(R \times \p^1,R_{\infty})}_{\PT, (\beta,d)}\left( \lambda \middle| \tau_{\mu}(\Gamma_{0})  \right)$ is the Laurent expansion
of an element in $\BC(q,t)$
and
\[
C^{(R \times \p^1,R_{\infty})}_{\PT, (\beta,d)}
\left( \lambda \middle| \tau_{\mu}(\Gamma_{0})  \right)
=
C^{(R \times \p^1, R_{\infty})}_{\GW, (\beta,d)}\left( \lambda \middle| \overline{\tau_{\mu}(\Gamma_{0})} \right)
\]
under the variable change $p=e^{z}$.
\end{prop}
\begin{proof}
This follows directly from Proposition 3 of \cite{PaPix_GWPT}
because any rational elliptic surface is deformation equivalent to a toric surface
(A rational surface is the (iterated) blow-up of $9$ points on $\p^2$
so by moving these $9$ points to the fixed points of the standard $(\BG_m)^2$-action on $\p^2$,
the surface becomes toric).
\end{proof}
\begin{rmk} \label{rmk:corr RxP1, Rinf}
The correspondence $\overline{\tau_{\mu}(\Gamma_{0})}$ is applied here for the relative geometry $(R \times \p^1, R_{\infty})$.
However, since all insertions are supported over $R_0$ away from the relative divisor this correspondence
is equivalent to the correspondence for the absolute geometry $R \times \p^1$ (we may use $c_i(T_{R\times \p^1})$ and $(R \times \p^1)^r$
instead of $c_i(T_{R\times \p^1}[-R_{\infty}])$ and $(R \times \p^1,R_{\infty})^r$)
\end{rmk}

\subsection{Vertex term of the bi-relative theory}
Let $\Gamma = (n, (\beta,d))$ for a curve class $(\beta,d) \in H_2(R \times \p^1,\BZ)$.
Let $d_E = \beta \cdot E$. Let
\[
V_{\Gamma,r} \subset P_{\Gamma,r}(R \times \p^1, E \times \p^1)
\]
be the open locus corresponding to $r$-marked stable pairs on $(R \times \p^1)[k]$
which meet the divisor $\widetilde{R}_{\infty} = p^{-1}(\infty)$ transversely,
where $p : (R \times \p^1)[k] \to R \times \p^1 \to \p^1$ is the projection.
The fixed locus of the $T$-action on $V_{\Gamma,r}$ is compact.
Via intersecting a stable pair with $\widetilde{R}_{\infty}$ we obtain an evaluation morphism
\[ \ev^{\rel}_{\widetilde{R}_{\infty}} : V_{\Gamma,r} \to (R,E)^{[d]}. \]
We also have an evaluation morphism at the relative divisor $E \times \p^1$.
Since stable pairs parametrized by $V_{\Gamma,r}$ never meet this relative divisor at $E \times \infty$,
it takes values in
\[
\ev^{\rel}_{E \times \p^1} : V_{\Gamma,r} \to (E \times \BA)^{[d_e]} \subset (E \times \p^1, E_{\infty})^{[d_e]}.
\]
Finally, we have interior evaluation maps:
\[ \ev: V_{\Gamma,r} \to (R \times \p^1, E \times \p^1)^r \]
obtained from restricting $\ev : P_{\Gamma,r}(R \times \p^1, E \times \p^1) \to (R \times \p^1, E \times \p^1)^r$ to $V_{\Gamma,r}$,

For any $\Gamma \in H^{\ast}( (R,E)^{r} )$ and partition
$\mu = (\mu_1, \ldots, \mu_{r})$ let
\begin{align*}
\tau_{\mu}(\Gamma_0) & := \tau_{\mu_1-1} \ldots \tau_{\mu_r-1}\Big( \gamma \cup \pi_1^{\ast}([R_0]) \cdots \pi_r^{\ast}( [R_0]) \Big).
\end{align*}
where we have suppressed the pullback by $(R \times \p^1, E \times \p^1)^r \to (R,E)^r$.
Consider also classes
\[ \lambda \in H^{\ast}( (R,E)^{[d]} ), \quad \rho \in H_{T}^{\ast}( (E \times \BC)^{[d_E]} ) \]
We define the vertex term
\[ \mathsf{V}_{(\beta,d)}^{\PT}\big( \lambda\, \big| \, \rho \, \big|  \, \tau_{\mu}(\Gamma_0)  \big) 
= \sum_{m} i^{2m} p^m \int_{[ V_{(m+\frac{1}{2} \d_{(\beta,d)}, (\beta,d))} ]^{\vir} } 
\tau_{\mu}(\Gamma_0) \cdot
(\ev^{\rel}_{E \times \p^1})^{\ast}(\rho) \cdot (\ev^{\rel}_{\widetilde{R}_{\infty}})^{\ast}(\lambda)
\]
which lives in $\BC(t)((p))$,
where $i = \sqrt{-1}$ and as in Section~\ref{sec:partition functions} we have taken the holomorphic Euler characteristic of a stable pair to be valued by
\[ n = m + \frac{1}{2} \d_{(\beta,d)} = m + d + \frac{1}{2} \int_{\beta} c_1(T_R). \]

\subsection{Rubber term of the bi-relative theory}
We now consider stable pairs on the relative geometry
\[ (R \times \p^1, R_0 \cup R_{\infty} \cup (E \times \p^1) ) \]
in the curve class $(0,d)$, taken in the rubber sense with respect to the $T$-action on $\p^1$.
Since the stable pairs have degree $0$ over $R$, they never meet the relative divisor $E \times \p^1$,
and the moduli space of these stable pairs, denoted
\begin{equation} P_{(n, (0,d))}( R \times \p^1, R_0 \cup R_{\infty} \cup (E \times \p^1) )^{\sim}, \label{rubber moduli space} \end{equation}
can be constructed with the usual tools, see \cite[Sec.5.4]{PaPix_GWPT} for a discussion.
We have evaluation maps at the relative divisors $R_0$ and $R_{\infty}$,
\begin{gather*}
\ev_{R_0}^{\rel} : P_{(n, (0,d))}( R \times \p^1, R_0 \cup R_{\infty} \cup (E \times \p^1) )^{\sim} \to (R,E)^{[d]}, \\
\ev_{R_{\infty}}^{\rel} : P_{(n, (0,d))}( R \times \p^1, R_0 \cup R_{\infty} \cup (E \times \p^1) )^{\sim} \to (R,E)^{[d]}.
\end{gather*}
Let also $\Psi_{R_0}, \Psi_{R_{\infty}}$ denote the pull-back of the cotangent-line classes
of the divisors $0, \infty \in \p^1$ via the classifying morphism
\[ P_{(n, (0,d))}( R \times \p^1, R_0 \cup R_{\infty} \cup (E \times \p^1) )^{\sim} \to \CT_{(\p^1, 0, \infty)^{\sim}}^{\mathrm{rub}} \]
given by collapsing the $R$-factor. Given classes
\[ \lambda_1, \lambda_2 \in H^{\ast}( (R,E)^{[d]} ) \]
we define the rubber term
\begin{multline*} \mathsf{R}_{(0,d)}^{\PT}\left( \lambda_1, \lambda_2 \, \middle| \, \varnothing \, \middle|  \frac{1}{-\Psi_{D_0} + t}\, \tau_{\mu}(\Gamma_0)  \right) 
= \\
\sum_{m} i^{2m} p^m 
\int_{[ P_{(m+d, (0,d))}( R \times \p^1, R_0 \cup R_{\infty} \cup E \times \p )^{\sim} ]}
(\ev_{D_0}^{\rel})^{\ast}(\lambda_1) 
(\ev_{D_{\infty}}^{\rel})^{\ast}(\lambda_2)
\frac{1}{- \Psi_{D_0} + t}.
\end{multline*}
Here we have the empty insertion $\varnothing$ in the middle factor, because the curve does not meet the divisor $E \times \p^1$.

\subsection{Definition of the bi-relative theory}
Given classes of the form (see Section~\ref{subsec:rel hilb scheme} for the notation)
\begin{gather*}
\lambda = \Nak_{(\lambda_1, \ldots, \lambda_{\ell})}(\delta) \in H^{\ast}( (R,E)^{[d]} ), \quad \delta \in H^{\ast}((R,E)^{\ell}) \\
\rho = \Nak_{(\rho_1, \ldots,\rho_{\ell'})}( \epsilon ) \in H_T^{\ast}( (E \times \BC)^{[d_E]} ), \quad \epsilon \in H^{\ast}_T( (E \times \BC)^{\ell'} ).
\end{gather*}
we define the bi-relative capped descendent residue theory of $\mathbf{Y}$ by
\begin{multline} 
\label{capp def}
C^{\mathbf{Y},\PT}_{(\beta,d)}\big( \lambda\, \big| \, \rho \, \big|  \, \tau_{\mu}(\Gamma_0)  \big)
:=
\mathsf{V}_{(\beta,d)}^{\PT}\big( \lambda\, \big| \, \rho \, \big|  \, \tau_{\mu}(\Gamma_0)  \big)  \\
+ 
\sum_{i}
\mathsf{V}_{(\beta,d)}^{\PT}\big( \phi_i \, \big| \, \rho \, \big|  \, \tau_{\mu}(\Gamma_0)  \big) 
\cdot
\mathsf{R}_{(0,d)}^{\PT}\left( \lambda, \phi_i^{\vee} \, \middle| \, \varnothing \, \middle|  \frac{1}{-\Psi_{D_0} + t}\, \tau_{\mu}(\Gamma_0)  \right)
\end{multline}
where we sum over the K\"unneth decomposition of the diagonal
\[ [\Delta] = \sum_{i} \phi_i \otimes \phi_i^{\vee} \in H^{\ast}( (R,E)^{[d]} \times (R,E)^{[d]}). \]

The definition of these invariants on the Gromov-Witten side goes completely parallel,
we just need to interpret the insertions $\lambda$ and $\rho$ as partitions of $d, d_E$
weighted by the cohomology of $(R,E)^{d}$ and $(E \times \p^1, E_{\infty})$ respectively.
Moreover, the invariants of the genus $g$ invariant, degree $(\beta,d)$ moduli space of stable maps
is weighted according to the conventions of Section~\ref{sec:partition functions}, that is by
\[ (-i)^{\d_{(\beta,d)}} (-1)^{ \ell(\lambda) - |\lambda| + \ell(\rho) - |\rho| }
z^{\d_{(\beta,d)} + \ell(\lambda) - |\lambda| + \ell(\rho) - |\rho| }
(-1)^{g-1} z^{2g-2}. \]
This yields cappend descendent invariants.
\[ C^{\mathbf{Y},\GW}_{(\beta,d)}\big( \lambda\, \big| \, \rho \, \big|  \, \tau_{\mu}(\Gamma_0)  \big) \in \BC(t)((u)). \]

\begin{rmk}
The definition \eqref{capp def} has one more term than the definition in \cite[Sec.5.5]{PaPix_GWPT}.
This is because for $n=d$, the rubber moduli space \eqref{rubber moduli space}
is empty according to our definition (all stable pairs in this class are pulled back from $(R,E)^{[d]}$,
so have infinite automorphism groups). This degenerate term is included in \cite{PaPix_GWPT}
in the definition of $\mathsf{R}$.
\end{rmk}
\begin{rmk}
For K3 surfaces $S$ non-trivial GW and PT invariants for $S \times \p^1$ and class $(\beta,d)$
can only occur in case $\beta=0$ and where the Euler characteristic $n$ is minimal (see Section~\ref{sec:GWPT K3} below).
One has something similar for rational elliptic surfaces, see \cite[Proof of Lemma 27]{OPix2}.
Using a log-symplectic form, one can show that all invariants of $(R \times \p^1, E \times \p^1)$
vanish for all curve class $(\beta,d)$ where $\beta \cdot E = 0$ and the Euler characteristic $n$ is not minimal.
In particular, for the rubber term the only term that can contribute is where $n$ is minimal,
that is $n=d$, but here the moduli space is empty. Thus we see that:
\[
\mathsf{R}_{(0,d)}^{\PT}\left( \lambda_1, \lambda_2 \, \middle| \, \varnothing \, \middle|  \frac{1}{-\Psi_{D_0} + t}\, \tau_{\mu}(\Gamma_0)  \right) 
=
0.
\]

However, on the Gromov-Witten side, the rubber term is non-trivial
and to make our discussion as parallel as possible we have included this term also on the PT side.
(We will do the same in Section~\ref{sec:capped descendents K3} below, when we talk about capped descendents of $(S \times \p^1, S_{\infty})$.)
\end{rmk}

\subsection{GW/PT correspondence}
Let $B$ be the class of a section of $R \to \p^1$, and let $F$ be the fiber class.
We specialize the discussion now to the curve class
\[ \beta_h = B+hF \in H_2(R,\BZ). \]
Then (see e.g. \cite{BL}) every curve on $R$ in class $\beta_h$ meets the elliptic fiber
in the point $0_E := B \cap E$.
Therefore the stable pairs parametrized by
\[ V_{n, (\beta_h,d),r} \subset P_{n, (\beta_h,d),r}(R \times \p^1, E \times \p^1) \]
meet the relative divisor $E \times \p^1$ in the single point $(0_E, 0_{\p^1})$.
We hence specialize our insertion on the divisor $(E \times \p^1,E_{\infty})$ to be
\[ \rho = 1. \]
We have the following:

\begin{thm}  \label{thm:GWPT for birelative theory}
For any insertions
$\lambda = \Nak_{(\lambda_1, \ldots, \lambda_{\ell})}(\delta)$
and $\tau_{\mu}(\Gamma_0)$ with
\begin{equation} \label{delta Gamma}
\begin{gathered}
\delta = \Delta_{(R,E),12}^{\rel} \cdots \Delta_{(R,E),2a-1,2a}^{\rel} \prod_{i=2a+1}^{\ell} \pi_{i}^{\ast}(\delta_i) \\
\Gamma = \Delta_{(R,E),12}^{\rel} \cdots \Delta_{(R,E),2b-1,2b}^{\rel} \prod_{i=2b+1}^{r}  \pi_i^{\ast}(\gamma_i )
\end{gathered}
\end{equation}
for $\delta_i, \gamma_i \in H^{\ast}(R)$ and $a,b \geq 0$, we have that
$C^{\mathbf{Y},\PT}_{(\beta_h,d)}\big( \lambda\, \big| \, 1 \, \big|  \, \tau_{\mu}(\Gamma_0)  \big)$
lies in $\BC(p,t)$ and
\[
C^{\mathbf{Y},\PT}_{(\beta_h,d)}\big( \lambda\, \big| \, 1 \, \big|  \, \tau_{\mu}(\Gamma_0)  \big)
=
 C^{\mathbf{Y},\GW}_{(\beta_h,d)}\big( \lambda\, \big| \, 1 \, \big|  \, \overline{ \tau_{\mu}(\Gamma_0) } \big)
\]
under the variable change $p=e^{z}$.
\end{thm}

\begin{rmk}
Theoretically one should use here the generalized correspondence of Section~\ref{sec:generalized correspondence}
for the birelative geometry $(R \times \p^1, R_{\infty} \cup (E \times \p^1))$.
However, because all descendent insertions are supported over $R_0$
it suffices to use the correspondence for $(R \times \p^1, E \times \p^1)$,
see also Remark~\ref{rmk:corr RxP1, Rinf}
\end{rmk}

\begin{proof}
Assume first that $\lambda, \delta$ are of the form
\[
\delta = \prod_{i=1}^{\ell} \pi_i^{\ast}(\delta_i), \quad \Gamma = \prod_{i=1}^{r} \pi_i^{\ast}( \gamma_i ).
\]
so that in particular
\[
\tau_{\mu}(\Gamma_0) = \prod_{i=1}^{r} \tau_{\mu_{i}-1}( [R_0] \cdot \gamma_i ).
\]
By Proposition~\ref{prop: comparision}
the capped bi-relative residue invariants
\[ C^{\mathbf{Y},\PT/\GW}_{(\beta,d)}\big( \lambda\, \big| \, \rho \, \big|  \, \tau_{\mu}(\Gamma_0)  \big) \]
specialize precisely to the theory discussed in \cite[Section 5.8]{PaPix_GWPT}.
Moreover, the claimed correspondence is then precisely
Conjecture 5 of \cite{PaPix_GWPT}
for $S$ the rational elliptic surface, in the case of curve class $\beta_h$.

Consider the cappend descendent series of the relative geometry $(R \times \p^1, R_{\infty})$,
Consider the degeneration to the normal cone of the elliptic fiber $E \subset R$,
\[ R \rightsquigarrow R \cup_{E} Q, \quad Q = E \times \p^1. \]
where we take $E \subset Q$ to be the fiber over $0_{\p^1} \in \p^1$. We use the K\"unneth decomposition
\[ H_2(Q,\BZ) \cong H_2(E, \BZ) \oplus H_2(\p^1,\BZ) \cong \BZ \oplus \BZ. \]
By the degeneration formula for cappend descendent invariants
in \cite[Sec.5.8]{PaPix_GWPT} we have
\begin{multline} \label{deg bir}
C^{(R \times \p^1, R_{\infty}),\PT}_{(\beta_h,d)}\big( \lambda\, \big|  \, \tau_{\mu}(\Gamma_0)  \big)
= \\
t \sum_{h=h_1+h_2} \sum_{\substack{ \{ 1 , \ldots, \ell \} = A_1 \sqcup A_2 \\ \{ 1, \ldots, r \} = B_1 \sqcup B_2 }}
C^{(R \times \p^1, R_{\infty} \cup (E \times \p^1)),\PT}_{(\beta_{h_1},|\lambda_{A_1}|)}\left( \lambda_{A_1}\, \middle| 1 \middle| \, \prod_{i \in B_1} \tau_{\mu_i-1}( [R_0] \gamma_i ) \right) \\
\times C^{(Q \times \p^1, Q_{\infty} \cup (E \times \p^1)),\PT}_{((h_2,1),|\lambda_{A_2}|)}\left( \lambda_{A_2}\, \middle| (1, [0_E]) \middle| \,
\prod_{i \in B_2} \tau_{\mu_i-1}( [Q_0] \gamma_i ) \right),
\end{multline}
where the $t$ factor comes from the restriction $\Delta_{(\p^1, \infty)}$ to the point $(0_{\p^1}, 0_{\p^1}) \in (\p^1,\infty)^2$.

In Section 6.6 of \cite{PaPix_GWPT},
precisely the bi-relative residue theory of
\[ (Q \times \p^1, Q_{\infty} \cup (E \times \p^1)) \]
was proven to satisfy the GW/PT correspondence.\footnote{In fact, we only need Conjecture 5 of \cite{PaPix_GWPT}
for even cohomology on the elliptic curve, which is proven in \cite[Sec.6.4]{PaPix_GWPT}.}
Moreover, by Proposition~\ref{prop:R capped GWPT}
the GW/PT correspondence holds for the capped theory:
\[ C^{(R \times \p^1, R_{\infty}),\PT}_{(\beta_h,d)}\big( \lambda\, \big|  \, \tau_{\mu}(\Gamma_0)  \big). \]
We can write the degeneration \eqref{deg bir} as
\begin{multline*}
C^{(R \times \p^1, R_{\infty}),\PT}_{(\beta_h,d)}\big( \lambda\, \big|  \, \tau_{\mu}(\Gamma_0)  \big)
= \\
t 
C^{(R \times \p^1, R_{\infty}),\PT}_{(\beta_h,d)}\big( \lambda\, \big| \, 1 \, \big|  \, \tau_{\mu}(\Gamma_0)  \big)
\times C^{(Q \times \p^1, Q_{\infty} \cup E \times \p^1),\PT}_{((0,1),0)}\left( 1 \, \middle| (1, [0_E]) \middle| \,
\right) 
+ \ldots
\end{multline*}
where the dots stand for terms which are lower in
an ordering by a lexicographic ordering on $h$, $d$, and the degree of $\tau_{\mu}(\Gamma_0)$.
It is straightforward to see that\footnote{For the minimal Euler characteristic
$n=1$ the moduli space is isomorphic to $E \times \p^1$,
and the equivariant integral over $[0_E]$ yields the invariant $\frac{1}{t}$.}
\[ 
C^{(Q \times \p^1, Q_{\infty} \cup E \times \p^1),\PT}_{(0,1,0)}\left( 1 \, \middle| (1, [0_E]) \middle| \,
\right)  \neq 0. \]
A similar degeneration applies on the Gromov-Witten side.
By the compatibility of the GW/PT correspondence with the degeneration formula
(as stated in \cite[Sec.5.8, p.436]{PaPix_GWPT}) we conclude that the claim holds.

We now consider the case of $\delta$ and $\Gamma$ as in the \eqref{delta Gamma}.
Whenever all $\Gamma$ are of the form $\pi_1^{\ast}(\gamma_1) \cdots \pi_r^{\ast}(\gamma_r)$,
we know that the GW/PT correspondence holds for the cappend descendent residue invariants of
$(Q \times \p^1, Q_{\infty} \cup (E \times \p^1))$ by \cite[Sec.6.6]{PaPix_GWPT},
and also for $\mathbf{Y}$ in curve classes $\beta_h$ by the above.
The general case hence follows by the compatibility of the splitting formula with the generalized GW/PT correspondence
proven in Proposition~\ref{prop:relGWPT}: Indeed, the insertion $\tau_{\mu}(\Gamma_0)$
is supported over $R_0$ and hence does not interact with the divisor $R_{\infty}$,
hence the splitting formula only takes place relative to the single divisor $E \times \p^1$.
(This is also reflected that the relative part of our insertions is pulled back from $(R,E)^{r}$.)
\end{proof}

\section{$K3 \times \mathrm{Curve}$: Primitive case} 
\label{sec:GWPT K3}
Let $S$ be a smooth projective K3 surface, let
$C$ be a smooth curve and let $z=(z_1, \ldots, z_N)$ be a tuple of distinct points $z_i \in C$.
We consider the 
relative geometry
\begin{equation} (S \times C, S_z), \quad S_{z} = \bigsqcup S \times \{ z_i \}. \label{relative geometry} \end{equation}
The curve classes will be denoted by
\[ (\beta, d) = \iota_{\ast} \beta + d [C] \in H_2(S \times C, \BZ) \cong H_2(S,\BZ) \oplus \BZ [C]. \]
We have
\[ \d_{(\beta,d)} = \int_{(\beta,d)} c_1(T_{S \times C}) = d (2-2g(C)). \]

\subsection{Pandharipande-Thomas theory} \label{sec:K3xC PT theory}
Let $\Gamma = (n, (\beta,d))$ and consider the moduli space $P_{\Gamma}(S \times C,S_z)$ of relative stable pairs on \eqref{relative geometry},
which is of virtual dimension
\[ \vd = \d_{(\beta,d)} = 2 d (1-g(C)). \]
Since the K3 surface $S$ carries a everywhere non-vanishing holomorphic $2$-form,
the obstruction sheaf on $P_{\Gamma}(S \times C,S_z)$ admit
naturally a cosection.
We refer to \cite{MPT} for the construction in case $\beta \neq 0$ and to \cite[Sec.4.3]{PaPixJap} in case $\beta =0$,
see also the recent treatment of Nesterov \cite{N2}.
As these references show, the cosection can be choosen to be everywhere surjective
unless we are in the minimal case
\begin{equation} \beta = 0, \quad n = \frac{1}{2} \d_{(\beta,d)} = d ( 1-g(C)), \label{minimal case} \end{equation}
in which case the moduli space parametrizes stable pairs which are obtained by pulling back a length $d$ subscheme of $S$.
In the non-minimal case $(\beta,n) \neq (0, d (1-g(C)))$
we obtain by the results of Kiem and Li \cite{KL}
the the standard virtual class vanishes,
\[
[ P_{\Gamma}(S \times C,S_z) ]^{\vir} = 0,
\]
as well as that there exists a reduced virtual fundamental class:
\[
[ P_{\Gamma}(S \times C,S_z) ]^{\red} \in A_{\mathsf{rvd}}(P_{\Gamma}(S \times C,S_z)), \quad \mathsf{rvd} = \vd + 1.
\]

We will need here both the standard and the reduced virtual class.
The reasons is that they interact with each other in the degeneration formula, the splitting formulas, etc.
To capture this interaction it is handy to introduce a formal variable $\epsilon$ living in the ring of dual numbers
\[ \BQ[\epsilon]/ (\epsilon^2), \]
see \cite{HilbK3} for the origins of the idea.
We also extend the definition of the reduced virtual class to the minimal case as follows:
\[ [ P_{\Gamma}(S \times C,S_z) ]^{\red} := 0, \quad \text{ if } (\beta,n) = (0, d(1-g(C))). \]
One defines the $\full$ virtual fundamental class for all $n,\beta$ by:
\[
[ P_{\Gamma}(S \times C,S_z) ]^{\full} := [ P_{\Gamma}(S \times C,S_z) ]^{\vir} + \epsilon [ P_{\Gamma}(S \times C,S_z) ]^{\red}
\]
in $A_{\ast}(P_{\Gamma}(S \times C,S_z)) \otimes \BQ[\epsilon]/\epsilon^2$.
By a basic check (see e.g. \cite{MPT, HilbK3} and Remark~\ref{rmk:epsilon} below) 
the full virtual class satisfies exactly the same splitting rules and degeneration formulas as the standard virtual class.
We will denote the Pandharipande-Thomas brackets, partitions functions
defined using the reduced virtual class  $[ - ]^{\red}$ and the full virtual class $[ - ]^{\full}$ 
with the supscripts $\red$ and $\full$ respectively.
If we want to stress that our invariants are defined using the standard virtual class, we will add the supscript $\vir$.

\begin{rmk} \label{rmk:epsilon}
Let $\pi : \CS \to \p^1$ be the (non-algebraic) twistor family of the K3 surface $S$,
and let $\BE = \pi_{\ast} \Omega_{\CS/\p^1}$ be the Hodge bundle with fiber $H^0(S_t, \Omega_{S_t}^2)$. 
Then we may identify $\epsilon$ with the first Chern class $c_1(\BE) \in A^1(\p^1)$.
Under this identification,
the full virtual class can then be viewed as the virtual class of the moduli space of stable pairs on the total space of the family $\CX \times C \to \p^1$.
The invariants are defined by pushforward to $\p^1$ and hence take value in $\BQ[\epsilon]/\epsilon^2$.
This gives a conceptual proof for the 'basic check' alluded to above,
and also connects to the origins of the reduced Gromov-Witten invariants in the work of Bryan and Leung \cite{BL},
see also \cite{KT1} for an algebraic treatment. \qed
\end{rmk}

The partition functions take on the usual form:
Since the relative divisor $S_z$ has $N$ connected components,
we need to specify the relative conditions
\[ \lambda_1, \ldots, \lambda_N \in H^{\ast}(S^{[d]}). \]
Let also $\gamma \in H^{\ast}( (S \times C, S_z)^r )$ be a class.
We then define:
\begin{multline*}
Z^{(S \times C,S_z), \vir/\red/\full}_{\PT, (\beta,d)}
\left( \lambda_1, \ldots, \lambda_N \middle| \tau_{k_1} \cdots \tau_{k_r}(\gamma) \right) \\
=
\sum_{m \in \BZ} (-p)^m
\big\langle \, \lambda_1, \ldots, \lambda_N \, \big| \, \tau_{k_1} \cdots \tau_{k_r} (\gamma) \big\rangle^{(S \times C,S_z), \PT, \vir/\red/\full}_{m + d(1-g(C)),(\beta,d)}.
\end{multline*}
Since for the $\vir$-invariants, only the minimal degree $(\beta,n) = (0,\frac{1}{2} \d_{\beta})$ contributes,
we have
\[
Z^{(S \times C,S_z), \vir}_{\PT, (\beta,d)}
\left( \lambda_1, \ldots, \lambda_N \middle| \tau_{k_1} \cdots \tau_{k_r}(\gamma) \right) \in \BQ.
\]

To give an example how to work with the full invariants, we state the degeneration formula:
Let $C \rightsquigarrow C_1 \cup_x C_2$ be a degeneration of $C$.
Let
\[ \{ 1, \ldots, N \} = A_1 \sqcup A_2 \]
be a partition of the index set of relative divisors,
and write $z(A_i) = \{ z_j | j \in A_i \}$.
We assume that the points in $A_i$ specialize to the curve $C_i$ disjoint from $x$.
\begin{prop} \label{prop: K3 degeneration formula}
For any curve class $\beta \in H_2(S,\BZ)$ we have:
\begin{gather*}
Z^{(S \times C, S_z), \full}_{\PT, (\beta,n)}\left( \lambda_1, \ldots, \lambda_N \middle| \prod_{i} \tau_{k_i}(\alpha_i) \right) =
\sum_{ \{ 1, \ldots, r \} = B_1 \sqcup B_2 } \sum_{\beta = \beta_1 + \beta_2} \\
Z^{(S \times C_1, S_{z(A_1),x}), \full}_{\PT,(\beta_1,n)}\left( \prod_{i \in A_1} \lambda_i, \Delta_1 \middle| \prod_{i \in B_1} \tau_{k_i}( \alpha_i ) \right) 
Z^{(S \times C_2, S_{z(A_2),x}), \full}_{\PT,(\beta_2,n)}\left( \prod_{i \in A_2} \lambda_i, \Delta_2 \middle| \prod_{i \in B_2} \tau_{k_i}( \alpha_i ) \right).
\end{gather*}
where $(\Delta_1, \Delta_2)$ stands for summing over the K\"unneth decomposition of the diagonal class $\Delta_{S^{[n]}} \in H^{\ast}((S^{[n]})^2)$.
\end{prop}
\begin{rmk}
By taking the $\epsilon^0$ and $\epsilon^1$ coefficient the degeneration formula for the standard and reduced invariants respectively.\\

\noindent 
\textbf{Non-reduced case:}
\begin{gather*}
Z^{(S \times C, S_z), \vir}_{\PT, (0,n)}\left( \lambda_1, \ldots, \lambda_N \middle| \prod_{i} \tau_{k_i}(\alpha_i) \right) =
\sum_{ \{ 1, \ldots, r \} = B_1 \sqcup B_2 } \\
Z^{(S \times C_1, S_{z(A_1),x}), \vir}_{\PT,(0,n)}\left( \prod_{i \in A_1} \lambda_i, \Delta_1 \middle| \prod_{i \in B_1} \tau_{k_i}( \alpha_i ) \right) 
Z^{(S \times C_2, S_{z(A_2),x}), \vir}_{\PT, (0,n)}\left( \prod_{i \in A_2} \lambda_i, \Delta_2 \middle| \prod_{i \in B_2} \tau_{k_i}(  \alpha_i ) \right)
\end{gather*}
\textbf{Reduced case:}
\begin{gather*}
Z^{(S \times C, S_z), \red}_{\PT, (\beta,n)}\left( \lambda_1, \ldots, \lambda_N \middle| \prod_{i} \tau_{k_i}(\alpha_i) \right) = 
\sum_{ \{ 1, \ldots, r \} = B_1 \sqcup B_2 } \Bigg( \\
\phantom{+}
Z^{(S \times C_1,S_{z(A_1),x}), \red}_{(\beta,n)}\left( \prod_{i \in A_1} \lambda_i, \Delta_1 \middle| \prod_{i \in B_1} \tau_{k_i}( \alpha_i ) \right)
Z^{(S \times C_2,S_{z(A_2),x}), \vir}_{(0,n)}\left( \prod_{i \in A_2} \lambda_i, \Delta_2 \middle| \prod_{i \in B_2} \tau_{k_i}(  \alpha_i ) \right) \phantom{\Bigg)}\\
+
Z^{(S \times C_1,S_{z(A_1),x}),\vir}_{(0,n)}\left( \prod_{i \in A_1} \lambda_i, \Delta_1 \middle| \prod_{i \in B_1} \tau_{k_i}( \alpha_i ) \right)
Z^{(S \times C_2,S_{z(A_2),x}),\red}_{(\beta,n)}\left( \prod_{i \in A_2} \lambda_i, \Delta_2 \middle| \prod_{i \in B_2} \tau_{k_i}(  \alpha_i ) \right) \Bigg).
\end{gather*}
\qed
\end{rmk}

\subsection{Gromov-Witten theory}
\label{Section:Relative_Gromov_Witten_theory_of_P1K3}
For $i \in \{ 1, \ldots, N \}$, consider $H^{\ast}(S)$-weighted partitions
\[ \lambda_i =  \big((\lambda_{i,j}, \delta_{i,j}) \big)_{j=1}^{\ell(\lambda_i)} \]
of size $d$ with underlying partition $\vec{\lambda}_i$.

The discussion of the reduced classes on the Gromov-Witten side is very similar to the stable pairs case.
By \cite{MP, MPT} the obstruction sheaf of the moduli space of relative stable maps
$\Mbar_{g,r,(\beta,d)}(S \times C, S_{z}, \vec{\lambda}_i)$
has a surjective cosection whenever $\beta > 0$.
By \cite{KL} we obtain a reduced virtual class, and reduced invariants.
For $\beta = 0$ we define the reduced virtual class to be the zero class.
The full invariants are then identical to before,
and satisfy the same degeneration formula and splitting rules as ordinary Gromov-Witten invariants.

The partition function takes the usual form:
\begin{multline}
Z^{(S \times C, S_z), \vir/\red/\full}_{\GW, (\beta,d)}\left( \lambda_1, \ldots, \lambda_N \middle| \tau_{k_1} \cdots \tau_{k_r}(\gamma) \right)  \\
=
(-1)^{d(1-g(C)) + \sum_i (\ell(\lambda_i) - |\lambda_i|)}
z^{2d(1-g(C)) + \sum_i ( \ell(\lambda_i) - |\lambda_i|)} \\
\sum_{g \in \BZ} (-1)^{g-1} z^{2g-2}
\left\langle \, \lambda_1, \ldots, \lambda_N \, \middle| \, \tau_{k_1} \ldots \tau_{k_r}(\gamma) \right\rangle^{(S \times C,S_z), \bullet, \vir/\red/\full}_{g, (\beta,d)}
\end{multline}
corresponding to the standard (i.e. non-reduced), the reduced, or the full invariants.

\subsection{GW/PT correspondence (primitive case)} \label{subsec:GWPT K3}
Given a class $\alpha \in H^{\ast}(S)$ and $\alpha' \in H^{\ast}(C)$ we use the shorthand
\[ \alpha \cdot \alpha' := \pi_{S}^{\ast}(\alpha) \cdot \pi_{C}^{\ast}(\alpha') \]
where $\pi_S, \pi_C$ the projections of $S \times C$ to the factors.
Let also
\[ \omega \in H^2(C,\BZ) \]
be the class of a point.

We will prove the following GW/PT correspondence in the primitive case:

\begin{thm} \label{thm:GWPT K3 general} 
Let $\lambda_1, \dots, \lambda_N$ be $H^{\ast}(S)$-weighted partition, let $\alpha_i \in H^{\ast}(S)$, and let $\beta \in H_2(S,\BZ)$ be a primitive effective curve class.
Then
\[
Z^{(S \times C,S_{z}),\red}_{\PT, (\beta,n)}\left( \lambda_1, \ldots, \lambda_N \middle| \prod_{i} \tau_{k_i}(\alpha_i \omega) \right)
\]
is the Laurent expansion of a rational function in $p$ and we have that
\[ 
Z^{(S \times C,S_{z}),\red}_{\PT, (\beta,n)}\left( \lambda_1, \ldots, \lambda_N \middle| \prod_{i} \tau_{k_i}(\alpha_i \omega) \right)
=
Z^{(S \times C,S_{z}),\red}_{\GW, (\beta,n)}\left( \lambda_1, \ldots, \lambda_N \middle| \overline{\prod_{i} \tau_{k_i}(\alpha_i \omega)} \right)
\]
under the variable change $p=e^{z}$.
\end{thm}

\section{Proof of Theorem~\ref{thm:GWPT K3 general}} \label{sec:Proof gwpt k3}
\subsection{Strategy}
For the proof of Theorem~\ref{thm:GWPT K3 general} we will argue in the following 4 steps:
\begin{enumerate}
\item By the invertibility (Lemma~\ref{lemma:PaPixJap}, Lemma~\ref{lemma:PaPixJap_GW}) and the degeneration formula
we reduce to the case $(S \times \p^1, S_{\infty})$.
\item By a dimension induction argument we reduce to the GW/PT correspondence for capped descendents of $(S \times \p^1, S_{\infty})$.
\item Use the monodromy of the K3 surface $S$ and the arguments of \cite{ABPZ} 
we reduce to capped descendent invariants on an elliptic K3 surface $S$ where all the insertions
given by diagonal, unit, point, section or fiber classes.
\item We apply the degeneration formula
for the degeneration $S \rightsquigarrow R_1 \cup_E R_2$ of $S$ into the union of two rational elliptic surfaces glued along an elliptic curve.
We are reduced to the marked relative invariants on the birelative residue theory of
$(R \times \p^1, R_{\infty} \cup E \times \p^1)$ for $R \in \{ R_1, R_2 \}$,
which have been studied in Section~\ref{sec:birelative residue theory}.
\end{enumerate}

\subsection{Step 1: Reduction to $(S \times \p^1, S_{\infty})$} \label{sec:induction scheme}
Let $\beta \in H_2(S,\BZ)$ be an effective curve class and assume first that $g(C) > 0$ or $N>0$.
All Gromov-Witten and Pandharipande-Thomas partition functions of the form
\begin{equation} Z^{(S \times C,S_{z}), \red}_{(\beta,d)}(I | \lambda_1, \ldots, \lambda_N ),
\quad I = \prod_{i=1}^{\ell} \tau_{k_i}( \omega \alpha_i )
\label{Z general}
 \end{equation}
can be expressed in terms of the invariants of $(S \times \p^1, S_{\infty})$ of the form
\begin{equation} Z^{(S \times \p^1,S_{\infty}),\red}_{(\beta,d)}(I' | \lambda' ), \quad I' = \prod_{i} \tau_{k'_i}( \omega \alpha'_i ), \label{Z cap} \end{equation}
where we have used again the convention that if we do not have a $\GW$ or $\PT$
subscript, the claim should hold for both theories.
The idea for this reduction goes back to work of Okounkov and Pandharipande, e.g. \cite{OPLocal, OkPandVir}.

\vspace{6pt}
\noindent
\textbf{Reduction scheme:}
We reduce the general invariants \eqref{Z general} to invariants \eqref{Z cap}
by induction on the genus $g(C)$ of the curve $C$ and the number of relative markings $k$, ordered lexicographically.
If $g(C) > 0$ we degenerate $C$ to a curve with a single node, and apply the degeneration formula in this case. This reduces the genus.
If $g(C) = 0$ and $k \geq 2$, consider the invariant
\begin{equation} \label{3wefsf}
Z^{(S \times \p^1,S_{z_1, \ldots, z_{k-1}})}_{(\beta,d)}(I \cdot \tau[\lambda_k] | \lambda_1, \ldots, \lambda_{k-1} ).
\end{equation} 
where we let 
\[ \tau[\lambda] := \prod_{i=1}^{\ell(\lambda)} \tau_{\lambda_i-1}( \delta_i \cdot \omega ) \]
for an arbitrary $H^{\ast}(S)$-weighted partition $\lambda = (\lambda_i, \delta_i)_{i=1}^{\ell(\lambda)}$.

The invariant \eqref{3wefsf} is known to reduce to $(S \times \p^1, S_{\infty})$
by the induction hypothesis. Applying the degeneration formula in the reduced case (Proposition~\ref{prop: K3 degeneration formula}) yields:
\begin{gather*}
Z^{(S \times \p^1,S_{z_1, \ldots, z_{k-1}}),\red}_{(\beta,d)}(I \tau[\lambda_k] | \lambda_1, \ldots, \lambda_{k-1} )
= \\
  Z^{(S \times \p^1,S_{z_1, \ldots, z_{k-1},x}), \red}_{(\beta,d)}(I | \lambda_1, \ldots, \lambda_{k-1}, \Delta_1 ) Z^{(S \times \p^1,S_{x}), \vir}_{(0,d)}    (\tau[\lambda_k] | \Delta_2 ) \\
+ Z^{(S \times \p^1,S_{z_1, \ldots, z_{k-1},x}), \vir}_{(0,d)}    (I | \lambda_1, \ldots, \lambda_{k-1}, \Delta_1 ) Z^{(S \times \p^1,S_{x}), \red}_{(\beta,d)}(\tau[\lambda_k] | \Delta_2 ).
\end{gather*}
Subtracting the second term on the right of the equality,
and using the invertibility of the relative cap matrix (Lemma~\ref{lemma:PaPixJap} and Lemma \ref{lemma:PaPixJap_GW} for the GW and PT side respectively)
we see that
\eqref{Z general} is a ($\BQ$-linear for PT, and $\BQ((u))$-linear for GW) combination of terms which are lower in the ordering.
Since the base case is $(g(C),k) = (0,1)$, this concludes the scheme. \qed

We obtain the following:

\begin{prop}  \label{prop:reduction to SxP1}
If Theorem~\ref{thm:GWPT K3 general} holds for $(S \times \p^1, S_{\infty})$,
then it holds for all geometries $(S \times C, S_z)$.
\end{prop}
\begin{proof}
If $g(C) > 1$ or $N>0$ this follows by the compatibility of the GW/PT correspondence with the degeneration formula (Proposition~\ref{prop:Compatibility GWPT with degeneration formula})
and the above reduction scheme.
For the case of the absolute theory of $S \times \p^1$ we apply the degeneration formula for the degeneration $\p^1 \rightsquigarrow \p^1 \cup \p^1$,
which reduces us again to $(S \times \p^1, S_{\infty})$.
\end{proof}

\subsection{Step 2: Reduction to capped invariants of $(S \times \p^1, S_{\infty})$} \label{sec:capped descendents K3}
Let $\Gamma = (n, (\beta,d))$ for $\beta \in H_2(S,\BZ)$.
We define the capped descendent invariants as in Section~\ref{subsec:capped descendents}. 
Concretely,
for both $\star \in \{ \vir, \red \}$ we let
\[
C^{(S \times \p^1,S_{\infty}), \star}_{\PT, (\beta,d)}
\left( \lambda \middle| \tau_{\mu}(\Gamma_{0})  \right) \\
=
\sum_{m \in \BZ} (-p)^m
\int_{[ U_{m+d, (\beta,d)} ]^{\star} } (\ev^{\text{rel}}_{S_{\infty}})^{\ast}(\lambda) \tau_{\mu}(\Gamma_{0}) 
\quad \in \BQ(t)((p)).
\]
on the Pandharipande-Thomas side, and 
\begin{multline*}
C^{(S \times \p^1, S_{\infty}), \star}_{\GW, (\beta,d)}\left( \lambda \middle| \tau_{\mu}(\Gamma_{0}) \right)  \\
=
(-1)^{\ell(\lambda)} z^{d + \ell(\lambda)}
\sum_{g \in \BZ} (-1)^{g-1} z^{2g-2}
\int_{
[ \widetilde{U}_{g,\vec{\lambda},(\beta,d)} ]^{\star} }
\prod_i \ev_i^{\ast}(\gamma_i [S_0]) \psi_{i}^{\mu_i-1}  
\prod_{j} ( \ev^{\rel}_{\infty,i} )^{\ast}(\delta_i).
\end{multline*}
on the Gromov-Witten side.

In Section~\ref{sec:conclusion} below we will prove the following:
\begin{thm} \label{thm: capped GWPT}
For $\star \in \{ \vir, \red \}$ we have that
\[ C^{(S \times \p^1,S_{\infty}), \star}_{\PT, (\beta,d)} \left( \lambda \middle| \tau_{\mu}(\Gamma_{0}) \right) \]
is the Laurent expansion of a rational function of $p,t$ (i.e. lies in $\BQ(p,t)$) and
\[
C^{(S \times \p^1,S_{\infty}), \star}_{\PT, (\beta,d)} \left( \lambda \middle| \tau_{\mu}(\Gamma_{0}) \right)
=
C^{(S \times \p^1, S_{\infty}), \star}_{\GW, (\beta,d)}\left( \lambda \middle| \overline{ \tau_{\mu}(\Gamma_{0}) } \right)
\]
under the variable change $p=e^{z}$, under the following conditions:
\begin{enumerate}
\item[(i)] $\beta \in H_2(S,\BZ)$ is effective and primitive and $\star = \red$, or
\item[(ii)] $\beta = 0$, and $\star = \vir$.
\end{enumerate}
\end{thm}

In this section we prove the following reduction:
\begin{prop} \label{prop:reduction to capped}
Theorem~\ref{thm: capped GWPT} implies Theorem~\ref{thm:GWPT K3 general} for $(S \times \p^1, S_{\infty})$.
\end{prop}
\begin{proof}
The argument in the non-reduced case is precisely given in \cite[Sec.2.4]{PaPix_GWPT}, see also Sec.2.5 of \emph{loc.cit} for a discussion.
The argument in the reduced case works like-wise: One uses the '$\full$ partition functions' introduced in Section~\ref{sec:K3xC PT theory},
which have the same formal properties as the non-reduced invariants, and then follows Section 2.4 of \cite{PaPix_GWPT}.
Since we can hardly achieve a better presentation here than given there, we simply refer to \cite[Sec.2.4]{PaPix_GWPT}
and note that one needs the following slightly modified cases in Section 2.4.6:
\begin{itemize}
\item[] Case I: $|\hat{\alpha}| - 2 \ell(\hat{\alpha}) + \deg(\hat{\Gamma}) \geq d_{\beta} - \theta(\nu) - \theta(\mu) + 1$
\item[] Case II: $|\hat{\alpha}| - 2 \ell(\hat{\alpha}) + \deg(\hat{\Gamma}) < d_{\beta} - \theta(\nu) - \theta(\mu) + 1$.
\end{itemize}
\end{proof}

\subsection{Step 3: Reduction to non-vanishing cohomology}
By deformation invariants of GW and PT invariants we can assume that
$S$ is an elliptic K3 surface with section, and that
\[ \beta = B + hF \]
where $B,F$ is the section and fiber class respectively.
Our goal in this step is to reduce Theorem~\ref{thm: capped GWPT} to the case
where all cohomological insertions on the K3 side are given by $1, \pt, F, B$ and $\Delta_{S}$.
This will be useful in using the degeneration formula in Step 4.


\begin{prop} \label{prop:reduction to non-vanishing}
Let $\beta \in H_2(S,\BZ)$ be any curve class and $\star \in \{ \vir, \red \}$.
Let
\begin{gather*}
\lambda = \Fq_{\lambda_1} \ldots \Fq_{\lambda_{\ell}} (\delta) \in H^{\ast}(S^{[d]}), \quad \delta \in H^{\ast}(S^{\ell(\lambda)}) \\
\tau_{\mu}(\Gamma_{0}) = \left( \tau_{\mu_1-1} \ldots \tau_{\mu_{\ell}-1} \right)\left( \gamma \cdot \pi_1^{\ast}([S_0]) \cdots \pi_{\ell}^{\ast}([S_0]) \right),
\quad \gamma \in H^{\ast}(S^{\ell}).
\end{gather*}
The set of all invariants of the form
\begin{equation} \label{Iinvariants}
C^{(S \times \p^1,S_{\infty}), \star}_{\PT, (\beta,d)} \left( \lambda \middle| \tau_{\mu}(\Gamma_{0}) \right)
\end{equation}
can be effectively reconstructed from the subset of invariants \eqref{Iinvariants} where
\begin{equation} \label{type}
\begin{gathered}
\delta = \Delta_{S,12} \Delta_{S,34} \ldots \Delta_{S, 2a-1,2a} \prod_{i=2a+1}^{\ell(\lambda)} \pi_{i}^{\ast}(\delta_i) \\
\Gamma = \Delta_{S,12} \ldots \Delta_{S, 2 b-1, 2b} \prod_{i=2b+1}^{\ell} \pi_i^{\ast}(\gamma_i )
\end{gathered}
\end{equation}
for any $a,b$ and $\delta_i, \gamma_i \in \{ 1, F, B, \pt \}$ together with their permutations.

Similarly, the set of all invariants
\[ C^{(S \times \p^1, S_{\infty}), \star}_{\GW, (\beta,d)}\left( \lambda \middle| \overline{ \tau_{\mu}(\Gamma_{0}) } \right) \]
can be effectively reconstructed from the subset of $\lambda,\gamma$
of the above form, and moreover, the reconstruction algorithm is the same as for PT invariants.
In particular, the Gromov-Witten/Pandharipande-Thomas correspondence is compatible with this reconstruction,
that is if Theorem~\ref{thm: capped GWPT} holds
for all invariants for $\lambda, \Gamma$ of the form \eqref{type}, then it holds for all $\lambda, \Gamma$.
\end{prop}

As preparation we have to recall basic statements about the monodromy of K3 surfaces.
Assume that $\beta \in H_2(S,\BZ)$ is a effective curve class (the case $\beta = 0$ is parallel).

Let $\Mon(S) \subset O(H^2(S,\BZ))$ be the subgroup generated by all monodromy operators.
By the global Torelli theorem for K3 surface (see \cite{HuybrechtsK3} for an introduction) we have
\[ \Mon(S) \cong \widetilde{O}(H^2(S,\BZ)) \]
where $\widetilde{O}(H^2(S,\BZ))$ is the subgroup of orientation-preserving\footnote{Let
$\CC = \{ x \in H^2(S,\BR) | \langle x, x \rangle > 0 \}$ be the positive cone. Then $\CC$ is homotopy equivalent to the $2$-sphere $S^2$. An automorphism is orientation preserving
if it acts by $+1$ on $H^2(\CC) = \BZ$.} lattice isomorphisms of $H^2(S,\BZ)$.
We are interested here in the stabilizer of $\beta$ in the monodromy group 
\[ \Mon_{\beta}(S) = \widetilde{O}(H^2(S,\BZ))_{\beta}. \]
Using the Torelli theorem again, it is generated by monodromies for which the class $\beta$ stays of Hodge type on all fibers.
By the deformation invariants of reduced Pandharipande-Thomas invariants for deformations of $(S,\beta)$ for which $\beta$ stays algebraic on all fibers,
we have that
\begin{equation} \label{invariance}
\forall \varphi \in \Mon_{\beta}(S):\quad 
C^{(S \times \p^1,S_{\infty}), \red}_{\PT, (\beta,d)} \left( \lambda \middle| \tau_{\mu}(\Gamma_{0}) \right)
=
C^{(S \times \p^1,S_{\infty}), \red}_{\PT, (\beta,d)} \left( \varphi(\lambda) \middle| \tau_{\mu}(\varphi(\Gamma)_{0}) \right),
%
\end{equation}
where $\varphi$ acts factorwise on $\lambda$ and $\Gamma$, that is
\begin{align*}
\varphi(\lambda) & = \Fq_{\lambda_1} \ldots \Fq_{\lambda_{\ell}} ( \varphi \otimes \ldots \otimes \varphi(\delta) ) \\
\varphi(\Gamma) & = \varphi \otimes \ldots \otimes \varphi(\Gamma)
\end{align*}
(The first line matches the induced action of $\varphi$ on $H^{\ast}(S^{[d]})$.).
%
Since $\Mon_{\beta}(S)$ is an arithmetic subgroup of $O(H^2(S,\BC))_{\beta} = O(\beta^{\perp} \otimes \BC)$ it is Zariski dense in $O(\beta^{\perp})$.
We find that \eqref{invariance} holds for all $\varphi \in O(\beta^{\perp} \otimes \BC)$.

Similarly, since $c(T_{S})$ is monodromy invariant, we have that:
\[
\forall \varphi \in O(\beta^{\perp} \otimes \BC): \quad 
 C^{(S \times \p^1, S_{\infty}), \red}_{\GW, (\beta,d)}\left( \lambda \middle| \overline{ \tau_{\mu}(\Gamma_{0}) } \right) 
= C^{(S \times \p^1, S_{\infty}), \red}_{\GW, (\beta,d)}\left( \varphi(\lambda) \middle| \overline{ \tau_{\mu}(\varphi(\Gamma)_{0}) } \right) \]

\begin{proof}
The proof follows again closely the ideas of \cite[Section 4]{ABPZ},
and we refer to \emph{loc.cit.} for further details and references on the representation theory that is used below.

Consider the ortogonal complement,
\[ V = \mathrm{Span}(B,F)^{\perp} \subset H^2(S,\BC). \]
Let $\{ e_i \}_{i=1}^{20}$ be a basis of $V$ and consider the basis
\[ \CB = \{ 1, \pt, B, F \} \cup \{ e_i \}_{i=1}^{20}. \]
We weigh our cohomology partitions by elements of $\CB$.
Consider a general invariant
\begin{equation}
C^{(S \times \p^1,S_{\infty}), \star}_{\PT, (\beta,d)} \left( \lambda \middle| \tau_{\mu}(\Gamma_{0}) \right) \label{dgsdg} \end{equation}
where $\lambda, \mu$ are of the form \eqref{type} but with
$\delta_i$ and $\gamma_i$ arbitrary elements of the basis $\CB$.
%
Assume that the first $f_1$ of the elements $\delta_i$ lie in $V$, and the remainder are in $\{ 1, \pt, B, F\}$,
and similarly, that the first $f_2$ factors of $\gamma_i$ are taken from $V$, and the remainder not.
We then consider the invariant \eqref{dgsdg} as a function
\[ C: V^{\otimes (f_1+f_2)} \to \BQ(t)((p)) \]
given by
\[ v_1 \otimes \ldots \otimes v_{f_1+f_2} \mapsto 
C^{(S \times \p^1,S_{\infty}), \star}_{\PT, (\beta,d)} \left( \lambda \middle| \tau_{\mu}(\Gamma_{0}) \right)\Big|_{\substack{ \delta_{2a+i} = v_i \text{ for } i=1,\ldots,f_1 \\ \gamma_{2b+i} = v_{i} \text{ for } i=1,\ldots, f_2}.}
\]

By the monodromy invariance \eqref{invariance} (in case $\beta=0$ the invariant \eqref{dgsdg} is invariant under $O(H^2(S,\BC))$) we find that
\[ C \in ((V^{\ast})^{\otimes (f_1+f_2)})^{O(V)}. \]
Hence by standard invariant theory of the orthogonal group,
$C$ lies in the subring of 
\[ T(V) = \bigoplus_{n \geq 0} (V^{\ast})^{\otimes n} \]
generated by the pullbacks
\[ Q_{ij} = (\pr_i^{\ast} \otimes \pr_j^{\ast})( Q ) \]
by the projection $\pr_i \otimes \pr_j : (V^{\ast})^{\otimes n} \to V^{\vee} \otimes V^{\vee}$ of the class
of the inner pairing:
\[ Q \in V^{\ast} \otimes V^{\ast}, \quad Q(v,w) = \langle v, w \rangle. \]
In particular $2k=f_1+f_2$ is even.
Moreover, a basic representation-theoretic fact (see \cite{ABPZ}) is that the class
$C$ is then determined by its intersections agains all big-diagonal classes:
\[ \sigma( \Delta_V \otimes \Delta_V \otimes \cdots \otimes \Delta_V) \in V^{\otimes 2k}. \]
where $\sigma$ runs over all permutations of $2k$, and
$\Delta_V = Q^{\vee} \in V \otimes V$ is the dual of $Q$.

We have
\[ \Delta_S = \Delta_V + \pt \otimes 1 + 1 \otimes \pt + (B + F) \otimes F + F \otimes (B+F) \]
Hence by the induction hypothesis, the invariant \eqref{dgsdg}
is determined when
\[ \delta_{2a+1} \otimes \ldots \otimes \delta_{2a+f_1} \otimes \gamma_{2b+1} \otimes \ldots \otimes \gamma_{2b+f_2} \]
run over the classes
\[ \sigma( \Delta_S \otimes \Delta_S \otimes \cdots \otimes \Delta_S) \in (H^{\ast}(S))^{\otimes 2k}. \]

This completes the first part.
The relations we obtain from this process are identical for the Gromov-Witten invariants,
hence the algorithm is compatible with the GW/PT correspondence.
This proves the second part.
\end{proof}

\subsection{Step 4: Degeneration formula}
Let $S$ be the elliptic K3 surface with section, and consider the degeneration
\begin{equation} S \rightsquigarrow R_1 \cup_E R_2 \label{Sdeg} \end{equation}
of $S$ to the union of two rational elliptic surfaces $R_i \to \p^1$ glued a long a smooth elliptic fiber $E \subset R_i$.
The degeneration can be choosen such that the total space of the degeneration $W \to \Delta$
admits an elliptic fibration with section which restricts to the given section on $S$:
\[ \pi : W \to B \times \Delta, \quad j : B \times \Delta \to W, \quad \pi \circ j = \id. \]
In particular, there exists cohomology classes
\[ \widetilde{F}, \widetilde{B}, \widetilde{\pt}, \widetilde{1} \in H^{\ast}(W) \]
such that
\[ \widetilde{F}|_{S} = F, \quad \widetilde{B}|_{S} = B, \quad \widetilde{\pt}|_{S} = \pt, \quad \widetilde{1}|_{S} = 1 \in H^{\ast}(S).. \]

The degeneration \eqref{Sdeg} is famously used in \cite{MPT} to prove the GW/PT correspondence for $S \times \BC$.
Here we will use it also.
We also refer to \cite{Greer} for a study of the Hodge structure of the degeneration
(it is a standard type $2$ Kulikov degeneration),
and to \cite[Sec.4.3.2]{Baltes} for an explicit presentation of the degeneration in terms of equations.
One can show that the image of
\[ H^2(W, \BC) \to H^2(S,\BC) \]
is a sublattice of rank $20$ (so has codimension $2$),
hence there exists vanishing cohomology in $H^2(S,\BC)$ for the degeneration.

The degeneration \eqref{Sdeg} induces a degeneration
\[
S \times \p^1 \rightsquigarrow (R_1 \times \p^1) \cup_{E \times \p^1} (R_2 \times \p^1).
\]
We obtain the degeneration of relative theories:
\begin{equation} \label{3sdf}
(S \times \p^1, S_{\infty}) \rightsquigarrow \mathbf{Y}_1 \cup_{(E \times \p^1,E \times \infty)} \mathbf{Y}_2
\end{equation}
where
\[ \mathbf{Y}_i = (R_i \times \p^1, R_{i, \infty} \cup E \times \p^1)
\]

For $i \in \{ 1, 2 \}$, let $B_i = \widetilde{B}|_{R_i}$ and let $F_i$ be the fiber class on $R_i \to \p^1$
We use the curve classes
\[ \beta_h = B+hF \in H_2(S,\BZ), \quad \beta^{R_i}_{h} = B_i + h F_i \quad \in H_2(R_i,\BZ). \]
When clear from notation we drop the supscript $R_i$.

We want to apply the degeneration formula for the degeneration \eqref{3sdf} to the invariants:
\begin{equation}
C^{(S \times \p^1,S_{\infty}), \red}_{\PT, (\beta_h,d)} \left( \lambda \middle| \tau_{\mu}(\Gamma_{0}) \right),
\quad
C^{(S \times \p^1, S_{\infty}), \red}_{\GW, (\beta_h,d)}\left( \lambda \middle| \overline{ \tau_{\mu}(\Gamma_{0}) } \right)
\end{equation}
where
\begin{gather*}
\lambda = \Fq_{\lambda_1} \ldots \Fq_{\lambda_{\ell}} (\delta) \vacuum \in H^{\ast}(S^{[d]}), \quad \delta \in H^{\ast}(S^{\ell(\lambda)}) \\
\tau_{\mu}(\Gamma_{0}) = \left( \tau_{\mu_1-1} \ldots \tau_{\mu_{\ell}-1} \right)\left( \gamma \cdot \pi_1^{\ast}([S_0]) \cdots \pi_{\ell}^{\ast}([S_0]) \right),
\quad \gamma \in H^{\ast}(S^{\ell}).
\end{gather*}
with
\begin{equation}
\begin{gathered}
\delta = \Delta_{S,12} \Delta_{S,34} \ldots \Delta_{S, 2a-1,2a} \prod_{i=2a+1}^{\ell(\lambda)} \pi_{i}^{\ast}(\delta_i) \\
\Gamma = \Delta_{S,12} \ldots \Delta_{S, 2 b-1, 2b} \prod_{i=2b+1}^{\ell} \pi_i^{\ast}(\gamma_i )
\end{gathered}
\end{equation}
for any $a,b$ and $\delta_i, \gamma_i \in \{ 1, F, B, \pt \}$ together with their permutations.

We consider the lifts:
\begin{gather*}
\widetilde{\lambda} = \Nak_{\vec{\lambda}}( \widetilde{\delta} ) \in H^{\ast}( (\CW/\Delta)^{[\ell(\lambda)]} )  \\
\widetilde{\Gamma} =  \Delta^{\rel}_{(\CW/\Delta),12} \ldots \Delta_{(\CW/\Delta), 2 b-1, 2b}^{\rel} \prod_{i=2b+1}^{\ell} \pi_i^{\ast}(\widetilde{\gamma}_i )
\end{gather*}
where $\vec{\lambda} = (\lambda_1, \ldots, \lambda_{\ell})$ is the partition underlying $\lambda$, and
\[
\widetilde{\delta} = \Delta^{\rel}_{(\CW/\Delta),12} \ldots \Delta^{\rel}_{(\CW/\Delta),2a-1,2a} 
\prod_{i=2a+1}^{\ell(\lambda)} \pi_{i}^{\ast}(\widetilde{\delta}_i)
\in H^{\ast}( (\CW/\Delta)^{\ell(\lambda)}.
\]

Given a splitting $\{ 1, \ldots, \ell \} = B_1 \sqcup B_2$ recall the gluing morphism
\[ \xi_{A_1} : (R_1,E)^{B_1} \times (R_2, E)^{B_2} \to (\CW/\Delta)^{\ell}. \]
We then have
\[ 
\xi_{B_1}^{\ast}(\widetilde{\Gamma}) = 
\begin{cases}
\Gamma_{B_1} \otimes \Gamma_{B_2} & \text{ if } \forall i \in \{1, \ldots, b\} : 2i-1, 2i \in B_1 \text{ or } 2i-1, 2i \in B_2 \\
0 & \text{ else }.
\end{cases}
\]
where for $t \in \{1,2 \}$ we have:
\[
\Gamma_{B_t} = \prod_{\substack{ 1 \leq c \leq b \\ 2c-1 \in A_j}} \Delta^{\rel}_{(R_t,E),2c-1, 2c} \prod_{\substack{2b+1 \leq j \leq \ell \\ j \in B_t }} \pi_{j}^{\ast}( \widetilde{\gamma}_j|_{R_t} ).
\]

Similarly, for any splitting $\ell(\lambda) = k_1 + k_2$ recall the gluing
\[
\xi_{k_1} : (R_1,E)^{[k_1]} \times (R_2, E)^{[k_2]} \to (\CW/\Delta)^{[ \ell(\lambda) ]}
\]
By Lemma~\ref{lemma:rel hilb gluing} we then have:
\[
\xi_{k_1}^{\ast}(\widetilde{\lambda})
=
\sum_{A_1, A_2}
\Nak_{\vec{\lambda}_{A_1}}( \delta_{A_1} ) \otimes \Nak_{\vec{\lambda}_{A_2}}( \delta_{A_2} ),
\]
where the sum runs over all splittings 
$\{1 ,\ldots , \ell(\lambda)\} = A_1 \sqcup A_2$ such that 
\begin{itemize}[itemsep=0pt]
\item $|\vec{\lambda}_{A_t}| = k_t$ for $t=1,2$, and
\item for all $i \in \{ 1, \ldots, a \}$ we have $2i-1,2i \in A_{1}$ or $2i-1,2i \in A_2$.
\end{itemize}
Moreover, $\delta_{A_t}$ is defined similar to $\Gamma_{B_t}$, that is
\[
\delta_{A_t} = \prod_{\substack{ 1 \leq c \leq a \\ 2c-1 \in A_j}} \Delta^{\rel}_{(R_t,E),2c-1, 2c} \prod_{\substack{2b+1 \leq j \leq \ell \\ j \in A_t }} \pi_{j}^{\ast}( \widetilde{\delta}_j|_{R_t} ).
\]
We write
\[ \lambda_{A_t} := \Nak_{\vec{\lambda}_{A_t}}( \delta_{A_t} ), \quad t=1,2. \]

%
We now apply the degeneration formula of \cite{Li1, Li2, LiWu, MPT},
adapted to the setting of capped descendents as discussed in \cite{PaPix_GWPT} for the invariant
\[ C^{(S \times \p^1,S_{\infty}), \red}_{\PT, (\beta_h,d)} \left( \lambda \middle| \tau_{\mu}(\Gamma_{0}) \right). \]
As shown in \cite{MPT} the reduced virtual class of the moduli space splits here into the product of
two usual virtual classes associated to the factors $R_1, R_2$,
but with a modified diagonal splitting.
Concretely, on the elliptic curve factor one replaces the diagonal $\Delta_E \in H^{\ast}(E \times E)$
by the insertion $1_E \otimes 1_E \in H^{\ast}(E \times E)$.\footnote{
Any curve on $R_i$ in class $\beta_h$ is a comb curve (i.e. as a divisor of the form $B_i+F_1 + \ldots + F_h$ for some fibers $F_i$ of $R_i \to \p^1$, see e.g. \cite{BL}).
That means for curves $C_i \subset R_i$ in class $\beta_{h_i}$ the condition that $C_1 \cap E = C_2 \cap E$
is automatically satisfied,
or in other words that the Gysin pullback by $\Delta_E^{!}$ has a $1$-dimensional trivial excess contribution.
Maulik, Pandhariande, and Thomas in \cite{MPT} show then that this excess contribution precisely matches 
the trivial piece in the obstruction sheaf on the K3 side,
hence that when working with the reduced virtual class the Gysin pullback $\Delta_E^{!}$ is \emph{not} applied.}
Hence on the product $(E \times \p^1, E_{\infty})$ we replace the diagonal term $\Delta_{(E \times \p^1,E_{\infty})}^{\rel} = \Delta_E \cdot \Delta_{(\p^1,\infty)}^{\rel}$
by the class $\Delta_{(\p^1,\infty)}^{\rel}$,
where we suppressed the pulled back by the projections from $(E \times \p^1,E_{\infty})^2$ to $E^2$ and $(\p^1, \infty)^2$ respectively.
The curve components
which meet the relative divisor $(E \times \p^1, E_{\infty})$ 
have non-trivial degree over $R_1$ or $R_2$,
hence for the capped invariants they must meet the relative divisor in the fixed point $\mathbf{0} = (0_E, 0_{\p^1}) \in E \times \p^1$.
We hence find the diagonal splitting term contributes
\[ \Delta_{(\p^1,\infty)}^{\rel}|_{\mathbf{0}, \mathbf{0}} = t. \]
In total, the result is the following: 
\begin{multline*}
C^{(S \times \p^1,S_{\infty}), \red}_{\PT, (\beta_h,d)} \left( \lambda \middle| \tau_{\mu}(\Gamma_{0}) \right)
= \\
t \sum_{h = h_1 + h_2} 
\sum_{\substack{ \{1 , \ldots , \ell(\lambda) \} = A_1 \sqcup A_2  \\ \{ 1 , \ldots, \ell(\mu) \} = B_1 \sqcup B_2 }}
C^{\mathbf{Y}_1}_{(\beta_{h_1},d)}\Big( \lambda_{A_1}\, \Big| \, 1 \, \Big| \, \tau_{\mu_{B_1}}\big( (\Gamma_{B_1})_{0} \big) \, \Big)
C^{\mathbf{Y}_2}_{(\beta_{h_2},d)}\Big( \lambda_{A_2}\, \Big| \, 1 \, \Big| \, \tau_{\mu_{B_2}}\big( (\Gamma_{B_2})_{0} \big) \, \Big).
\end{multline*}

The discussion on the Gromov-Witten side is completely parallel.
As we have seen in 
Proposition~\ref{prop:Compatibility GWPT with degeneration formula}
the $\overline{ ( \ - \ ) }$ operation is compatible with the degeneration formula.
Hence (with the obvious notation for $\lambda_{A_1}, \lambda_{A_2}$) one obtains:
\begin{multline*}
C^{(S \times \p^1, S_{\infty}), \red}_{\GW, (\beta_h,d)}\left( \lambda \middle| \overline{ \tau_{\mu}(\Gamma_{0}) } \right)
=
\\
t \sum_{h = h_1 + h_2} 
\sum_{\substack{ \{1 , \ldots , \ell(\lambda) \} = A_1 \sqcup A_2  \\ \{ 1 , \ldots, \ell(\mu) \} = B_1 \sqcup B_2 }}
C^{\mathbf{Y}_1}_{(\beta_{h_1},d)}\Big( \lambda_{A_1}\, \Big| \, 1 \, \Big| \, \overline{\tau_{\mu_{B_1}}\big( (\Gamma_{B_1})_{0} \big)} \, \Big)
C^{\mathbf{Y}_2}_{(\beta_{h_2},d)}\Big( \lambda_{A_2}\, \Big| \, 1 \, \Big| \, \overline{\tau_{\mu_{B_2}}\big( (\Gamma_{B_2})_{0} \big)} \, \Big).
\end{multline*}

\subsection{Proof of Theorem~\ref{thm: capped GWPT} and conclusion} \label{sec:conclusion}
The above steps established Theorem~\ref{thm:GWPT K3 general} as follows:
By Proposition~\ref{prop:reduction to SxP1} 
and Proposition~\ref{prop:reduction to capped} (Step 1 and 2)
we reduced the proof of the theorem to the GW/PT correspondence for the capped invariants stated in Theorem~\ref{thm: capped GWPT}.
To prove Theorem~\ref{thm: capped GWPT},
by Proposition~\ref{prop:reduction to non-vanishing}
it suffices to only consider capped descendents with non-vanishing cohomology.
By the degeneration formula in Step 4 finally,
the GW/PT correspondence for these invariants are
reduced to the GW/PT correspondence for the birelative theory
$(R \times \p^1, R_{\infty} \cup (E \times \p^1)$.
This correspondence was established in Theorem~\ref{thm:GWPT for birelative theory}. 

This finishes the proof of  Theorem~\ref{thm: capped GWPT} and Theorem~\ref{thm:GWPT K3 general}.
\qed

\section{$K3 \times \mathrm{Curve}$: Imprimitive case} \label{sec:K3 Curve imprimitive case}
Let $S$ be a K3 surface. We consider here again the GW/PT correspondence of the pair
\begin{equation} (S \times C, S_z), \quad S_{z} = \bigsqcup S \times \{ z_i \}. \label{relative geometryXX} \end{equation}
but for curve classes $(\beta,d)$ where $\beta$ is allowed to have arbitrary divisibility.

\subsection{Rationality}
By \cite[Thm.5.1]{QuasiK3}
the Pandharipande-Thomas invariants of \eqref{relative geometryXX}
satisfy the following multiple cover formula:
\begin{equation} \label{eqn:mcf pt}
Z^{(S \times C,S_{z}),\red}_{\PT, (\beta,d)}\left( \lambda_1, \ldots, \lambda_N \right)
=
\sum_{k|r} Z^{(S \times C,S_{z}),\red}_{\PT, (\varphi_k(\beta/k),d)}\left( \varphi_k(\lambda_1), \ldots, \varphi_k(\lambda_N) \right)(p^k).
\end{equation}

By the first part of Theorem~\ref{thm:GWPT K3 general} we immediately obtain:

\begin{lemma} \label{cor:PT rationality K3}
For any curve class $\beta \in H_2(S,\BZ)$ the series
\[
Z^{(S \times C,S_{z}),\red}_{\PT, (\beta,n)}\left( \lambda_1, \ldots, \lambda_N \right)
\]
is the Laurent expansion of a rational function in $p$.
\end{lemma}

\subsection{Multiple cover formula for K3 surfaces} \label{subsec:MCF K3}
Consider the reduced Gromov-Witten invariants of a K3 surface:
\[
\big\langle \, \taut ; \gamma_1, \ldots, \gamma_r \big\rangle^{S,\GW,\red}_{g,\beta}
=
\int_{ [ \Mbar_{g,r,\beta}(S) ]^{\red} }
\tau^{\ast}(\taut) \prod_{i=1}^{r} \ev_i^{\ast}(\gamma_i).
\]
where we have twisted by a {\em tautological class} $\taut \in R^{\ast}(\Mbar_{g,r})$
pulled back by the forgetful morphism $\tau : \Mbar_{g,r}(S,\beta) \to \Mbar_{g,r}$,
and with the convention that in the unstable case $2g-2+r \leq 0$ we define $\tau$ to be the projection to the point and $R^{\ast}(pt) = H^{\ast}(pt)$.

For every divisor $k|\beta$ let $S_k$ be a K3 surface and let
\[ \varphi_k : H^{2}(S,\BC) \to H^{2}(S_k, \BC) \]
be a {\em complex} isometry such that $\varphi_k(\beta/k) \in H_2(S_k,\BZ)$
is a primitive effective curve class.
We extend $\varphi_k$ to the full cohomology lattice
by $\varphi_k(\pt) = \pt$ and $\varphi_k(1) = 1$.

The following was conjectured in \cite[Conj.C2]{K3xE}:
\begin{conj}[Multiple Cover Formula] \label{conj:MCF}
\[
\big\langle \, \tau_{k_1}(\gamma_1) \cdots \tau_{k_r}(\gamma_r) \big\rangle^{S,\GW,\red}_{g,\beta}
= \sum_{k | \beta}  
k^{2g-3+\sum_i \deg(\gamma_i)}
\big\langle \, \tau_{k_1}(\varphi_k(\gamma_1)) \cdots \tau_{k_r}(\varphi_k(\gamma_r)) \big\rangle^{S,\GW,\red}_{g, \varphi_r\left( \frac{\beta}{k} \right)}
\]
\end{conj}
\begin{rmk} \label{rmk:K3 MCF}
The (ordinary) virtual class of the moduli space of stable maps to a K3 surface
satisfies
\[
[ \Mbar_{g,r}(S,0) ]^{\vir}
=
\begin{cases}
[ \Mbar_{0,r} \times S] & \text{ if } g=0 \\
c_2(S) \cap [ \Mbar_{1,r} \times S] & \text{ if } g=1 \\
0 & \text{ if } g \geq 2.
\end{cases}
\]
It follows that the Multiple Cover Formula of Conjecture~\ref{conj:MCF}
is equivalent to the same statement but where we worked with the disconnected Gromov-Witten invariants
$\langle \cdots \rangle^{S,\GW,\bullet}$.
\end{rmk}

\subsection{GW/PT correspondence}
We prove here the following:
\begin{prop} \label{prop:mcf implies mcf for SxC}
Assume the multiple cover formula of Conjecture~\ref{conj:MCF} holds for
an effective curve class $\beta \in H_2(S,\BZ)$.
Then we have
\begin{gather*} 
Z^{(S \times C,S_{z}),\red}_{\GW, (\beta,d)}\left( \lambda_1, \ldots, \lambda_N \right)(z)
=
\sum_{k|\beta} Z^{(S \times C,S_{z}),\red}_{\GW, (\varphi_k(\beta),d)}\left( \varphi_k(\lambda_1), \ldots, \varphi_k(\lambda_N) \right)(kz)
\end{gather*}
for any $H^{\ast}(S)$-weighted partitions $\lambda_i = (\lambda_{i,j}, \delta_{i,j})_{j=1}^{\ell(\lambda_i)}$,
where $\varphi_r(\lambda_i) = (\lambda_{i,j}, \varphi_r(\delta_{i,j}))$.
\end{prop}
\begin{proof}
Recall the partition function
\begin{multline*}
Z^{(S \times C,S_{z}),\red}_{\GW, (\beta,d)}\left( \lambda_1, \ldots, \lambda_N \right)
=
(-1)^{d(1-g(C)) + \sum_i (\ell(\lambda_i) - |\lambda_i|)}
z^{2d(1-g(C) + \sum_i ( \ell(\lambda_i) - |\lambda_i|)} \\
\cdot \sum_{g \in \BZ} (-1)^{g-1} z^{2g-2}
\left\langle \, \lambda_1, \ldots, \lambda_N \, \right\rangle^{(S \times C,S_z), \bullet, \red}_{g, (\beta,d)}
\end{multline*}
We derive the multiple cover formula for the invariant
$\left\langle \, \lambda_1, \ldots, \lambda_N \, \right\rangle^{(S \times C,S_z), \bullet, \red}_{g, (\beta,d)}$.
By an argument parallel to Remark~\ref{rmk:K3 MCF} it suffices to consider connected invariants.
Then by the product formula of \cite{LQ} we find that
\[
\left\langle \, \lambda_1, \ldots, \lambda_N \, \right\rangle^{(S \times C,S_z), \red}_{g, (\beta,d)}
=
\left\langle \taut ; ( \delta_{ij}  )_{i,j} \right\rangle^{S,\GW,\red}_{g,\beta}
\]
for some tautological class $\taut \in R^{\ast}(\Mbar_{g,\sum_i \ell(\lambda_i)})$.
Assuming that Conjecture~\ref{conj:MCF} holds for the class $\beta$ this becomes:
\begin{align*}
\left\langle \taut ; ( \delta_{ij}  )_{i,j} \right\rangle^{S,\GW,\red}_{\beta}
& = \sum_{k|r} k^{2g-3 + \sum_{i,j} \deg(\delta_{i,j})} \left\langle \taut ; ( \varphi_r(\delta_{ij})  )_{i,j} \right\rangle^{S,\GW,\red}_{g,\varphi_k(\beta/k)} \\
& = \sum_{k|r} k^{2g-3 + \sum_{i,j} \deg(\delta_{i,j})} 
\left\langle \, \varphi_k(\lambda_1), \ldots, \varphi_k(\lambda_N) \, \right\rangle^{(S \times C,S_z), \red}_{g, (\varphi_k(\beta/k),d)}.
\end{align*}
The claim now follows from the dimension constraint (the invariants are zero if it is violated, so we may assume it):
\[ 2d (1-g(C)) + \sum_i \ell(\lambda_i) - |\lambda_i| + 1 = \sum_{i,j} \deg(\delta_{ij}). \]
\end{proof}

\begin{proof}[Proof of Proposition~\ref{prop:MCF implies GWPT}]
The first part is Lemma \ref{cor:PT rationality K3}.
The second claim follows
from Proposition~\ref{prop:mcf implies mcf for SxC},
equation \eqref{eqn:mcf pt}
and the primitive case of the GW/PT correspondence given by Theorem~\ref{thm:GWPT K3 general}.
\end{proof}

\appendix
\section{Higher descendents invariants} \label{sec:higher descendent}
The paper \cite{ABPZ} introduced nodal Gromov-Witten invariants
which count stable maps with prescribed nodes on the domain curve.
These counts can be computed by resolving the nodes and inserting (relative) diagonals.
We introduce here a corresponding notion for stable pairs
which we call 'higher descendents'.
As for nodal GW invariants, they can be computed through our marked relative invariants
by inserting (relative) diagonals, and at least very na\"ively, they carry information
on stable pairs whose underlying curve has nodes.

\subsection{Definition} \label{subsubsec:higher descendent invariants}
Let $(\BF,s)$ be the universal stable pair on $P_{\Gamma}(X,D) \times_{\CT} \CX$ and consider the diagram
\[
\begin{tikzcd}
& P_{\Gamma}(X,D) \times_{\CT} \CX \ar{dr}{\pi_X} \ar{dl}[swap]{\pi} &  \\
P_{\Gamma}(X,D) & & X
\end{tikzcd}
\]
where $\pi_X$ is the projection to $\CX$ followed by the universal contraction morphism $\CX \to X$.

Given a sequence $\mathbf{a} = (a_1, \ldots, a_r) \in \BZ^r$ and a class $\gamma \in H^{\ast}(X)$ we define the higher descendents
\[ \tau_{\mathbf{a}}(\gamma) = \pi_{\ast}( \ch_{2+a_1}(\BF ) \cdots \ch_{2+a_r}(\BF ) \cup \pi_X^{\ast}(\gamma) ) \in H^{\ast}(P_{\Gamma}(X,D)). \]

Let $\lambda \in H^{\ast}(D^{[\ell]})$ be a cohomology class, and for $i \in \{1 , \ldots, s \}$ consider
tuples $\mathbf{a}_i = (a_{ij}) \in \BZ^{\ell(\mathbf{a}_i)}$,
and classes $\gamma_i \in H^{\ast}(X)$.
\begin{defn} 
The higher descendent Pandharipande-Thomas invariants are defined by
\[
\left\langle \, \lambda \middle| \tau_{\mathbf{a}_1}(\gamma_1) \cdots \tau_{\mathbf{a}_s}(\gamma_s) \right\rangle^{(X,D), \PT, \textup{hi.-desc.}}_{\Gamma} \\
=
\int_{ [ P_{\Gamma}(X,D) ]^{\vir} }
\ev^{\rel \ast}(\lambda) \cup \prod_i \tau_{\mathbf{a}_i}(\gamma_i).
\]
\end{defn}

\subsection{Comparision result}
We have the following comparision result
which essentially says that we can trade higher-descendent insertion
for marked relative insertions (considered in Section~\ref{subsec:relative descendents}) by the rule
\[ 
\tau_{a_1, \ldots, a_r}(\delta) \quad \rightsquigarrow \quad \tau_{a_1} \tau_{a_2} \cdots \tau_{a_r}( \Delta_{r \ast}^{\rel}(\delta) ) \]
where 
\[ \Delta_{r}^{\rel} : X \cong (X,D)^1 \to (X,D)^r \]
is the relative diagonal.

\begin{prop} For every $i \in \{1, \ldots, s\}$ let $\mathbf{a}_i = (a_{ij})_{j=1}^{\ell(\mathbf{a}_i)} \in \BZ^{\ell(\mathbf{a}_i)}$. Then
\[
\left\langle \, \lambda\, \middle| \,\prod_{i} \tau_{\mathbf{a}_i}(\gamma_i) \right\rangle^{(X,D), \PT, \textup{hi.-desc.}}_{\Gamma} \\
=
\left\langle \, \lambda\, \middle|\, \prod_{i} \tau_{a_{i1}} \cdots \tau_{a_{i \ell(\mathbf{a}_i)}}( \Delta^{\rel}_{\ell(\mathbf{a}_i) \ast}(\gamma_i) \right\rangle^{(X,D), \PT, \textup{marked}}_{\Gamma} 
\]
\end{prop}

\begin{proof}
Write $P_r = P_{\Gamma,r}(X,D)$. Consider the relative diagonal map:
\[
\Delta^{\rel}_{\ell(\mathbf{a}_1) \cdots \ell(\mathbf{a}_s)} : (X,D)^s \to (X,D)^{r}
\]
given by sending $(p_1, \ldots, p_s) \in X[\ell]^s$ to
\[ (\underbrace{p_1, \ldots, p_1}_{\ell(\mathbf{a}_1) \text{ times }}, \underbrace{p_2, \ldots, p_2}_{\ell(\mathbf{a}_2) \text{ times }}, \ldots, \underbrace{p_s, \ldots, p_s}_{\ell(\mathbf{a}_1) \text{ times }}) \in X[\ell]^r \]
where $r = \sum_i \ell(\mathbf{a}_i)$.
We have
\[ \Delta^{\rel}_{\ell(\mathbf{a}_1) \cdots \ell(\mathbf{a}_s) \ast}(\gamma_1 \otimes \ldots \otimes \gamma_r) = 
\prod_{i} \Delta^{\rel}_{\ell(\mathbf{a}_i) \ast}(\gamma_i).
\]
Moreover we have the fiber diagram
\[
\begin{tikzcd}
\CX_s \ar{d}{\pi_s} \ar{rr}{\dbtilde{\Delta}} && \CX_r \ar{d}{\pi_r} \\
P_{s} \ar{rr}{\widetilde{\Delta}} \ar{d}{\ev} & & P_r \ar{d}{\ev} \\
(X,D)^s \ar{rr}{\Delta^{\rel}_{\ell(\mathbf{a}_1) \cdots \ell(\mathbf{a}_s)}} & & (X,D)^{r}
\end{tikzcd} 
\]

Let $(\BF_r, \sigma_r)$ denote the universal stable pair on $\CX_r \to P_r$.
We have $\dbtilde{\Delta}^{\ast}(\BF_r) = \BF_s$ because both are pulled back from $\CX_0 \to P_0$ and the universal targets agree.
This shows that for every $b \in \{ 1, \ldots, r \}$ we have that
\[ \widetilde{\Delta}^{\ast} p_b^{\ast}( \ch_k(\BF_r) ) = p_{i(b)}^{\ast} \ch_k( \dbtilde{\Delta}^{\ast}(\BF_r) ) = p_{i(b)}^{\ast} \ch_k(\BF_s) \]
where $i(b) \in \{ 1, \ldots, s \}$ is the index such that $\mathrm{pr}_{i(b)} = \mathrm{pr}_b \circ \Delta^{\rel}_{\ell(\mathbf{a}_1) \cdots \ell(\mathbf{a}_s)}$.
We hence get
\begin{multline}
\label{3523524}
\left\langle \, \lambda \middle| \prod_{i} \tau_{a_{i1}} \cdots \tau_{a_{i \ell(\mathbf{a}_i)}}( \Delta^{\rel}_{\ell(\mathbf{a}_i) \ast}(\gamma_i)) \right\rangle^{(X,D), \PT, \textup{marked}}_{\Gamma} \\
=
\int_{ [ P_{\Gamma,s}(X,D) ]^{\vir} }
\ev_{\mathrm{rel} \ast}(\lambda) \cdot
\prod_{i=1}^{s} \left( \prod_{j=1}^{\ell(\mathbf{a}_i)} p_{i}^{\ast}( \ch_{a_{ij}}(\BF_s) ) \right) \cdot \ev^{\ast}( \gamma_1 \otimes \ldots \otimes \gamma_s ).
\end{multline}
Arguing now precisely as in the proof of Proposition~\ref{prop: comparision}
(i.e. using push-pull for the morphism $\tilde{\pi} : P_s \to \CX_0 \times_{P_0} \cdots \times_{P_0} \CX_0$
this becomes
\[
\int_{\CX_0 \times_{P_0} \cdots \times_{P_0} \CX_0}  \ev_{\rel}^{\ast}(\lambda)
\prod_{i=1}^{s} \left( \prod_{j=1}^{\ell(\mathbf{a}_i)} \rho_i^{\ast}\big( \ch_{a_{ij}}(\BF)  \pi_X^{\ast}(\gamma_i) \big)  \right)
\cap \pi^{\ast}[ P_0 ]^{\vir}
\]
which is precisely $\left\langle \, \lambda \middle| \prod_{i} \tau_{\mathbf{a}_i}(\gamma_i) \right\rangle^{(X,D), \PT, \textup{hi.-desc.}}_{\Gamma}$ as desired.
\end{proof}

Mathemathisches Institut, Universit\"at Bonn

georgo@math.uni-bonn.de \\

\end{document}